\documentclass[10pt]{article}
\usepackage{latexsym}
\usepackage[top=1in, bottom=1in, left=1in, right=1in]{geometry}
\usepackage{adjustbox}
\usepackage{amssymb}
\usepackage{graphicx} 
\usepackage{color} 
\usepackage{amsmath, amsthm, amssymb}
\usepackage{enumerate}
\usepackage{float}
\usepackage{url}
\usepackage{stackrel}
\usepackage{mathrsfs,dsfont}
\usepackage[labelfont=bf]{caption}
\usepackage{subcaption}
\usepackage{wrapfig}
\usepackage{bbm}
\usepackage{bm}
\usepackage{epigraph}
\usepackage{physics}
\usepackage[colorinlistoftodos, shadow]{todonotes} 
\usepackage{multirow}
\usepackage{hyperref}
\usepackage{soul}
\usepackage{pdflscape}
\usepackage{tikz}
\usetikzlibrary{shapes,arrows,automata,graphs,decorations.pathmorphing,decorations.markings}
\usepackage{tikz-cd}
\usepackage{fancybox}
\usepackage{booktabs}
\usepackage{colortbl} 
\usepackage[table]{xcolor} 
\usepackage{nicematrix}
\usepackage{algorithm}
\usepackage{algpseudocode}
\definecolor{rowtint}{gray}{0.96}
\definecolor{tablegrey}{gray}{0.95}
\algrenewcommand\algorithmicrequire{\textbf{Input:}}
\algrenewcommand\algorithmicensure{\textbf{Output:}}
\usepackage[title]{appendix}



\usepackage{mdwlist}
\usepackage{pifont}

\newtheorem{theorem}{Theorem}[section]
\newtheorem{corollary}[theorem]{Corollary}

\newtheorem{lemma}[theorem]{Lemma}
\newtheorem{example}[theorem]{Example}

\theoremstyle{definition}

\theoremstyle{definition}
\newtheorem{definition}[theorem]{Definition}
\theoremstyle{definition}

\newcommand{\been}{\begin{enumerate}}
\newcommand{\enen}{\end{enumerate}}
\newcommand{\beit}{\begin{itemize}}
\newcommand{\enit}{\end{itemize}}

\def\cre
\def\cgn{\color{olive}}
\def\cbl
\def\cma
\def\ccy{\color{cyan}}

\usepackage{tikz}    
\usetikzlibrary{shapes,automata,positioning,arrows,fit}
\usepackage{mathtools}  
\usetikzlibrary{snakes}

\usepackage{caption}
\usepackage{tikz-cd}
\usetikzlibrary{babel}

\usepackage{tkz-euclide}

\usepackage{tikz-3dplot}

\usepackage{relsize}

 \tikzset{every node/.style={auto}}
 \tikzset{every state/.style={rectangle, minimum size=0pt, draw=none, font=\Large}}
  \tikzset{bend angle=7}
  
  \usepackage{array}
  
  \usepackage{chemfig}

\usepackage{authblk}

\makeatletter
\newcommand{\xrightleftarrows}[2]{
  \mathrel{\mathop{
    \vcenter{\offinterlineskip\m@th
      \ialign{\hfil##\hfil\cr
        \hphantom{$\scriptstyle\mspace{8mu}{#1}\mspace{8mu}$}\cr
        \rightarrowfill\cr
        \vrule height0pt width 2em\cr
        \leftarrowfill\cr
        \hphantom{$\scriptstyle\mspace{8mu}{#2}\mspace{8mu}$}\cr
        \noalign{\kern-0.3ex}
      }
    }
  }\limits^{#1}_{#2}}
}
\newcommand{\xleftrightarrows}[2]{
  \mathrel{\mathop{
    \vcenter{\offinterlineskip\m@th
      \ialign{\hfil##\hfil\cr
        \hphantom{$\scriptstyle\mspace{8mu}{#1}\mspace{8mu}$}\cr
        \leftarrowfill\cr
        \vrule height0pt width 2em\cr
        \rightarrowfill\cr
        \hphantom{$\scriptstyle\mspace{8mu}{#2}\mspace{8mu}$}\cr
        \noalign{\kern-0.3ex}
      }
    }
  }\limits^{#1}_{#2}}
}
\makeatother

\title{Autocatalysis and Boundary Stability in Chemical Reaction Systems}
\author{Matthew D. Johnston\thanks{Department of Mathematics and Computer Science, Lawrence Technological University, 21000 W 10 Mile Rd, Southfield, MI 48075.}\ \thanks{Corresponding author: \texttt{mjohnsto1@ltu.edu}} \; and Pol Jarne Cupons\textsuperscript{$\ast$}}

\begin{document}
\maketitle

\begin{abstract}
The next-generation matrix method is a powerful tool for computing the basic reproduction number in compartmental mathematical models of infectious diseases. The method has been recently extended to mathematical biochemistry, where it can be used to establish parameter regions of stability and instability for boundary steady states. Several significant challenges in its application remain, however, particularly around establishing conditions under which the method is mathematically valid, computationally tractable, and biologically meaningful in the biochemical setting. In this paper, we address these challenges by shifting the interpretation from new infections in the epidemiological setting to autocatalysis in the biochemical setting. We introduce a graph, called the ALT-graph (autocatalysis-leak-transition graph), which decomposes the contribution of each reaction to the Jacobian as autocatalytic, leak, or transition edges. We then present a systematic algorithm for splitting the ALT-graph, which is guaranteed to produce a valid reproduction number, $\rho(FV^{-1})$, while also decreasing computational complexity by lowering the rank of $FV^{-1}$. We apply the method to models of both biochemical reaction networks and infectious disease spread.
\end{abstract}


\section{Introduction}
\label{sec:introduction}

Boundary steady states are critical determinants in dynamical models from mathematical biology. In mathematical epidemiology, they correspond to disease-free states, and the \emph{stability} of such states is highly desired since this often corresponds to the long-term eradication of disease. In mathematical biochemistry, by contrast, they correspond to states of extinction of metabolites, and the \emph{instability} of such states is highly desired since this corresponds to the long-term survival of the biochemical species.  

Despite the common interest in whether species may tend toward zero, however, the approaches developed for boundary behavior have been largely domain-specific \cite{avram2024advancing,johnston2026boundary}. In mathematical epidemiology, the bifurcation threshold between stability and instability is governed by the parameter-dependent \emph{basic reproduction number}, which is the number of secondary infections produced by a typical primary infection in a fully susceptible population \cite{diekmann1990definition}. Van den Driessche and Watmough \cite{van2002reproduction,van2008further} established the widely used next-generation matrix (NGM) method for computing the basic reproduction number in SIR-type infectious disease models. This method has been applied to a broad range of epidemiological models, including multistrain models \cite{lazebnik2022generic}, vector-borne disease models \cite{van2017reproduction,van2017basic}, and vaccination models \cite{ehrhardt2019sir,johnston2025effect}.  In mathematical biochemistry, by contrast, boundary stability has typically been studied as a parameter-free structural property with the goal of establishing structural motifs that guarantee persistence of all species. Work has focused on network properties such as weak reversibility \cite{horn1972general,feinberg1972chemical,horn1972necessary}, endotacticity \cite{craciun2013persistence,gopalkrishnan2014geometric}, siphons \cite{angeli2007petri,shiu2010siphons}, and dynamic nonemptiability \cite{angeli2007petri,johnston2011weak}, among other structural properties \cite{leenheer2007monotone,angeli2010graph,siegel2011stratum,ali2020computational}.

Recent work has extended the NGM method to the study of boundary steady state stability in biochemical applications \cite{avram2024advancing,johnston2026boundary}. To contextualize our results, we recall the framework for the next-generation matrix method for establishing boundary steady state stability \cite{van2002reproduction,van2008further}. For the system $c'(t) = f(c)$, we start by reordering the species $c = (x, y)$, where $x$ are the species that vanish at the boundary steady state $c^* = (x^*, y^*)$. We then divide $x' = f(x, y)$ into $f(x, y) = \mathcal{F}(x, y) - \mathcal{V}(x, y)$ so that the Jacobian of the system on the vanishing species at the boundary steady state is given by $J = F-V$ where $J = D_xf(x^*,y^*) \in \mathbb{R}^{m \times m}$, $F = D_x\mathcal{F}(x^*,y^*) \in \mathbb{R}^{m \times m}$, and $V = D_x\mathcal{V}(x^*,y^*) \in \mathbb{R}^{m \times m}$. If $F$ is a nonzero nonnegative matrix and $V$ is a $Z$-matrix (i.e., a matrix with nonpositive off-diagonal entries) with nonnegative inverse $V^{-1}$, all eigenvalues of 
$J$ have strictly negative real part if and only if $\rho(FV^{-1}) < 1$, where $\rho(\cdot)$ is the spectral radius. Consequently, the bifurcation threshold for the boundary steady state is $\rho(FV^{-1}) = 1$, directly analogous to the epidemiological threshold $\mathscr{R}_0=1$. 

To guarantee the required mathematical conditions on $F$ and $V$ in mathematical epidemiology, it is sufficient to choose $\mathcal{F}$ to contain the ``new infections'' (see (A1)-(A5) of \cite{van2002reproduction}). Furthermore, the matrices $F$ and $V^{-1}$ can be interpreted biologically as the ``force of infection'' and ``dwell time'' matrices, respectively, so that the product $FV^{-1}$ represents the combined effect of the infectivity rate and time spent infectious. In mathematical biochemistry, however, there is no notion of ``new infections'' and consequently no natural criterion for obtaining the splitting $J=F-V$.  In \cite{johnston2026boundary}, the authors presented a heuristic for performing the splitting; however, several significant challenges remain:
\begin{enumerate}
\item How do we determine a splitting which is guaranteed to produce a valid NGM?
\item How do we determine a valid splitting which minimizes the computational complexity of the NGM?
\item How do we interpret the components of the splitting ($F$ and $V$) in biochemical, rather than epidemiological, terms? Specifically, what is the analogous notion in biochemistry to ``new infections?''
\end{enumerate}

In this paper, we propose autocatalysis as the biochemical analogue of new infections in mathematical epidemiology and therefore establish a critical first step in performing the Jacobian splitting $J = F - V$ for biochemical models. We connect this foundational work to the study of boundary steady state stability by introducing the autocatalysis-leak-transition graph, or ALT-graph (Definition \ref{def:alt-graph}). The ALT-graph identifies the sources of autocatalysis in the underlying chemical reaction network and provides explicit criteria for performing the splitting $J = F - V$ required by the next-generation matrix method. In this way, the method allows us to interpret $F$ as a ``force of autocatalysis'' matrix and $V^{-1}$ as a ``dwell time'' matrix, where ``force of infection'' in epidemiology can be seen as a special case of the ``force of autocatalysis.'' We demonstrate how this method can be used to simplify computational complexity of the reproduction number by guaranteeing low-rank matrices $F$ and $FV^{-1}$. 

To illustrate how the method works, consider the following chemical reaction network corresponding to a single-site futile cycle with a bifunctional enzyme \cite{shinar2009robustness,shinar2010structural}:
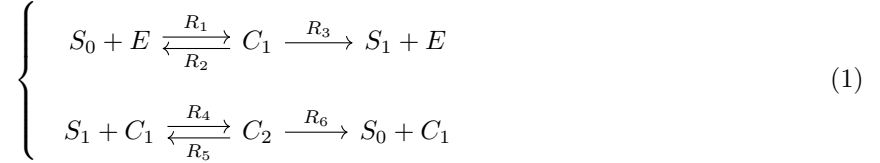
\begin{equation}
\label{eq:bifunctional} \left\{ \; \;
\begin{tikzcd}[ampersand replacement=\&]
            S_0 + E \arrow[r,yshift=0.5ex,"R_1"] \& C_1 \arrow[l,yshift=-0.5ex,"R_2"] \arrow[r,"R_3"] \& S_1 + E \\
            S_1 + C_1 \arrow[r,yshift=0.5ex,"R_4"] \& C_2 \arrow[l,yshift=-0.5ex,"R_5"] \arrow[r,"R_6"] \& S_0 + C_1
            \end{tikzcd} \right.
\end{equation}
Under general assumptions on the kinetic rate functions $R_k(c)$, the chemical reaction system corresponding to \eqref{eq:bifunctional} has a boundary steady state $(S_0,S_1,E,C_1,C_2) = (0,S_{tot},E_{tot},0,0)$ where $S_{tot} > 0$ and $E_{tot} > 0$ are conserved total amounts of substrate and enzyme, respectively. In order to determine a splitting $J = F - V$ which produces a valid next-generation matrix $FV^{-1}$, we relate \eqref{eq:bifunctional} to the following ALT-graphs, which embed the network's six reactions into a graph structure on the vanishing species set $\{ S_0, C_1, C_2\}$:
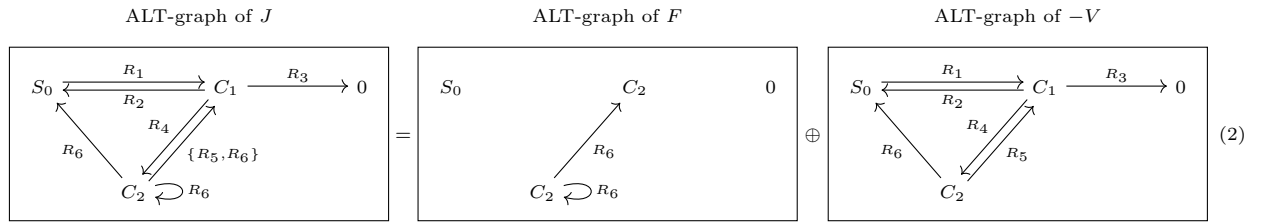
\begin{equation}
\label{eq:ex12}
\scriptsize
\begin{tikzcd}[ampersand replacement=\&,execute at end picture={
        \draw[black, thin] 
            ([shift={(-5pt,-5pt)}]current bounding box.south west) 
            rectangle 
            ([shift={(5pt,5pt)}]current bounding box.north east);
            \node[above] at ([shift={(0,5pt)}]current bounding box.north) {ALT-graph of $J$};
    }]
            S_0 \arrow[rr,yshift=0.5ex,"R_1"] \& \& C_1 \arrow[ll,yshift=-0.5ex,"R_2"] \arrow[ddl,xshift=-0.4ex,yshift=0.4ex,"R_4"']\arrow[rr,"R_3"] \& \& 0 \\ \\
              \& C_2 \arrow[uur,xshift=0.4ex,yshift=-0.4ex,"{\{R_5,R_6\}}"'] \arrow[uul,"R_6"] \arrow[loop right,"R_6"] \& \& \& 
            \end{tikzcd}
=
\begin{tikzcd}[ampersand replacement=\&,execute at end picture={
        \draw[black, thin] 
            ([shift={(-5pt,-5pt)}]current bounding box.south west) 
            rectangle 
            ([shift={(5pt,5pt)}]current bounding box.north east);
            \node[above] at ([shift={(0,5pt)}]current bounding box.north) {ALT-graph of $F$};
    }]
            S_0 \arrow[rr,white,yshift=0.5ex,"R_1"] \& \& C_2 \& \& 0 \\ \\
              \& C_2 \arrow[uur,"R_6"'] \arrow[loop right,"R_6"] \& \& \& 
            \end{tikzcd}
\; \oplus \;
            \begin{tikzcd}[ampersand replacement=\&,execute at end picture={
        \draw[black, thin] 
            ([shift={(-5pt,-5pt)}]current bounding box.south west) 
            rectangle 
            ([shift={(5pt,5pt)}]current bounding box.north east);
            \node[above] at ([shift={(0,5pt)}]current bounding box.north) {ALT-graph of $-V$};
    }
]
            S_0 \arrow[rr,yshift=0.5ex,"R_1"] \& \& C_1 \arrow[ll,yshift=-0.5ex,"R_2"] \arrow[ddl,xshift=-0.4ex,yshift=0.4ex,"R_4"']\arrow[rr,"R_3"] \& \& 0 \\ \\
              \& C_2 \arrow[uur,xshift=0.4ex,yshift=-0.4ex,"{R_5}"'] \arrow[uul,"R_6"]  \& \& \& 
            \end{tikzcd}
\end{equation}

An ALT-graph contains autocatalytic edges (A; loops), leaks (L; edges to zero), and transitions (T; edges between nonzero nodes). In order to compute a valid next-generation matrix, we split the ALT-graph in such a way that the ALT-graph of $F$ corresponds to an autocatalytic decomposition (Definition \ref{def:splitting}), the ALT-graph of $-V$ contains no critical autocatalytic edges or critical strongly connected components (Definition \ref{def:autocatalytic}), and there is a directed path from every vertex to $0$ in the ALT-graph of $-V$ (Theorem \ref{thm:1}). The ALT-graphs in the splitting \eqref{eq:ex12} correspond to the following ALT-matrices, which give the Jacobian decomposition $J = F - V$:
\[ \scriptsize
\begin{bmatrix}
-\partial_1 R_1 & \partial_2 R_2 & \partial_3 R_6 \\
\partial_1 R_1 & -\partial_2 R_2 - \partial_2 R_3 -\partial_2 R_4 & \partial_3 R_5+\partial_3 R_6 \\
0 & \partial_2 R_4 & -\partial_3 R_5-\partial_3 R_6
\end{bmatrix} = \begin{bmatrix} 0 & 0 & 0\\
0 & 0 & \partial_3 R_6 \\
0 & 0 & 0
\end{bmatrix} - \begin{bmatrix}
\partial_1 R_1 & -\partial_2 R_2 & -\partial_3 R_6 \\
-\partial_1 R_1 & \partial_2 R_2 + \partial_2 R_3 + \partial_2 R_4 & -\partial_3 R_5 \\
0 & -\partial_2 R_4 & \partial_3 R_5+\partial_3 R_6
\end{bmatrix}
\]
where $\partial_i R_k := \frac{\partial R_{k}(x^*,y^*)}{\partial x_i}$. Notice that $F$ is a nonzero nonnegative matrix and $V$ is an irreducible $Z$-matrix with each column sum nonnegative and at least one strictly positive. It follows from our earlier result Theorem 2 of \cite{johnston2026boundary} that this splitting produces a valid reproduction number $\rho(FV^{-1})$.

This paper is organized as follows: In Section \ref{sec:background}, we introduce the background definitions relevant to chemical reaction networks, chemical reaction systems, siphons, the next-generation matrix method, and multiset-valued graphs. In Section \ref{sec:results}, we introduce the ALT-graph (Definition \ref{def:alt-graph}) and establish the main theoretical results (Theorems \ref{thm:jacobian} and \ref{thm:1}) and algorithmic methods (Algorithm \ref{alg:1}). In Section \ref{sec:examples}, we present illustrative examples from mathematical biochemistry (e.g., EnvZ-OmpR signaling system, single- and two-site bifunctional enzyme models) and mathematical epidemiology (e.g., vector-borne disease models). In Section \ref{sec:conclusions}, we summarize our results and outline directions for future work. In the Appendices, we present the proofs of major results and technical details underlying the examples.

\section{Background}
\label{sec:background}

In this section, we present the background results and terminology required for the development of the ALT-graph of a chemical reaction network. Specifically, we introduce chemical reaction networks, chemical reaction systems, siphons, the next-generation matrix method, and multiset-valued graphs. We use the network \eqref{eq:bifunctional} as a running example.

\subsection{Chemical Reaction Networks}

We introduce the following foundational definitions \cite{horn1972general,feinberg2019foundations}.

\begin{definition}
\label{def:crn}
A \emph{chemical reaction network} $(\mathcal{S},\mathcal{R})$ consists of a \emph{species set} $\mathcal{S} = \{ S_1, \ldots, S_n \}$ and a \emph{reaction set} $\mathcal{R} = \{ R_1, \ldots, R_r \}$ where each reaction has the form
\[
R_{k}: \sum_{i=1}^n \alpha_{ik} S_i \longrightarrow \sum_{i=1}^n \beta_{ik} S_i
\]
where $\alpha_{ik}, \beta_{ik} \in \mathbb{R}_{\geq 0}$ are the \emph{stoichiometric coefficients} of the network. The \emph{stoichiometric matrix} $\Gamma \in \mathbb{R}^{n \times r}$ of a network is given by $\Gamma_{ik} = \beta_{ik} - \alpha_{ik}$. The \emph{stoichiometric subspace} is given by $S = \mbox{im}(\Gamma)$ and the conservation laws are given by $W = \{ w \in \mathbb{R}^n \mid w^T \Gamma = 0^T \}$. The \emph{support} of a conservation law $w \in W$ is given by $\mbox{supp}(w) = \{ S_i \in \mathcal{S} \mid w_i \not= 0 \}$. A chemical reaction network $(\mathcal{S},\mathcal{R})$ is furthermore said to be \emph{integer-valued} if all stoichiometric coefficients are integer-valued, i.e., $\alpha_{ik}, \beta_{ik} \in \mathbb{Z}_{\geq 0}$, and \emph{binary} if the source and target stoichiometric sums satisfy $\sum_{i=1}^n \alpha_{ik} \leq 2$ and $\sum_{i=1}^n \beta_{ik} \leq 2$ for every $R_k \in \mathcal{R}$.
\end{definition}

For simplicity, we restrict our consideration to chemical reaction networks $(\mathcal{S},\mathcal{R})$ which are integer-valued and binary. This implies that we only allow reactions with source and target complexes of the form $0$, $S_i$, $2S_i$, and $S_i + S_j$. As most naturally occurring elementary chemical reactions are assumed to be at most binary, this assumption does not significantly limit applicability.

\begin{definition}
\label{def:crs}
Given a chemical reaction network $(\mathcal{S},\mathcal{R})$ and a kinetic rate function $R: \mathbb{R}_{\geq 0}^n \to \mathbb{R}_{\geq 0}^r$, the corresponding \emph{chemical reaction system} is given by
\begin{equation}
\label{eq:crs}
\frac{dc}{dt} = \Gamma R(c)
\end{equation}
with concentration vector $c = (c_1, \ldots, c_n)^T \in \mathbb{R}_{\geq 0}^n$ and reaction rate vector $R(c) = (R_1(c), \ldots, R_r(c))^T \in \mathbb{R}_{\geq 0}^r$.
\end{definition}
\noindent From \eqref{eq:crs}, we have $\frac{dc}{dt} \in \mbox{im}(\Gamma) = S$. Consequently, trajectories of \eqref{eq:crs} are restricted to \emph{stoichiometric compatibility classes} $\mathsf{C}_{c_0} = (c_0 + S) \cap \mathbb{R}_{\geq 0}^n$ where $c(0)=c_0$ in the initial condition of \eqref{eq:crs}.

We make the following \emph{kinetic assumptions} on the components $R_k(c)$ of the reaction rate vector $R(c) = (R_1(c), \ldots, R_r(c))^T \in \mathbb{R}_{\geq 0}^r$, where $\delta_{ij} \in \{0, 1 \}$ for $i,j = 1, \ldots, n,$ is the Kronecker delta:
\begin{enumerate}
\item[] \textbf{(K1)} $R_k(c)$ is $\mathcal{C}^1$ on $\mathbb{R}_{\geq 0}^n$;
\item[] \textbf{(K2)} $R_k(c) \geq 0$ with $R_k(c) > 0$ if and only if $c_i > 0$ for all $i$ such that $\alpha_{ik} > 0$;
\item[] \textbf{(K3)} $\displaystyle \frac{\partial R_k(c)}{\partial c_i} = 0$ if $\alpha_{ik} = 0$;
\item[] \textbf{(K4)} $\displaystyle \frac{\partial R_k(c)}{\partial c_i} > 0$ if $\alpha_{ik} > 0$ and $c_j > 0$ for all $j$ such that $\alpha_{jk} - \delta_{ij} > 0$; and
\item[] \textbf{(K5)} $\displaystyle \frac{\partial R_k(c)}{\partial c_i} = 0$ if $\alpha_{ik} > 0$ but $c_j = 0$ for some $j$ such that $\alpha_{jk} - \delta_{ij} > 0$.
\end{enumerate}
We interpret these requirements as follows: \textbf{(K1)} reaction rates increase and decrease smoothly; \textbf{(K2)} reaction rates are nonnegative and positive if and only if all reactants are present;  \textbf{(K3)} reaction rates are insensitive to changes in nonreactants; \textbf{(K4)} reaction rates are strictly increasing in a given reactant if either: (a) all of its reactants are present, or (b) that reactant has multiplicity one and the rest of the reactants are present; \textbf{(K5)} reaction rates are insensitive to changes in absent reactants if either: (a) it has a multiplicity greater than one, or (b) there is another reactant which is absent.  

Notably, \textbf{(K1)}--\textbf{(K5)} are satisfied by the commonly used mass-action kinetics \cite{guldberg1867etudes}, the rate of each reaction is proportional to the product of the reactant concentrations \cite{guldberg1867etudes}. That is, reaction $R_k$ is assigned the rate $\kappa_k c_1^{\alpha_{1k}} \cdots c_n^{\alpha_{nk}}$ where $\kappa_k > 0$ is the rate constant of reaction $R_k$. A chemical reaction system \eqref{eq:crs} with mass-action kinetics is called a \emph{mass-action system}. 

The condition \textbf{(K5)} is unique to this paper. It is a generalization of a property found in mass-action kinetics at boundary steady states but rarely considered in more general settings. Specifically, for reactions of the form $2S_i \to \cdots$ and $S_i+S_j \to \cdots$, we have corresponding rates $R_k(c_i)=\kappa_k c_i^2$ and $R_k(c_i,c_j) = \kappa_k c_i c_j$, respectively. This gives the partial derivatives $\partial_i R_k = 2\kappa_k c_i$ and $\partial_i R_k = \kappa_k c_j$, respectively, both of which vanish at any boundary state where $c_i=c_j=0$. Kinetic assumption \textbf{(K5)} captures both of these situations.

\begin{example}
\label{ex:1000}
Consider the following chemical reaction network introduced in Section \ref{sec:introduction} as equation \eqref{eq:bifunctional}. 
This network has been studied in the context of absolute concentration robustness \cite{shinar2010structural,joshi2022foundations,joshi2023prevalence} and boundary steady state stability \cite{avram2024stability,johnston2026boundary,joshi2026bistability}. Under general kinetics, the network \eqref{eq:bifunctional} corresponds to the following chemical reaction system:
\begin{equation}
\label{eq:bifunctional-des} 
\left\{
\begin{aligned}
\frac{dS_0}{dt} &= -R_1(S_0,E) + R_2(C_1) + R_6(C_2) \\
\frac{dS_1}{dt} &= -R_4(S_1,C_1) + R_3(C_1) + R_5(C_2) \\
\frac{dE}{dt} &= -R_1(S_0,E) + R_2(C_1)+R_3(C_1) \\
\frac{dC_1}{dt} &= R_1(S_0,E) - R_2(C_1) - R_3(C_1) - R_4(S_1,C_1) + R_5(C_2) + R_6(C_2) \\
\frac{dC_2}{dt} &= R_4(S_1,C_1) - R_5(C_2)-R_6(C_2).
\end{aligned}
\right.
\end{equation}
The network \eqref{eq:bifunctional} has the stoichiometric subspace $S = \mbox{span} \left\{ (-1,0,-1,1,0), (0,1,1,-1,0), (0,-1,0,-1,1)\right\}$ so that $\mbox{dim}(S) = 3$, and the conservation laws $S_{tot} = S_0 + S_1 + C_1 + 2C_2$ and $E_{tot} = E + C_1 + C_2$. The network is integer-valued and binary, as all source and target complexes are of the form $S_i$ or $S_i + S_j$. Solutions of \eqref{eq:bifunctional-des} are restricted to the three-dimensional stoichiometric compatibility classes $\mathsf{C}_{c_0} = (c_0 + S) \cap \mathbb{R}_{\geq 0}^5$. The network \eqref{eq:bifunctional} admits the boundary steady state $(S_0,S_1,E,C_1,C_2) = (0,S_{tot},E_{tot},0,0)$ consistent with the conservation laws.
\end{example}

\subsection{Siphons and Siphon-Reduced Networks}
\label{sec:siphon}

The following concepts were originally developed for applications in Petri Net theory \cite{petri1962kommunikation,commoner1971marked} and was adapted to chemical reaction network theory in the context of persistence in \cite{angeli2007petri}. The concept has since been applied to boundary steady state stability and the next-generation matrix method in \cite{avram2025similarity,johnston2026boundary}.

\begin{definition}
\label{def:siphon}
Consider a chemical reaction network $(\mathcal{S},\mathcal{R}$). A nonempty subset $\mathcal{X} \subseteq \mathcal{S}$ is said to be a \emph{siphon} of $(\mathcal{S},\mathcal{R})$ if, for any reaction $R_k \in \mathcal{R}$, $\beta_{ik} > 0$ for some $S_i \in \mathcal{X}$ implies $\alpha_{jk} > 0$ for some $S_j \in \mathcal{X}$. The set $\mathcal{Y} = \mathcal{S} \setminus \mathcal{X}$ is called an \emph{anti-siphon}. A siphon $\mathcal{X} \subseteq \mathcal{S}$ is called \emph{critical} if it does not contain the support of any nonzero conservation law.
\end{definition}
\noindent In other words, $\mathcal{X}$ is a siphon if every reaction that produces a species in $\mathcal{X}$ also consumes a species in $\mathcal{X}$. 

Given a siphon $\mathcal{X} \subseteq \mathcal{S}$ and anti-siphon $\mathcal{Y} = \mathcal{S} \setminus \mathcal{X}$ with sizes $|\mathcal{X}| = m$ and $|\mathcal{Y}| = n-m$, we relabel and reindex the species set $\mathcal{S} = \{ S_1, \ldots, S_n \}$ so that $\mathcal{S} = \mathcal{X} \cup \mathcal{Y}$ with $\mathcal{X} = \{ X_1, \ldots, X_m\}$ and $\mathcal{Y} = \{ Y_1, \ldots, Y_{n-m} \}$. We relabel the corresponding stoichiometric coefficients $\alpha_{ik}$ and $\beta_{ik}$ so that the network reactions have the form
\begin{equation}
\label{eq:reactionxy}
R_{k}: \sum_{i=1}^m \alpha'_{ik} X_i + \sum_{i=1}^{n-m} \alpha''_{ik} Y_i \longrightarrow \sum_{i=1}^m \beta'_{ik} X_i + \sum_{i=1}^{n-m} \beta''_{ik} Y_i, \; \; \; k = 1, \ldots, r.
\end{equation}
The siphon property of $\mathcal{X}$ together with kinetic assumption \textbf{(K2)} implies that no species in $\mathcal{X}$ can be produced when all species in $\mathcal{X}$ are absent. Consequently, the boundary face of the chemical reaction system \eqref{eq:crs} with $x_i=0$ for every $X_i \in \mathcal{X}$ is positively invariant and the set of species vanishing at any boundary steady state $c^* \in \partial \mathbb{R}_{\geq 0}^n$ of \eqref{eq:crs} necessarily forms to a siphon. This reindexing allows us to introduce the following reduced network and corresponding reduced system.

\begin{definition}
\label{def:reduced-graph}
Consider a chemical reaction network $(\mathcal{S},\mathcal{R})$ with a siphon $\mathcal{X} \subseteq \mathcal{S}$. The \emph{$\mathcal{X}$-reduced network} of $(\mathcal{S},\mathcal{R})$ is the chemical reaction network $(\mathcal{X},\mathcal{R}')$ with reactions $R_k' \in \mathcal{R}'$ given by
\begin{equation}
\label{eq:reactionx}
R'_{k}: \sum_{i=1}^m \alpha'_{ik} X_i \longrightarrow \sum_{i=1}^m \beta'_{ik} X_i
\end{equation} 
for all $k \in \{ 1, \ldots, r\}$ where $\alpha'_{ik} \not= \beta'_{ik}$ for at least one $i \in \{ 1, \ldots, m \}$. The corresponding \emph{$\mathcal{X}$-reduced system} is
\begin{equation}
\label{eq:crs2}
\frac{dx}{dt} = \Gamma' R(x,y) =: f(x,y) 
\end{equation}
where $x = (x_1, \ldots, x_m)^T \in \mathbb{R}_{\geq 0}^m$ 
is the concentration vector of $\mathcal{X}$, $y = (y_1, \ldots, y_{n-m})^T \in \mathbb{R}_{\geq 0}^{n-m}$ is the concentration vector of $\mathcal{Y}$, $R(x,y) = (R_1(x,y), \ldots, R_r(x,y))^T \in \mathbb{R}_{\geq 0}^r$ is the reaction rate vector, and 
$\Gamma' \in \mathbb{R}^{m \times r}$ 
is the restriction of the stoichiometric matrix $\Gamma$ to $\mathcal{X}$ 
with entries $\Gamma'_{ik} = \beta'_{ik} - \alpha'_{ik}$. 
\end{definition}


\noindent In other words, the $\mathcal{X}$-reduced network is the network obtained by removing the species in the anti-siphon $\mathcal{Y}$ from the network (i.e., setting $\alpha''_{ik} = 0$ and $\beta''_{ik} = 0$ for $i=1, \ldots, n-m$ in \eqref{eq:reactionxy}). Note that the corresponding variables $y$ may appear in the reduced system \eqref{eq:crs2} through $f(x,y)$.

\begin{example}
\label{ex:1001}
Reconsider the network \eqref{eq:bifunctional} in Example \ref{ex:1000}. The network has the siphon $\mathcal{X} = \{ S_0, C_1, C_2 \}$ since every reaction with a species from $\mathcal{X}$ in its target has a species from $\mathcal{X}$ in its source. Since $\mathcal{X}$ is a siphon and $R(c)$ satisfies \textbf{(K2)}, we have that $(S_0^*,S_1^*,E^*,C_1^*,C_2^*) = (0,S_{tot},E_{tot},0,0)$ is a boundary steady state at which every species of $\mathcal{X}$ vanishes.

To highlight the system's behavior on the siphon $\mathcal{X}$, we reindex the species as $X_1 = S_0, X_2 = C_1, X_3 = C_2, Y_1 = S_1,$ and $Y_2 = E$ in accordance with \eqref{eq:reactionxy}. The network \eqref{eq:bifunctional} then becomes:
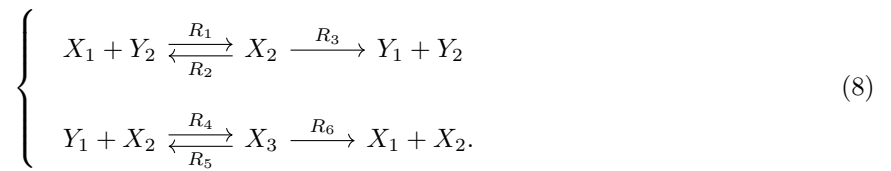
\begin{equation}
\label{eq:bifunctional2} \left\{ \; \;
\begin{tikzcd}[ampersand replacement=\&]
            X_1 + Y_2 \arrow[r,yshift=0.5ex,"R_1"] \& X_2 \arrow[l,yshift=-0.5ex,"R_2"] \arrow[r,"R_3"] \& Y_1 + Y_2 \\
            Y_1 + X_2 \arrow[r,yshift=0.5ex,"R_4"] \& X_3 \arrow[l,yshift=-0.5ex,"R_5"] \arrow[r,"R_6"] \& X_1 + X_2.
            \end{tikzcd} \right.
\end{equation}
This gives the following $\mathcal{X}$-reduced network $(\mathcal{X},\mathcal{R}')$ \eqref{eq:reactionx}:
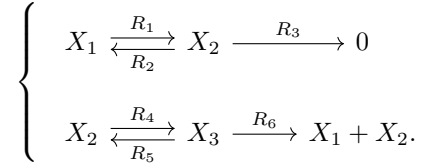
\[ \left\{ \; \;
\begin{tikzcd}[ampersand replacement=\&]
            X_1 \arrow[r,yshift=0.5ex,"R_1"] \& X_2 \arrow[l,yshift=-0.5ex,"R_2"] \arrow[r,"R_3"] \& 0 \\
            X_2 \arrow[r,yshift=0.5ex,"R_4"] \& X_3 \arrow[l,yshift=-0.5ex,"R_5"] \arrow[r,"R_6"] \& X_1 + X_2.
            \end{tikzcd} \right.
\]
The corresponding $\mathcal{X}$-reduced system \eqref{eq:crs2} is:
\begin{equation}
\label{eq:bifunctional-des2} 
\left\{
\begin{aligned}
\frac{dx_1}{dt} &= -R_1(x_1,y_2) + R_2(x_2) + R_6(x_3) \\
\frac{dx_2}{dt} &= R_1(x_1,y_2) - R_2(x_2) - R_3(x_2) - R_4(y_1,x_2) + R_5(x_3) + R_6(x_3) \\
\frac{dx_3}{dt} &= R_4(y_1,x_2) - R_5(x_3)-R_6(x_3).
\end{aligned}
\right.
\end{equation}
\end{example}





\subsection{Next-generation Matrix Method}

We now outline the next-generation matrix method for establishing the stability of a boundary steady state of a chemical reaction system \eqref{eq:crs}. This approach was pioneered in \cite{van2002reproduction,van2008further} as a method for computing the basic reproduction number in epidemiological models and adapted to boundary steady state stability for biochemical systems in \cite{avram2024advancing,johnston2026boundary}.

Consider a chemical reaction network $(\mathcal{S},\mathcal{R})$ and associated chemical reaction system \eqref{eq:crs} with boundary steady state $c^* \in \partial \mathbb{R}_{\geq 0}^n$ and associated siphon $\mathcal{X} \subseteq \mathcal{S}$. Reindex the species as in Section \ref{sec:siphon} so that $c = (x, y) \in \mathbb{R}_{\geq 0}^n$ and note that the boundary steady state has the form $c^* = (x^*, y^*)$ with $x^* = 0 \in \mathbb{R}^m$ and $y^* \in \mathbb{R}_{>0}^{n-m}$. 
We split the right-hand side of \eqref{eq:crs2} into $f(x,y) = \mathcal{F}(x,y) - \mathcal{V}(x,y)$ and compute the Jacobian with respect to $x$ and evaluated at $(x^*,y^*)$ to get $J = F-V$ where $J = D_xf(x^*, y^*)$, $F = D_x\mathcal{F}(x^*, y^*)$, and $V = D_x\mathcal{V}(x^*, y^*)$.

We have the following definition and results from our earlier work \cite{johnston2026boundary}.

\begin{definition}
\label{def:zmatrix}
A matrix $A \in \mathbb{R}^{d \times d}$ is a $Z$-matrix if its off-diagonal entries satisfy $A_{ij} \le 0$ for all $i \neq j$.
\end{definition}

\begin{theorem}[Theorem 1, \cite{johnston2026boundary}]
\label{thm:ngm}
Suppose $F$ is nonnegative and nonzero, $V$ is a $Z$-matrix with a nonnegative inverse, and $c^*$ is locally asymptotically stable with respect to the invariant subsystem at the intersection of the boundary face containing $c^*$ and the stoichiometric compatibility class $\mathsf{C}_{c_0}$. Then $c^*$ is locally asymptotically stable relative to $\mathsf{C}_{c_0}$ if $\rho(FV^{-1}) < 1$, and unstable if $\rho(FV^{-1}) > 1$, where $\rho(\cdot)$ denotes the spectral radius.
\end{theorem}

\begin{theorem}[Theorem 2, \cite{johnston2026boundary}]
\label{thm:ngm2}
Suppose that $V$ is a $Z$-matrix and that there is a permutation matrix $P$ such that
\[PVP^T = \left[ \begin{array}{cccc}
A_1 & B_{12} & \cdots & B_{1q} \\
0 & A_2 & \cdots & B_{2q} \\
\vdots & \vdots & \ddots & \vdots \\
0 & 0 & \cdots & A_q
\end{array}\right]
\]
where $B_{ij}$ are nonpositive rectangular matrices of appropriate size, $0$ denotes zero matrices of appropriate size, $\mathbbm{1} \in \mathbb{R}^d$ is the vector of all ones, and $A_i$ are irreducible $Z$-matrices which satisfy the condition that $\mathbbm{1}^T A_i$ is nonnegative and nonzero. Then $V$ has a nonnegative inverse.
\end{theorem}


Theorems \ref{thm:ngm} and \ref{thm:ngm2} provide conditions under which the stability of a boundary steady state can be determined from the spectral radius of the matrix $FV^{-1}$ rather than the eigenvalues of the subsystem Jacobian $J = F-V$. Crucially, $F$ and $V$ can often be chosen so that $FV^{-1}$ is significantly lower rank than $F-V$ and consequently its spectral radius is typically easier to compute. In our earlier work \cite{johnston2026boundary}, we presented a heuristic for producing valid and computationally efficient splittings $f(x,y) = \mathcal{F}(x,y) - \mathcal{V}(x,y)$; however, we left a rigorous algorithmic treatment and biological interpretation for future work. Establishing explicit conditions, algorithms, and biological interpretations for determining splittings is the central focus of this paper.

We also prove the following structural result, which is analogous to Lemma 1 of \cite{van2002reproduction}. The proof is in Appendix \ref{app:a}.

\begin{theorem}
\label{thm:2}
Consider a chemical reaction network $(\mathcal{S},\mathcal{R})$ satisfying kinetic assumptions \textbf{(K1)}--\textbf{(K5)}. Suppose there is a boundary steady state $c^* = (x^*,y^*) \in \partial \mathbb{R}_{\geq 0}^n$ with $x^* = 0$, where $\mathcal{X}$ is the set of vanishing species at $c^*$. Let $J = D_xf(x^*,y^*)$ denote the Jacobian of the $\mathcal{X}$-reduced system evaluated at $(x^*,y^*)$. Then $-J$ is a $Z$-matrix.
\end{theorem}

Consider the following example.

\begin{example}
\label{ex:1002}
Reconsider the network \eqref{eq:bifunctional} in Examples \ref{ex:1000} and \ref{ex:1001}. The Jacobian of \eqref{eq:bifunctional-des2} is given by
\begin{equation}
\label{eq:ex8}
J = \begin{bmatrix}
-\partial_1 R_1 & \partial_2 R_2 & \partial_3 R_6 \\
\partial_1 R_1 & -\partial_2 (R_2 + R_3 + R_4) & \partial_3( R_5+R_6) \\
0 & \partial_2 R_4 & -\partial_3 (R_5+ R_6)
\end{bmatrix}.
\end{equation}
Note that $\partial_i R_k > 0$ for all nonzero entries in \eqref{eq:ex8} by kinetic assumption \textbf{(K4)}, evaluated at the boundary steady state where $y^* \in \mathbb{R}_{> 0}^{n-m}$, and $\partial_i R_k = 0$ for the zero entries by \textbf{(K3)}. Consistent with Theorem \ref{thm:2}, $-J$ is a $Z$-matrix. The Jacobian can be decomposed as $J = F - V$ where
\[
F = \begin{bmatrix}
0 & 0 & 0 \\
0 & 0 & \partial_3 R_6 \\
0 & 0 & 0
\end{bmatrix} \mbox{ and } V = \begin{bmatrix}
\partial_1 R_1 & -\partial_2 R_2 & -\partial_3 R_6 \\
-\partial_1 R_1 & \partial_2 (R_2 + R_3 +R_4 ) & -\partial_3 R_5 \\
0 & -\partial_2 R_4 & \partial_3 (R_5+R_6)
\end{bmatrix}.
\]
We have that $F$ is nonnegative and nonzero, and $V$ is an irreducible $Z$-matrix. Furthermore, we have that $\mathbbm{1}^T V = (0, \partial_2 R_3, 0)$ which is nonnegative and nonzero so that Theorem \ref{thm:ngm2} implies $V$ is invertible. 
The next-generation matrix $FV^{-1}$ is given by
\[
FV^{-1} = \begin{bmatrix}
0 & 0 & 0 \\
\displaystyle\frac{\partial_2 R_4 \partial_3 R_6}{\partial_2 R_3\partial_3 (R_5 + R_6)} & \displaystyle\frac{\partial_2 R_4 \partial_3 R_6}{\partial_2 R_3 \partial_3 (R_5 + R_6)} & \displaystyle\frac{\partial_3 R_6 \partial_2 (R_3 + R_4)}{\partial_2 R_3\partial_3 (R_5 + R_6)} \\
0 & 0 & 0
\end{bmatrix}
\]
Since $F$ is a nonzero nonnegative matrix and $V^{-1}$ is nonnegative, the hypotheses of Theorem \ref{thm:ngm} are satisfied and we may compute the reproduction number. We have
\[
\rho(FV^{-1}) = \displaystyle\frac{\partial_2 R_4 \partial_3 R_6}{\partial_2 R_3 \partial_3 (R_5 + R_6)}.
\]
The steady state $c^*$ is locally asymptotically stable relative to its stoichiometric compatibility class if $\rho(FV^{-1}) < 1$, and unstable if $\rho(FV^{-1})>1$.
\end{example}

\subsection{Multisets and Directed Graphs}
\label{sec:directedgraph}

The following generalized notion of a set will be important when developing graphical approaches for applying the next-generation matrix method to biochemical reaction networks. This notion generalizes the concept of a set by allowing elements to appear with multiplicity greater than one \cite{blizard1991development,stanley2011enumerative}.

\begin{definition}
\label{def:multiset}
Consider a universal set $\Omega = \{ \omega_1, \ldots, \omega_n \}$. A \emph{multiset} $M$ over $\Omega$ is a pair $(\Omega, m_M)$ where $m_M \colon \Omega \to \mathbb{N}_0$, where $\mathbb{N}_0 = \{ 0, 1, 2, \ldots \}$ denotes the nonnegative integers, is a multiplicity function mapping each element $x \in \Omega$ to a number of occurrences $m_M(x)$. We also define the following:
\begin{enumerate}
    \item Given two multisets $A = (\Omega, m_A)$ and $B = (\Omega, m_B)$ over $\Omega$, the \emph{multiset sum} $A \oplus B = (\Omega, m_{A \oplus B})$ is the multiset defined by the operation $m_{A \oplus B}(\omega_i) = m_A(\omega_i) + m_B(\omega_i)$ for all $\omega_i \in \Omega$.
    \item Given two multisets $A = (\Omega, m_A)$ and $B = (\Omega, m_B)$ over $\Omega$, the \emph{multiset subtraction} $A \ominus B = (\Omega, m_{A \ominus B})$ is the multiset defined by the operation $m_{A \ominus B}(\omega_i) = \max(0, m_A(\omega_i) - m_B(\omega_i))$ for all $\omega \in \Omega$.
    \item The \emph{power multiset} of a universal set $\Omega$ (denoted $\mathcal{M}(\Omega)$) is the set of all possible multisets over $\Omega$.
\end{enumerate}
\end{definition}

We will represent multisets as enumerated sets where elements with multiple instances are listed multiple times, where each element $\omega_i$ is listed $m_A(\omega_i)$ times. Multisets will be important in the main results as we construct a new graph (called the ALT-graph) whose edges are mapped to multisets. We illustrate this notation in the following example.

\begin{example}
Consider the multisets $A = \{ \omega_1, \omega_1, \omega_2, \omega_3\}$ and $B = \{ \omega_1, \omega_3, \omega_3, \omega_3\}$ over the universal set $\Omega = \{ \omega_1, \omega_2, \omega_3\}$. That is, we have $m_A(\omega_1)=2$, $m_A(\omega_2) = 1$, $m_A(\omega_3) = 1$, $m_B(\omega_1)=1$, $m_B(\omega_2)=0$, and $m_B(\omega_3)=3$. Note that multiple instances of elements are allowed, as with $\omega_1$ in $A$ and $\omega_3$ in $B$. We use the notation $\omega_1 \in A$ to indicate that $\omega_1$ is in the set, which is equivalent to $m_A(\omega_1) \geq 1$. For example, we have $\omega_1 \in B$ but $\omega_2 \not\in B$. However, if we wish to specify that multiple instances of an element are in a set, we will use subset notation ($\subseteq$). For example, we have that $\{ \omega_1, \omega_1 \} \subseteq A$, which is equivalent to $m_A(\omega_1) \geq 2$, and $\{ \omega_3, \omega_3, \omega_3 \} \subseteq B$, which is equivalent to $m_B(\omega_3) \geq 3$.

Applying the multiset sum and subtraction rules, we have that $A \oplus B = \{ \omega_1, \omega_1, \omega_1, \omega_2, \omega_3, \omega_3, \omega_3, \omega_3 \}$ and $A \ominus B = \{ \omega_1, \omega_2\}$. For multiset sum, we combine the two sets without removing any instances of duplicated elements. For multiset subtraction, we must be more careful not to remove more instances of an element than are present. In this example, we may not take more instances of $\omega_3$ away from $A$ than are initially present so that we compute $m_{A \ominus B}(\omega_3) = \max(0, m_A(\omega_3) - m_B(\omega_3)) = \max(0,-2) = 0$. A consequence of this is that we have $(A \ominus B) \oplus B \not= A$ so that multiset subtract is not the inverse of multiset sum. This limitation is worth noting but will not affect the analysis that follows.
\end{example}

We now extend the notion of multisets to directed graphs by using multisets as edge labels.

\begin{definition}
\label{def:directedgraph}
A \emph{multiset-valued directed graph} $\mathcal{G} = (\mathcal{V}, \mathcal{E}, \Omega, \phi)$ consists of a vertex set $\mathcal{V}$, an edge set $\mathcal{E} \subseteq \mathcal{V} \times \mathcal{V}$, a universal label set $\Omega = \{\omega_1, \dots, \omega_n\}$, and a labeling function $\phi : \mathcal{V} \times \mathcal{V} \to \mathcal{M}(\Omega)$ such that $\phi(e) \neq \emptyset$ if and only if $e \in \mathcal{E}$. Given two multiset-valued directed graphs $\mathcal{G}_1 = (\mathcal{V}, \mathcal{E}_1, \Omega, \phi_1)$ and $\mathcal{G}_2 = (\mathcal{V}, \mathcal{E}_2, \Omega, \phi_2)$ on a common vertex set $\mathcal{V}$ and universal label set $\Omega$, we define the following operations:
\begin{enumerate}
    \item The \emph{graph sum} $\mathcal{G}_1 \oplus \mathcal{G}_2 = (\mathcal{V}, \mathcal{E}_{\oplus}, \Omega, \phi_{\oplus})$ has edge set $\mathcal{E}_{\oplus} = \left\{ e \in \mathcal{V} \times \mathcal{V} \;\middle|\; \phi_1(e) \oplus \phi_2(e) \neq \emptyset \right\}$ and labeling function $\phi_{\oplus}(e) = \phi_1(e) \oplus \phi_2(e)$ for all $e \in \mathcal{V} \times \mathcal{V}$.
    \item The \emph{graph subtraction} $\mathcal{G}_1 \ominus \mathcal{G}_2 = (\mathcal{V}, \mathcal{E}_{\ominus}, \Omega, \phi_{\ominus})$ has edge set $\mathcal{E}_{\ominus} = \left\{ e \in \mathcal{V} \times \mathcal{V} \;\middle|\; \phi_1(e) \ominus \phi_2(e) \neq \emptyset \right\}$ and labeling function $\phi_{\ominus}(e) = \phi_1(e) \ominus \phi_2(e)$ for all $e \in \mathcal{V} \times \mathcal{V}$.
\end{enumerate}
\end{definition}

\noindent A multiset-valued directed graph is a directed graph where each edge is labeled with a multiset drawn from some universal set $\Omega$ (Definition \ref{def:multiset}). The assignment of multisets to edges is given by the function $\phi$. 

Standard graph-theoretic notions---such as directed paths, reachability, and strong connectivity---apply to multiset-valued directed graphs, where an edge $e = (v_i, v_j)$ is considered present if and only if $\phi(e) \not= \emptyset$. We have the following definitions.

\begin{definition}
\label{def:connectivity}
Consider a multiset-valued directed graph $\mathcal{G} = (\mathcal{V}, \mathcal{E}, \Omega, \phi)$. We introduce the following graph-theoretic notions for multiset-valued directed graphs:
\begin{enumerate}
    \item Two vertices $v_i, v_j \in \mathcal{V}$ are \emph{connected} if there is a sequence $\{e_{i_0 i_1}, e_{i_1 i_2}, \dots, e_{i_{p-1} i_p}\}$ with $i_0 = i$ and $i_p = j$ such that for every $\ell = 1, \dots, p$, either $e_{i_{\ell-1} i_\ell} \in \mathcal{E}$ or $e_{i_\ell i_{\ell-1}} \in \mathcal{E}$. 

    \item There is a \emph{directed path} from $v_i \in \mathcal{V}$ to $v_j \in \mathcal{V}$ if there is a sequence $\{e_{i_0 i_1}, e_{i_1 i_2}, \dots, e_{i_{p-1} i_p}\}$ with $i_0 = i$ and $i_p = j$ such that $e_{i_{\ell-1} i_\ell} \in \mathcal{E}$. 

    \item A \emph{connected component} is a maximal nonempty subset $\mathcal{Z} \subseteq \mathcal{V}$ such that every $v_i \in \mathcal{Z}$ is connected to every $v_j \in \mathcal{Z}$.

    \item A \emph{strongly connected component} is a maximal nonempty subset $\mathcal{Z} \subseteq \mathcal{V}$ such that there is a directed path from every $v_i \in \mathcal{Z}$ to every $v_j \in \mathcal{Z}$. We let $\mathcal{Z}_1, \dots, \mathcal{Z}_c$ denote the strongly connected components of $\mathcal{G}$ and note that the strongly connected components form a partition of $\mathcal{V}$ so that $\mathcal{V} = \mathcal{Z}_1 \cup \dots \cup \mathcal{Z}_c$.
\end{enumerate}
\end{definition}

We illustrate these operations in the following example.

\begin{example}
Consider the vertex set $\mathcal{V} = \{ v_1, v_2, v_3 \}$, universal label set $\Omega = \{ \omega_1, \omega_2, \omega_3\}$, label functions $\phi_1$ and $\phi_2$ with $\phi_1(e_{11}) = \{ \omega_1 \}$, $\phi_1(e_{12}) = \{ \omega_1, \omega_2 \}$, $\phi_1(e_{13}) = \{ \omega_3 \}$, $\phi_1(e_{23})= \{ \omega_3, \omega_3 \}$, $\phi_1(e_{32})=\{ \omega_2 \}$, $\phi_2(e_{12}) = \{ \omega_1 \}$, $\phi_2(e_{23})= \{ \omega_3 \}$, $\phi_2(e_{32})=\{ \omega_2 \}$, and $\phi_1(e_{ij}) = \emptyset$ and $\phi_2(e_{ij}) = \emptyset$ for all remaining edges, and corresponding edge sets $\mathcal{E}_1 = \{ e_{11},e_{12}, e_{13}, e_{23}, e_{32} \}$ and $\mathcal{E}_2 = \{ e_{12}, e_{23}, e_{32} \}$. This gives the following multiset-valued directed graphs:\\
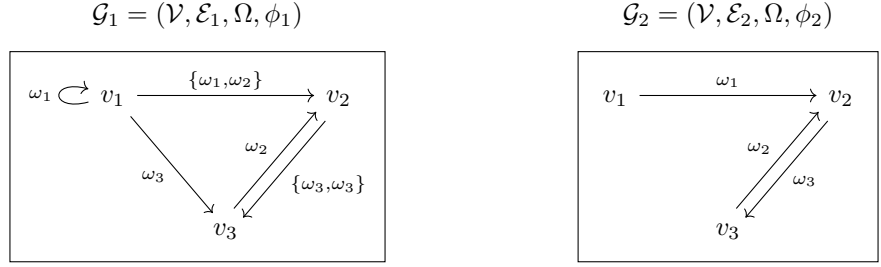
\[
\begin{tikzcd}[ampersand replacement=\&,execute at end picture={
        \draw[black, thin] 
            ([shift={(-5pt,-5pt)}]current bounding box.south west) 
            rectangle 
            ([shift={(5pt,5pt)}]current bounding box.north east);
            \node[above] at ([shift={(0,5pt)}]current bounding box.north) {$\mathcal{G}_1 = (\mathcal{V}, \mathcal{E}_1, \Omega, \phi_1)$};
    }]
v_1 \arrow[rr,"{\{\omega_1, \omega_2\}}"] \arrow[loop left,"\omega_1"] \arrow[rdd,"\omega_3"'] \& \& v_2 \arrow[ldd,xshift=0.4ex,yshift=-0.4ex,"{\{\omega_3, \omega_3\}}"] \\
\& \& \\
\& v_3 \arrow[ruu,xshift=-0.4ex,yshift=0.4ex,"\omega_2"] \&
\end{tikzcd} \hspace{1in} \begin{tikzcd}[ampersand replacement=\&,execute at end picture={
        \draw[black, thin] 
            ([shift={(-5pt,-5pt)}]current bounding box.south west) 
            rectangle 
            ([shift={(5pt,5pt)}]current bounding box.north east);
            \node[above] at ([shift={(0,5pt)}]current bounding box.north) {$\mathcal{G}_2 = (\mathcal{V}, \mathcal{E}_2, \Omega, \phi_2)$};
    }]
v_1 \arrow[rr,"{\color{white}\{\}}\omega_1{\color{white}\{\}}"] \& \& v_2 \arrow[ldd,xshift=0.4ex,yshift=-0.4ex,"\omega_3"] \\
\& \& \\
\& v_3 \arrow[ruu,xshift=-0.4ex,yshift=0.4ex,"\omega_2"] \&
\end{tikzcd}
\]
For notational simplicity, we label edges associated with a singleton set without set braces. Note that we allow self-loops (e.g., $\phi_1(e_{11}) = \{ \omega_1$\} in $\mathcal{G}_1$), edges associated with multiple labels (e.g., $\phi_1(e_{12}) = \{ \omega_1, \omega_2\}$ in $\mathcal{G}_1$), labels associated with multiple edges (e.g., $\omega_1 \in \phi_1(e_{11}) \cap \phi_1(e_{12})$ in $\mathcal{G}_1$), and multiple instances of a label associated with an edge (e.g., $\{ \omega_3, \omega_3 \} \subseteq \phi_1(e_{23})$ in $\mathcal{G}_1$). In both graphs, the connected component is the whole vertex set $\{ v_1, v_2, v_3 \}$ and the vertices are decomposed into strongly connected components $\mathcal{Z}_1 = \{ v_1 \}$ and $\mathcal{Z}_2 = \{ v_2, v_3 \}$.

Graph sum and subtraction are defined by applying the multiset sum and subtraction operations to each edge of the graph. We have the following:
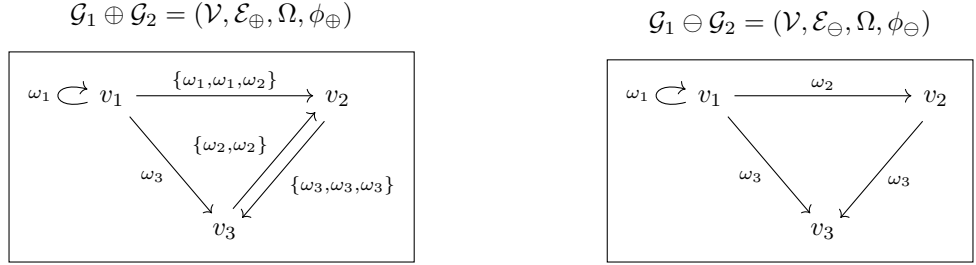
\[
\begin{tikzcd}[ampersand replacement=\&,execute at end picture={
        \draw[black, thin] 
            ([shift={(-5pt,-5pt)}]current bounding box.south west) 
            rectangle 
            ([shift={(5pt,5pt)}]current bounding box.north east);
            \node[above] at ([shift={(0,5pt)}]current bounding box.north) {$\mathcal{G}_1 \oplus \mathcal{G}_2 = (\mathcal{V}, \mathcal{E}_{\oplus}, \Omega, \phi_{\oplus})$};
    }]
v_1 \arrow[rr,"{\{\omega_1, \omega_1, \omega_2\}}"] \arrow[loop left,"\omega_1"] \arrow[rdd,"\omega_3"'] \& \& v_2 \arrow[ldd,xshift=0.4ex,yshift=-0.4ex,"{\{\omega_3, \omega_3, \omega_3\}}"] \\
\& \& \\
\& v_3 \arrow[ruu,xshift=-0.4ex,yshift=0.4ex,"{\{ \omega_2, \omega_2 \}}"] \&
\end{tikzcd} \hspace{1in} \begin{tikzcd}[ampersand replacement=\&,execute at end picture={
        \draw[black, thin] 
            ([shift={(-5pt,-5pt)}]current bounding box.south west) 
            rectangle 
            ([shift={(5pt,5pt)}]current bounding box.north east);
            \node[above] at ([shift={(0,5pt)}]current bounding box.north) {$\mathcal{G}_1 \ominus \mathcal{G}_2 = (\mathcal{V}, \mathcal{E}_{\ominus}, \Omega, \phi_{\ominus})$};
    }]
v_1 \arrow[rr,"\omega_2"] \arrow[loop left,"\omega_1"] \arrow[rdd,"\omega_3"'] \& \& v_2 \arrow[ldd,xshift=0.4ex,yshift=-0.4ex,"{\omega_3}"] \\
\& \& \\
\& v_3 \&
\end{tikzcd}
\]
Note that graph sum and subtraction can change the graph structure by adding or removing edges. For example, we have that the strongly connected components of $\mathcal{G}_1 \ominus \mathcal{G}_2$ are $\mathcal{Z}_1 = \{ v_1\}$, $\mathcal{Z}_2 =\{ v_2 \}$, and $\mathcal{Z}_3 = \{ v_3 \}$ since the removal of edge $e_{32}$ eliminates the directed cycle $v_2 \to v_3 \to v_2$, leaving no directed path from $v_3$ to $v_2$.
\end{example}

\section{Results}
\label{sec:results}

In this section, we introduce the ALT-graph of a chemical reaction network and a graph splitting algorithm for establishing a splitting $J=F-V$ which satisfies Theorem \ref{thm:ngm}. We note that the study of autocatalysis in chemical reaction networks has a rich history, with foundational work on self-organization and abiogenesis \cite{eigen1971selforganization,farmer1986autocatalytic,kauffman1986autocatalytic,hofbauer1981competition,hofbauer1984dynamics}. More recent work has focused on hypergraph representations \cite{hofbauer2002competitive,andersen2017chemical,joshi2021autocatalytic,craciun2022autocatalytic} and the cataloging of autocatalytic motifs \cite{blokhuis2020autocatalysis,andersen2021defining,golnik2026enumeration}. Our notion of autocatalysis is related to these approaches but has been developed independently.

\subsection{The Autocatalysis-Leak-Transition Graph (ALT-Graph)}
\label{sec:dependency}

\begin{table}[t]
\footnotesize
\centering
\begin{NiceTabular}{|cc|ccc|}[code-before = \rowcolors{2}{white}{gray!12}]
\toprule
 \begin{tabular}{c} $\mathcal{X}$-reduced \\ reaction ($R_k'$) \end{tabular} & \begin{tabular}{c}Jacobian \\ entry \end{tabular} & ALT-type &  \begin{tabular}{c} ALT-graph \\ edges \end{tabular} & \begin{tabular}{c} ALT-matrix \\ entry \end{tabular} \\
\midrule
\begin{tikzcd}[ampersand replacement=\&, baseline=(current bounding box.center)] X_i \arrow[r] \& 2X_i \end{tikzcd} &  $\begin{matrix} {}_i \\ {}_j \\ {}_{\ell} \end{matrix} \! \left[ \begin{matrix} \partial_i R_k \\ 0 \\ 0 \end{matrix} \right]$ &      \begin{tabular}{c} \textbf{Autocatalytic} \\ \textbf{(A)} \end{tabular}    & \begin{tikzcd}[ampersand replacement=\&, baseline=(current bounding box.center)] X_i \arrow[loop left,"{}_{R_k}"]  \end{tikzcd} & $\begin{bmatrix} \partial_i R_k \\ 0 \\ 0 \end{bmatrix}$ \\[15pt]
\begin{tikzcd}[ampersand replacement=\&, baseline=(current bounding box.center)] X_i \arrow[r] \& 0 \end{tikzcd} & $\begin{matrix} {}_i \\ {}_j \\ {}_{\ell} \end{matrix} \! \left[ \begin{matrix} -\partial_i R_k \\ 0 \\ 0 \end{matrix} \right]$ &      \begin{tabular}{c}  \textbf{Leak} \\ \textbf{(L)} \end{tabular}      & \begin{tikzcd}[ampersand replacement=\&, baseline=(current bounding box.center)] X_i \arrow[r,"{}_{R_k}"] \& 0 \end{tikzcd} & $\begin{bmatrix} -\partial_i R_k \\ 0 \\ 0 \end{bmatrix}$ \\[15pt]
\begin{tikzcd}[ampersand replacement=\&, baseline=(current bounding box.center)] X_i \arrow[r] \& X_j \end{tikzcd} & $\begin{matrix} {}_i \\ {}_j \\ {}_{\ell} \end{matrix} \! \left[ \begin{matrix} -\partial_i R_k \\ \partial_i R_k \\ 0 \end{matrix} \right]$ &    \begin{tabular}{c} \textbf{Transition}\\ \textbf{(T)} \end{tabular}     & \begin{tikzcd}[ampersand replacement=\&, baseline=(current bounding box.center)] X_i \arrow[r,"{}_{R_k}"] \& X_j \end{tikzcd} & $\begin{bmatrix} -\partial_i R_k \\ \partial_i R_k \\ 0 \end{bmatrix}$ \\[15pt]
\hline
\begin{tikzcd}[ampersand replacement=\&, baseline=(current bounding box.center)] X_i \arrow[r] \& X_i + X_j \end{tikzcd} &  $\begin{matrix} {}_i \\ {}_j \\ {}_{\ell} \end{matrix} \! \left[ \begin{matrix} 0 \\ \partial_i R_k \\ 0 \end{matrix} \right]$  &      \begin{tabular}{c} \textbf{Catalytic} \\ \textbf{Transition I} \\ \textbf{(CT-I)} \end{tabular}        & \begin{tikzcd}[ampersand replacement=\&, baseline=(current bounding box.center)] X_i \arrow[r,"{}_{R_k}"] \arrow[loop left,"{}_{R_k}"] \& X_j \end{tikzcd} &  $\begin{bmatrix} \partial_i R_k \\ 0 \\ 0 \end{bmatrix} + \begin{bmatrix} -\partial_i R_k \\ \partial_i R_k \\ 0 \end{bmatrix}$ \\[15pt]
\begin{tikzcd}[ampersand replacement=\&, baseline=(current bounding box.center)] X_i \arrow[r] \& 2X_j \end{tikzcd} &  $\begin{matrix} {}_i \\ {}_j \\ {}_{\ell} \end{matrix} \! \left[ \begin{matrix} -\partial_i R_k \\ 2\partial_i R_k \\ 0 \end{matrix} \right]$ &   \begin{tabular}{c} \textbf{Catalytic} \\ \textbf{Transition II} \\ \textbf{(CT-II)} \end{tabular}      & \begin{tikzcd}[ampersand replacement=\&, baseline=(current bounding box.center)] X_i \arrow[r,"{}_{{\{ R_k, R_k \}}}"] \arrow[loop left,"{}_{R_k}"] \& X_j \end{tikzcd} &  $\begin{bmatrix} \partial_i R_k \\ 0 \\ 0 \end{bmatrix} + 2\begin{bmatrix} -\partial_i R_k \\ \partial_i R_k \\ 0 \end{bmatrix}$  \\[15pt]
\begin{tikzcd}[ampersand replacement=\&, baseline=(current bounding box.center)] X_i \arrow[r] \& X_j + X_{\ell} \end{tikzcd} & $\begin{matrix} {}_i \\ {}_j \\ {}_{\ell} \end{matrix} \! \left[ \begin{matrix} -\partial_i R_k \\ \partial_i R_k \\ \partial_i R_k \end{matrix} \right]$ &    \begin{tabular}{c} \textbf{Catalytic} \\ \textbf{Transition III} \\ \textbf{(CT-III)} \end{tabular}      & \begin{tikzcd}[ampersand replacement=\&, baseline=(current bounding box.center)] X_j \& X_i \arrow[l,"{}_{R_k}"'] \arrow[loop below,"{}_{R_k}" right] \arrow[r,"{}_{R_k}"] \& X_\ell \end{tikzcd} & $\begin{bmatrix} \partial_i R_k \\ 0 \\ 0 \end{bmatrix} + \begin{bmatrix} -\partial_i R_k \\ \partial_i R_k \\ 0 \end{bmatrix} + \begin{bmatrix} -\partial_i R_k \\ 0 \\ \partial_i R_k \end{bmatrix}$\\[6pt]
\bottomrule
\end{NiceTabular}
\caption{Table of correspondences between reduced reactions $R'_{k} \in \mathcal{R}'$ of an integer-valued binary chemical reaction network (column 1, Definition \ref{def:crn}) and edges of an ALT-graph (column 4, Definition \ref{def:alt-graph}). In column 2, we give the reaction's contribution to the Jacobian of the corresponding chemical reaction system \eqref{eq:crs2} evaluated at the boundary steady state $c^* = (x^*, y^*)$. In column 5, we give the reaction's contribution to the ALT-matrix (see Definition \ref{def:alt-matrix}). In both cases, we let $\partial_i R_k = \frac{\partial R_{k}(x^*,y^*)}{\partial x_i}$.
\newline We classify the first three reactions in the ALT-graph as Autocatalytic (A), Leak (L), and Transition (T) and note that the three remaining reaction types can be represented as combinations of these three types of edges. The contribution of each reaction to the Jacobian (column 2) can be decomposed in the ALT-matrix (column 5) into a linear combination of all contributions of edges in the ALT-graph (column 4). Notably, the first three reaction types are native to mathematical epidemiology as well as mathematical biochemistry. Specifically, autocatalytic reactions can be interpreted as new infections, leaks as recovery events, and transitions as other changes of state, such as transitions between infectious or partially immune states. The lower three reaction types are native to mathematical biochemistry and significantly less common in epidemiology, although reactions of the fourth type can represent infection events between individuals of different types.}
\label{tab:jacobian}
\end{table}


The following is the central construction of this paper. The construction is summarized in Table \ref{tab:jacobian}.

\begin{definition}
\label{def:alt-graph}
Consider an integer-valued and binary chemical reaction network $(\mathcal{S}, \mathcal{R})$ with a siphon $\mathcal{X} \subseteq \mathcal{S}$ and corresponding $\mathcal{X}$-reduced network $(\mathcal{X}, \mathcal{R}')$. The \emph{autocatalysis-leak-transition graph} (or \emph{ALT-graph}) of $(\mathcal{S}, \mathcal{R})$ on $\mathcal{X}$ is the multiset-valued directed graph $\mathcal{G} = (\mathcal{V}, \mathcal{E}, \Omega, \phi)$ with: (i) vertex set $\mathcal{V} = \mathcal{X} \cup \{0\}$; (ii) edge set $\mathcal{E} \subseteq \mathcal{V} \times \mathcal{V}$; (iii) universal label set $\Omega = \mathcal{R}$; and (iv) minimal labeling function $\phi : \mathcal{V} \times \mathcal{V} \to \mathcal{M}(\Omega)$ with $\phi(e) \neq \emptyset$ if and only if $e \in \mathcal{E}$, satisfying the rules below for each reduced reaction $R'_k \in \mathcal{R}'$ where $R_k \in \mathcal{R}$ is the underlying reaction:
\begin{enumerate}
    \item[(a)] \textbf{Autocatalytic (A):} If $R'_k : X_i \to 2X_i$, then $R_k \in \phi(e_{ii})$;
    \item[(b)] \textbf{Leak (L):} if $R'_k : X_i \to 0$, then $R_k \in \phi(e_{i0})$;
    \item[(c)] \textbf{Transition (T):} if $R'_k : X_i \to X_j$, then $R_k \in \phi(e_{ij})$;
    \item[(d)] \textbf{Catalytic Transition I (CT-I):} if $R'_k : X_i \to X_i + X_j$, then $R_k \in \phi(e_{ii})$ and $R_k \in \phi(e_{ij})$;
    \item[(e)] \textbf{Catalytic Transition II (CT-II):} if $R'_k : X_i \to 2X_j$, then $R_k \in \phi(e_{ii})$ and $\{ R_k, R_k \} \subseteq \phi(e_{ij})$; and
    \item[(f)] \textbf{Catalytic Transition III (CT-III):} if $R'_k : X_i \to X_j + X_\ell$, then $R_k \in \phi(e_{ii})$, $R_k \in \phi(e_{ij})$, and $R_k \in \phi(e_{i\ell})$.
\end{enumerate}
where minimality means that $R_k \in \phi(e_{ij})$ only if required by rules (a)--(f) above. We say that the occurrence $R_k \in \phi(e_{ij})$ is an \emph{instance} of reaction $R_k$ in the ALT-graph.
\end{definition}

The ALT-graph partitions reactions into three fundamental components, based on their contribution to the Jacobian restricted to $\mathcal{X}$: autocatalytic edges ($e_{ii}$, $i\not=0$), leak edges ($e_{i0}$, $i\not=0$), and transition edges ($e_{ij}$, $i\not=j$, $i, j \not= 0$). The correspondence between individual reaction types in a chemical reaction network and their representation in the ALT-graph is given in Table \ref{tab:jacobian}.

Note that the conditions (a)--(f) in Definition \ref{def:alt-graph}---and the corresponding partition into autocatalytic, leak, and transition edges---being exhaustive depends strictly on the assumptions that the chemical reaction network is integer-valued, binary, and satisfies \textbf{(K1)}--\textbf{(K5)}. Specifically, the assumption that the network is integer-valued and binary allows us to restrict to consideration of sources and targets of the form $0$, $X_i$, $2X_i$, and $X_i+X_j$, while kinetic assumptions \textbf{(K3)} and \textbf{(K5)} allow us to remove consideration of reactions with sources $0$, $2X_i$, and $X_i + X_j$ as these reactions make no contribution to the Jacobian. This also implies that the ALT-graph never contains edges from $0$ so that we may not have edges $e_{00}$ or $e_{0i}$, $i \not= 0$. We note that most elementary reactions in applications are integer-valued and at most binary and that most common kinetic rate functions, such as mass-action kinetics, satisfy assumptions \textbf{(K1)}--\textbf{(K5)}. Consideration of more general stoichiometries and kinetic assumptions is left for future work.

The ALT-graph is related to the Jacobian $J$ through the following matrix structure.

\begin{definition}
\label{def:alt-matrix}
The \emph{autocatalysis-leak-transition matrix} (or ALT-matrix) of an ALT-graph $\mathcal{G} = (\mathcal{V},\mathcal{E},\Omega,\phi)$ is the matrix $L(\mathcal{G}) \in \mathbb{R}^{m \times m}$ with entries
\[
[L(\mathcal{G})]_{ji} = \begin{cases}
\displaystyle \sum_{R_k \in \phi(e_{ii})} \partial_i R_k - \sum_{R_k \in \phi(e_{i0})} \partial_i R_k -\sum_{\substack{p=1 \\ p \neq i}}^m \sum_{R_k \in \phi(e_{i p})} \partial_i R_k , \;\;\; & \text{if } i = j \\
\displaystyle \sum_{R_k \in \phi(e_{ij})} \partial_i R_k, & \text{if } i \neq j
\end{cases}
\]
where $\partial_i R_k := \frac{\partial R_k(x^*,y^*)}{\partial x_i}$ for $i,j = 1, \ldots, m$. Note that the summations over any $\phi(\cdot)$ are evaluated over multisets; that is, terms are summed according to their multiplicities, counting repeated elements as distinct terms in the sum. Note also that index conventions are transposed so that the edge $e_{ij}$ contributes to row $j$ and column $i$ of $L(\mathcal{G})$.
\end{definition}


The following result formalizes the connection between the ALT-graph and the Jacobian $J$. The proof is contained in Appendix \ref{app:b}.

\begin{theorem}
\label{thm:jacobian}
Consider an integer-valued and binary chemical reaction network $(\mathcal{S},\mathcal{R})$ satisfying kinetic assumptions \textbf{(K1)--(K5)}. Suppose the corresponding chemical reaction system \eqref{eq:crs} has a boundary steady state $c^* \in \partial \mathbb{R}_{\geq 0}^n$ with corresponding siphon $\mathcal{X} \subseteq \mathcal{S}$. Let $\mathcal{G} = (\mathcal{V}, \mathcal{E}, \Omega, \phi)$ be the ALT-graph of $(\mathcal{S},\mathcal{R})$ on $\mathcal{X}$ (Definition \ref{def:alt-graph}). Then $L(\mathcal{G})$ coincides with the Jacobian $J$ of the $\mathcal{X}$-reduced system \eqref{eq:crs2}, evaluated at $c^*$.
\end{theorem}

\begin{example}
\label{ex:1003}

\begin{table}[t]
\centering
\begin{NiceTabular}{ccccccc}[code-before = \rowcolors{2}{white}{gray!12}]
\toprule
\begin{tabular}{c}Reaction\\ label \end{tabular} & & \begin{tabular}{c}Chemical \\ reaction \end{tabular} & & \begin{tabular}{c}$\mathcal{X}$-reduced\\reaction\end{tabular} & & \begin{tabular}{c} ALT-graph\\edge(s)\end{tabular} \\
\midrule
$R_1:$ & & \begin{tikzcd}[ampersand replacement=\&, baseline=(current bounding box.center)] X_1 + Y_2 \arrow[r] \& X_2 \end{tikzcd} & $\Longrightarrow$ & \begin{tikzcd}[ampersand replacement=\&, baseline=(current bounding box.center)] X_1 \arrow[r] \& X_2 \end{tikzcd} & $\Longrightarrow$ & \begin{tikzcd}[ampersand replacement=\&, baseline=(current bounding box.center)] X_1 \arrow[r,"{}_{R_1}"] \& X_2 \end{tikzcd} \\[0pt]
$R_2:$ & & \begin{tikzcd}[ampersand replacement=\&, baseline=(current bounding box.center)] X_2 \arrow[r] \& X_1 + Y_2 \end{tikzcd} & $\Longrightarrow$ & \begin{tikzcd}[ampersand replacement=\&, baseline=(current bounding box.center)] X_2 \arrow[r] \& X_1 \end{tikzcd} & $\Longrightarrow$ & \begin{tikzcd}[ampersand replacement=\&, baseline=(current bounding box.center)] X_2 \arrow[r,"{}_{R_2}"] \& X_1 \end{tikzcd} \\[0pt]
$R_3:$ & & \begin{tikzcd}[ampersand replacement=\&, baseline=(current bounding box.center)] X_2 \arrow[r] \& Y_1 + Y_2 \end{tikzcd} & $\Longrightarrow$ & \begin{tikzcd}[ampersand replacement=\&, baseline=(current bounding box.center)] X_2 \arrow[r] \& 0 \end{tikzcd} & $\Longrightarrow$ & \begin{tikzcd}[ampersand replacement=\&, baseline=(current bounding box.center)] X_2 \arrow[r,"{}_{R_3}"] \& 0 \end{tikzcd} \\[0pt]
$R_4:$ & & \begin{tikzcd}[ampersand replacement=\&, baseline=(current bounding box.center)] X_2 + Y_1\arrow[r] \& X_3 \end{tikzcd} & $\Longrightarrow$ & \begin{tikzcd}[ampersand replacement=\&, baseline=(current bounding box.center)] X_2 \arrow[r] \& X_3 \end{tikzcd} & $\Longrightarrow$ & \begin{tikzcd}[ampersand replacement=\&, baseline=(current bounding box.center)] X_2 \arrow[r,"{}_{R_4}"] \& X_3 \end{tikzcd} \\[0pt]
$R_5:$ & & \begin{tikzcd}[ampersand replacement=\&, baseline=(current bounding box.center)] X_3 \arrow[r] \& X_2 + Y_1 \end{tikzcd} & $\Longrightarrow$ & \begin{tikzcd}[ampersand replacement=\&, baseline=(current bounding box.center)] X_3 \arrow[r] \& X_2 \end{tikzcd} & $\Longrightarrow$ & \begin{tikzcd}[ampersand replacement=\&, baseline=(current bounding box.center)] X_3 \arrow[r,"{}_{R_5}"] \& X_2 \end{tikzcd} \\[0pt]
$R_6:$ & & \begin{tikzcd}[ampersand replacement=\&, baseline=(current bounding box.center)] X_3 \arrow[r] \& X_1 + X_2 \end{tikzcd} & $\Longrightarrow$ & \begin{tikzcd}[ampersand replacement=\&, baseline=(current bounding box.center)] X_3 \arrow[r] \& X_1 + X_2 \end{tikzcd} & $\Longrightarrow$ & \begin{tikzcd}[ampersand replacement=\&, baseline=(current bounding box.center)] X_1 \& X_3 \arrow[l,"{}_{R_6}"'] \arrow[loop below,"{}_{R_6}" right] \arrow[r,"{}_{R_6}"] \& X_2 \end{tikzcd} \\
\bottomrule
\end{NiceTabular}
\caption{Shown above are the steps which take each reaction (left) from the single-site bifunctional enzyme network \eqref{eq:bifunctional2} to the corresponding ALT-graph edges (right). First, we construct the $\mathcal{X}$-reduced edge (middle) by removing the species not in the set $\mathcal{X} = \{ X_1, X_2 , X_3 \}$. We then construct the ALT-graph edges (right) by noting that $R_3$ is a leak edge, $R_6$ is a Catalytic transition (type III) edge, and the rest are transition edges. We apply the corresponding rules given in Table \ref{tab:jacobian}. Combining the ALT-graph edges (right) to a single graph on the vertices $\mathcal{X} \cup \{ 0 \}$ gives the ALT-graph \eqref{eq:ex4}.}
\label{tab:jacobian1}
\end{table}

Reconsider the network \eqref{eq:bifunctional2} in Example \ref{ex:1000}, \ref{ex:1001}, and \ref{ex:1002}. We now construct the ALT-graph. We start by constructing the $\mathcal{X}$-reduced network according to Definition \ref{def:reduced-graph}. We then construct the ALT-graph by applying the rules in Definition \ref{def:alt-graph}. These transitions are given in Table \ref{tab:jacobian1}. Combining these edges gives the following ALT-graph:
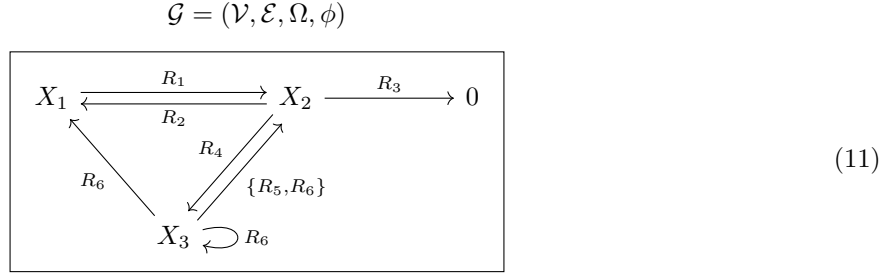
\begin{equation}
\label{eq:ex4}
\begin{tikzcd}[ampersand replacement=\&,execute at end picture={
        \draw[black, thin] 
            ([shift={(-5pt,-5pt)}]current bounding box.south west) 
            rectangle 
            ([shift={(5pt,5pt)}]current bounding box.north east);
            \node[above] at ([shift={(0,5pt)}]current bounding box.north) {$\mathcal{G} = (\mathcal{V}, \mathcal{E}, \Omega, \phi)$};
    }]
            X_1 \arrow[rr,yshift=0.5ex,"R_1"] \& \& X_2 \arrow[ll,yshift=-0.5ex,"R_2"] \arrow[ddl,xshift=-0.4ex,yshift=0.4ex,"R_4"']\arrow[rr,"R_3"] \& \& 0 \\ \\
              \& X_3 \arrow[uur,xshift=0.4ex,yshift=-0.4ex,"{\{R_5,R_6\}}"'] \arrow[uul,"R_6"] \arrow[loop right,"R_6"] \& \& \& 
            \end{tikzcd}
\end{equation}
We notice that there are three instances of $R_6 \in \mathcal{R}$ in the ALT-graph, specifically, $R_6 \in \phi(e_{31})$, $R_6 \in \phi(e_{32})$, and $R_6 \in  \phi(e_{33})$. All elements of $\mathcal{X}$ are connected by directed edges (e.g., $X_3$ is connected to $X_1$ through $e_{31}$ and $X_1$ is connected to $X_3$ by $e_{12}$ and $e_{23}$) so that $\mathcal{X}$ is a strongly connected component. The ALT-matrix (Definition \ref{def:alt-matrix}) of the ALT-graph \eqref{eq:ex4} coincides with the Jacobian matrix \eqref{eq:ex8}, which can be obtained by direct calculation but is guaranteed by Theorem \ref{thm:jacobian}.
\end{example}

\subsection{Critical Autocatalytic Reactions and Strongly Connected Components}

In order to apply the next-generation matrix method to biochemical examples, we need to understand how reactions are embedded into the edges of the ALT-graph. We introduce the following terminology.

\begin{definition}
\label{def:autocatalytic}
Consider an integer-valued and binary chemical reaction network $(\mathcal{S}, \mathcal{R})$ with siphon $\mathcal{X} \subseteq \mathcal{S}$, and ALT-graph $\mathcal{G} = (\mathcal{V}, \mathcal{E}, \Omega, \phi)$ with strongly connected components $\mathcal{Z} = \{\mathcal{Z}_1, \dots, \mathcal{Z}_c\}$.
\begin{enumerate}
    \item A reaction $R_k \in \mathcal{R}$ is \emph{autocatalytic} if $R_k \in \phi(e_{ii})$ for some $X_i$. Given an autocatalytic reaction $R_k \in \mathcal{R}_A$, an occurrence $R_k \in \phi(e_{ii})$ is an \emph{autocatalytic instance} while an occurrence $R_k \in \phi(e_{ij})$, $i \neq j$, is a \emph{catalytic instance}. Note that we allow $j=0$ for catalytic instances. The set of autocatalytic reactions is denoted $\mathcal{R}_A \subseteq \mathcal{R}$.
    \item An autocatalytic reaction $R_k \in \mathcal{R}_A$ is \emph{critical} if it does not have a catalytic instance $R_k \in \phi(e_{ij})$ satisfying $X_i \in \mathcal{Z}_h$ and $X_j \not\in \mathcal{Z}_h$ for some strongly connected component $\mathcal{Z}_h$. The set of critical autocatalytic reactions is denoted $\mathcal{R}_C \subseteq \mathcal{R}_A$.
    \item A strongly connected component $\mathcal{Z}_h \subseteq \mathcal{X}$ is \emph{critical} if $\mathcal{Z}_h$ can be disconnected from $0$ (i.e., there is no directed path from $\mathcal{Z}_h$ to $0$) by removing at most one catalytic instance $R_k \in \phi(e_{ij})$ satisfying $X_i \in \mathcal{Z}_h$ and $X_j \not\in \mathcal{Z}_h$ of each autocatalytic reaction $R_k \in \mathcal{R}_A$. The set of critical components is denoted $\mathcal{Z}_C \subseteq \mathcal{Z}$.
\end{enumerate}
\end{definition}

Note that every instance of an autocatalytic reaction $R_k \in \mathcal{R}_A$ is either an autocatalytic or catalytic instance by Definition \ref{def:alt-graph}. Also note that a strongly connected component is trivially critical if it is disconnected from $0$. The notions of critical autocatalytic reactions and critical strongly connected components are important as they prevent the corresponding ALT-matrix (Definition \ref{def:alt-matrix}) from satisfying the conditions of Theorem \ref{thm:ngm2}. Specifically, critical autocatalytic reactions prevent $V$ from satisfying the property that $\mathbbm{1}^T A_i$ is nonnegative and nonzero, while critical strongly connected components prevent $V$ from being invertible (see the examples in Appendix \ref{app:f} and \ref{app:g}). The term critical is used to indicate that the object in question requires resolution in our analysis. 

To illustrate the concepts of autocatalytic reactions and criticality, we reconsider our earlier example.

\begin{example}
\label{ex:1004}
Reconsider the network \eqref{eq:bifunctional} in Examples \ref{ex:1000}, \ref{ex:1001}, \ref{ex:1002}, and \ref{ex:1003}. We now assess criticality in the ALT-graph \eqref{eq:ex4} according to Definition \ref{def:autocatalytic}. 
The set of autocatalytic reactions is $\mathcal{R}_A = \{ R_6 \}$ since $R_6 \in \phi(e_{33})$. $R_6 \in \phi(e_{33})$ is an autocatalytic instance while the instances $R_6 \in \phi(e_{31})$ and $R_6 \in \phi(e_{32})$ are catalytic instances. The ALT-graph has strongly connected components given by $\mathcal{Z}_1 = \{X_1, X_2, X_3\}$ and the trivial sink component $\mathcal{Z}_2 = \{ 0 \}$. Since neither catalytic instance $R_6 \in \phi(e_{31})$ nor $R_6 \in \phi(e_{32})$ satisfies $X_j \not\in \mathcal{Z}_1$, we have that the critical reaction set is $\mathcal{R}_C = \{ R_6 \}$. 
Furthermore, since $R_6 \in \mathcal{R}_C$ has no catalytic instances which leave $\mathcal{Z}_1$, there is no way to remove a set of at most one catalytic instance of an autocatalytic reaction which could disconnect $\mathcal{Z}_1$ from $0$. Consequently, there are no critical strongly connected components so that $\mathcal{Z}_C = \emptyset$.

\end{example}

\subsection{Splittings and Decompositions}

In this section, we introduce an ALT-graph-based technique for performing the splitting $J=F-V$ which satisfies Theorem \ref{thm:ngm}. 

\begin{definition}
\label{def:splitting}
Consider an integer-valued and binary chemical reaction network $(\mathcal{S}, \mathcal{R})$ with a siphon $\mathcal{X} \subseteq \mathcal{S}$, and corresponding ALT-graph $\mathcal{G} = (\mathcal{V}, \mathcal{E}, \Omega, \phi)$. A pair of ALT-graphs, $\mathcal{G}_F = (\mathcal{V}, \mathcal{E}_F, \Omega, \phi_F)$ and $\mathcal{G}_{-V} = (\mathcal{V}, \mathcal{E}_{-V}, \Omega, \phi_{-V})$, is said to be a \emph{splitting} of $\mathcal{G}$ if $\mathcal{G} = \mathcal{G}_F \oplus \mathcal{G}_{-V}$. Furthermore, a splitting $\mathcal{G} = \mathcal{G}_F \oplus \mathcal{G}_{-V}$ is called:
    \begin{enumerate}
        \item an \emph{autocatalytic splitting} of $\mathcal{G}$ if $\mathcal{G}_F$ contains only a single instance of an autocatalytic reaction $R_k \in \mathcal{R}_A$, i.e., $\phi_F(e_{ii}) = \{R_k\}$ for some $X_i \in \mathcal{X}$, and $\phi_F(e) = \emptyset$ for all $e \in \mathcal{V} \times \mathcal{V} \setminus \{e_{ii}\}$.
        
        \item a \emph{catalytic splitting} of $\mathcal{G}$ if $\mathcal{G}_F$ contains only a single pair of autocatalytic and catalytic instances of an autocatalytic reaction $R_k \in \mathcal{R}_A$, i.e., $\phi_F(e_{ii}) = \{R_k\}$ and $\phi_F(e_{ij}) = \{R_k\}$ for some $X_i, X_j \in \mathcal{X}, i \neq j$, and $\phi_F(e) = \emptyset$ for all $e \in \mathcal{V} \times \mathcal{V} \setminus \{e_{ii}, e_{ij}\}$.
        \item an \emph{autocatalytic decomposition} of $\mathcal{G}$ if there is a sequence of splittings $\mathcal{G}_{-V_i} = \mathcal{G}_{F_{i+1}} \oplus \mathcal{G}_{-V_{i+1}}$, $i = 0, \ldots, p-1$, where each splitting is either autocatalytic or catalytic, with $\mathcal{G}_F = \mathcal{G}_{F_1} \oplus \cdots \oplus \mathcal{G}_{F_p}$, $\mathcal{G} = \mathcal{G}_{-V_0}$, and $\mathcal{G}_{-V} = \mathcal{G}_{-V_p}$. Note that an autocatalytic decomposition can be written as the following:
        \begin{equation}
        \label{eq:decomposition}
        \mathcal{G} = \mathcal{G}_{F_1} \oplus \mathcal{G}_{-V_1} = \mathcal{G}_{F_1} \oplus ( \mathcal{G}_{F_2} \oplus \mathcal{G}_{-V_2}) = \cdots = (\mathcal{G}_{F_1} \oplus \cdots \oplus \mathcal{G}_{F_p} ) \oplus \mathcal{G}_{-V_p} = \mathcal{G}_F \oplus \mathcal{G}_{-V}.
        \end{equation}
    \end{enumerate}
\end{definition}
\noindent We use the notation $\mathcal{G}_F$ and $\mathcal{G}_{-V}$ due to the correspondence with the matrices $F$ and $-V$ in the Jacobian decomposition $J = F- V$. 

In practice, we iteratively perform autocatalytic and catalytic splittings to obtain an autocatalytic decomposition $\mathcal{G} = \mathcal{G}_F \oplus \mathcal{G}_{-V}$ \eqref{eq:decomposition} satisfying special properties. 
In Section \ref{sec:mainresult}, we provide conditions on the final graphs $\mathcal{G}_{F}$ and $\mathcal{G}_{-V}$ that guarantee the assumptions of Theorem \ref{thm:ngm2} are satisfied.

An immediate consequence of Definitions \ref{def:alt-matrix} and the graph sum operation (Definition \ref{def:directedgraph}) is the following, which allows us to relate autocatalytic decompositions $\mathcal{G} = \mathcal{G}_F \oplus \mathcal{G}_{-V}$ to the Jacobian splitting $J = F - V$. The proof is contained in Appendix \ref{app:c}.

\begin{lemma}\label{lemma:multiset}
Consider two ALT-graphs, $\mathcal{G}_1 = (\mathcal{V}, \mathcal{E}_1, \Omega, \phi_1)$ and $\mathcal{G}_2 = (\mathcal{V}, \mathcal{E}_2, \Omega, \phi_2)$ on the same vertex set $\mathcal{V}=\mathcal{X}\cup\{0\}$ and universal label set $\Omega = \mathcal{R}$, with corresponding ALT-matrices $L(\mathcal{G}_1)$ and $L(\mathcal{G}_2)$, respectively. Then $L(\mathcal{G}_1 \oplus \mathcal{G}_2) = L(\mathcal{G}_1) + L(\mathcal{G}_2)$.
\end{lemma}

To illustrate autocatalytic and catalytic decompositions, consider the following example.

\begin{example}
\label{ex:1005}
Reconsider the chemical reaction network \eqref{eq:bifunctional} considered in Examples \ref{ex:1000}, \ref{ex:1001}, \ref{ex:1002}, \ref{ex:1003}, and \ref{ex:1004}. There are several splittings of $\mathcal{G} = (\mathcal{V}, \mathcal{E}, \Omega, \phi)$, some autocatalytic and some catalytic.

\paragraph{Splitting \#1 (Autocatalytic):} We have the following autocatalytic splitting of $\mathcal{G}$:
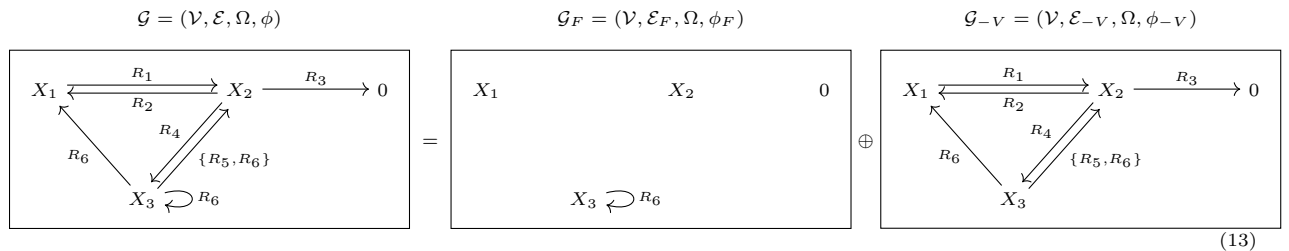
\begin{equation}
\label{eq:1006-1} \scriptsize
\begin{tikzcd}[ampersand replacement=\&,execute at end picture={
        \draw[black, thin] 
            ([shift={(-5pt,-5pt)}]current bounding box.south west) 
            rectangle 
            ([shift={(5pt,5pt)}]current bounding box.north east);
            \node[above] at ([shift={(0,5pt)}]current bounding box.north) {$\mathcal{G} = (\mathcal{V}, \mathcal{E}, \Omega, \phi)$};
    }]
            X_1 \arrow[rr,yshift=0.5ex,"R_1"] \& \& X_2 \arrow[ll,yshift=-0.5ex,"R_2"] \arrow[ddl,xshift=-0.4ex,yshift=0.4ex,"R_4"']\arrow[rr,"R_3"] \& \& 0 \\ \\
              \& X_3 \arrow[uur,xshift=0.4ex,yshift=-0.4ex,"{\{R_5,R_6\}}"'] \arrow[uul,"R_6"] \arrow[loop right,"R_6"] \& \& \& 
            \end{tikzcd}
            \; = \;
            \begin{tikzcd}[ampersand replacement=\&,execute at end picture={
        \draw[black, thin] 
            ([shift={(-5pt,-5pt)}]current bounding box.south west) 
            rectangle 
            ([shift={(5pt,5pt)}]current bounding box.north east);
            \node[above] at ([shift={(0,5pt)}]current bounding box.north) {$\mathcal{G}_F = (\mathcal{V}, \mathcal{E}_F, \Omega, \phi_F)$};
    }]
            X_1 \arrow[rr,yshift=0.5ex,white,"R_1"]\& \& X_2 \& \& 0 \\ \\
              \& X_3 \arrow[loop right,"R_6"] \& \& \& 
            \end{tikzcd}
            \; \oplus \;
            \begin{tikzcd}[ampersand replacement=\&,execute at end picture={
        \draw[black, thin] 
            ([shift={(-5pt,-5pt)}]current bounding box.south west) 
            rectangle 
            ([shift={(5pt,5pt)}]current bounding box.north east);
            \node[above] at ([shift={(0,5pt)}]current bounding box.north) {$\mathcal{G}_{-V} = (\mathcal{V}, \mathcal{E}_{-V}, \Omega, \phi_{-V})$};
    }
]
            X_1 \arrow[rr,yshift=0.5ex,"R_1"] \& \& X_2 \arrow[ll,yshift=-0.5ex,"R_2"] \arrow[ddl,xshift=-0.4ex,yshift=0.4ex,"R_4"']\arrow[rr,"R_3"] \& \& 0 \\ \\
              \& X_3 \arrow[uur,xshift=0.4ex,yshift=-0.4ex,"{\{R_5,R_6\}}"'] \arrow[uul,"R_6"]  \& \& \& 
            \end{tikzcd}
\end{equation}
The splitting is autocatalytic because $\mathcal{G}_F$ contains a single autocatalytic ($R_6 \in \phi(e_{33})$) instance of the autocatalytic reaction $R_6 \in \mathcal{R}_A$. Note that an instance of a reaction may be autocatalytic or catalytic in the original ALT-graph without being autocatalytic or catalytic in the ALT-graph it is assigned to in the splitting. For example, the instance $R_6 \in \phi(e_{31})$ is catalytic in $\mathcal{G}$ but not in $\mathcal{G}_{-V}$ because there are no autocatalytic instance $R_6 \in \phi(e_{33})$ in that network. 

\paragraph{Splitting \#2 (Catalytic):} We also have the following catalytic splitting of $\mathcal{G}$:
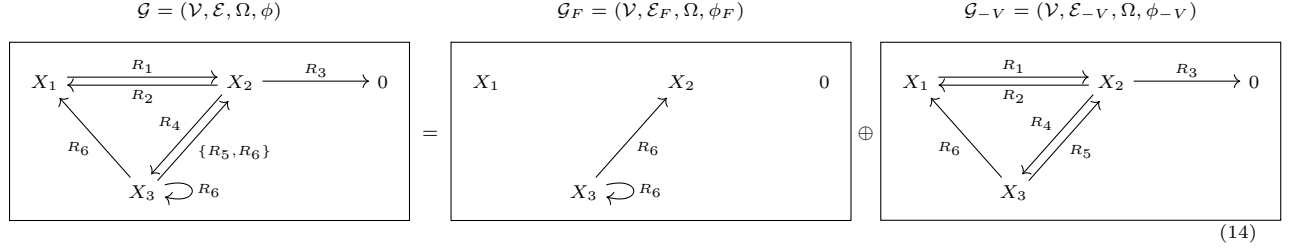
\begin{equation}
\label{eq:1006-2}
 \scriptsize
\begin{tikzcd}[ampersand replacement=\&,execute at end picture={
        \draw[black, thin] 
            ([shift={(-5pt,-5pt)}]current bounding box.south west) 
            rectangle 
            ([shift={(5pt,5pt)}]current bounding box.north east);
            \node[above] at ([shift={(0,5pt)}]current bounding box.north) {$\mathcal{G} = (\mathcal{V}, \mathcal{E}, \Omega, \phi)$};
    }]
            X_1 \arrow[rr,yshift=0.5ex,"R_1"] \& \& X_2 \arrow[ll,yshift=-0.5ex,"R_2"] \arrow[ddl,xshift=-0.4ex,yshift=0.4ex,"R_4"']\arrow[rr,"R_3"] \& \& 0 \\ \\
              \& X_3 \arrow[uur,xshift=0.4ex,yshift=-0.4ex,"{\{R_5,R_6\}}"'] \arrow[uul,"R_6"] \arrow[loop right,"R_6"] \& \& \& 
            \end{tikzcd}
            \; = \;
            \begin{tikzcd}[ampersand replacement=\&,execute at end picture={
        \draw[black, thin] 
            ([shift={(-5pt,-5pt)}]current bounding box.south west) 
            rectangle 
            ([shift={(5pt,5pt)}]current bounding box.north east);
            \node[above] at ([shift={(0,5pt)}]current bounding box.north) {$\mathcal{G}_F = (\mathcal{V}, \mathcal{E}_F, \Omega, \phi_F)$};
    }]
            X_1 \arrow[rr,yshift=0.5ex,white,"R_1"]\& \& X_2 \& \& 0 \\ \\
              \& X_3 \arrow[uur,"R_6"'] \arrow[loop right,"R_6"] \& \& \& 
            \end{tikzcd}
            \; \oplus \;
            \begin{tikzcd}[ampersand replacement=\&,execute at end picture={
        \draw[black, thin] 
            ([shift={(-5pt,-5pt)}]current bounding box.south west) 
            rectangle 
            ([shift={(5pt,5pt)}]current bounding box.north east);
            \node[above] at ([shift={(0,5pt)}]current bounding box.north) {$\mathcal{G}_{-V} = (\mathcal{V}, \mathcal{E}_{-V}, \Omega, \phi_{-V})$};
    }
]
            X_1 \arrow[rr,yshift=0.5ex,"R_1"] \& \& X_2 \arrow[ll,yshift=-0.5ex,"R_2"] \arrow[ddl,xshift=-0.4ex,yshift=0.4ex,"R_4"']\arrow[rr,"R_3"] \& \& 0 \\ \\
              \& X_3 \arrow[uur,xshift=0.4ex,yshift=-0.4ex,"{R_5}"'] \arrow[uul,"R_6"]  \& \& \& 
            \end{tikzcd}
\end{equation}
This splitting is catalytic because $\mathcal{G}_F$ contains both an autocatalytic instance ($R_6 \in \phi(e_{33})$) and a catalytic instance ($R_6 \in \phi(e_{32})$) of the autocatalytic reaction $R_6 \in \mathcal{R}_A$.

\paragraph{Splitting \#3 (Catalytic):} A second catalytic splitting is the following:
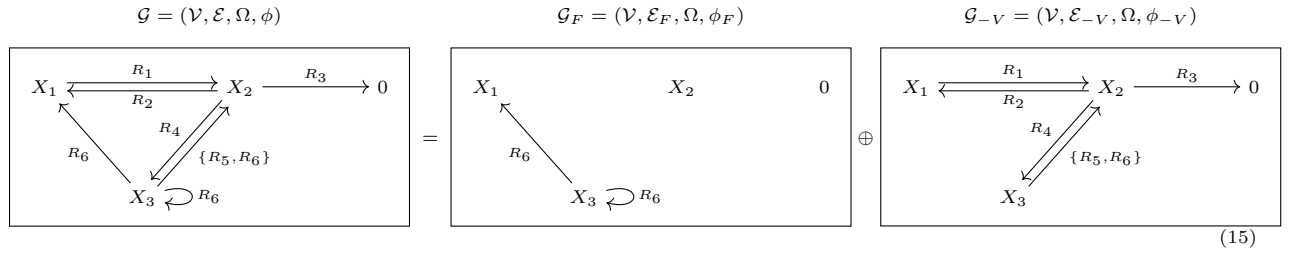
\begin{equation}
\label{eq:1006-3}
\scriptsize
\begin{tikzcd}[ampersand replacement=\&,execute at end picture={
        \draw[black, thin] 
            ([shift={(-5pt,-5pt)}]current bounding box.south west) 
            rectangle 
            ([shift={(5pt,5pt)}]current bounding box.north east);
            \node[above] at ([shift={(0,5pt)}]current bounding box.north) {$\mathcal{G} = (\mathcal{V}, \mathcal{E}, \Omega, \phi)$};
    }]
            X_1 \arrow[rr,yshift=0.5ex,"R_1"] \& \& X_2 \arrow[ll,yshift=-0.5ex,"R_2"] \arrow[ddl,xshift=-0.4ex,yshift=0.4ex,"R_4"']\arrow[rr,"R_3"] \& \& 0 \\ \\
              \& X_3 \arrow[uur,xshift=0.4ex,yshift=-0.4ex,"{\{R_5,R_6\}}"'] \arrow[uul,"R_6"] \arrow[loop right,"R_6"] \& \& \& 
            \end{tikzcd}
            \; = \;
            \begin{tikzcd}[ampersand replacement=\&,execute at end picture={
        \draw[black, thin] 
            ([shift={(-5pt,-5pt)}]current bounding box.south west) 
            rectangle 
            ([shift={(5pt,5pt)}]current bounding box.north east);
            \node[above] at ([shift={(0,5pt)}]current bounding box.north) {$\mathcal{G}_F = (\mathcal{V}, \mathcal{E}_F, \Omega, \phi_F)$};
    }]
            X_1 \arrow[rr,yshift=0.5ex,white,"R_1"]\& \& X_2 \& \& 0 \\ \\
              \& X_3 \arrow[uul,"R_6"] \arrow[loop right,"R_6"] \& \& \& 
            \end{tikzcd}
            \; \oplus \;
            \begin{tikzcd}[ampersand replacement=\&,execute at end picture={
        \draw[black, thin] 
            ([shift={(-5pt,-5pt)}]current bounding box.south west) 
            rectangle 
            ([shift={(5pt,5pt)}]current bounding box.north east);
            \node[above] at ([shift={(0,5pt)}]current bounding box.north) {$\mathcal{G}_{-V} = (\mathcal{V}, \mathcal{E}_{-V}, \Omega, \phi_{-V})$};
    }
]
            X_1 \arrow[rr,yshift=0.5ex,"R_1"] \& \& X_2 \arrow[ll,yshift=-0.5ex,"R_2"] \arrow[ddl,xshift=-0.4ex,yshift=0.4ex,"R_4"']\arrow[rr,"R_3"] \& \& 0 \\ \\
              \& X_3 \arrow[uur,xshift=0.4ex,yshift=-0.4ex,"{\{R_5,R_6\}}"']   \& \& \& 
            \end{tikzcd}
\end{equation}
This splitting is catalytic because we have a pair of autocatalytic ($R_6 \in \phi(e_{33})$) and catalytic ($R_6 \in \phi(e_{31})$) instances of the autocatalytic reaction $R_6 \in \mathcal{R}_A$.

\end{example}

\subsection{Main Result}
\label{sec:mainresult}

We now present the main result of this paper. The proof is contained in Appendix \ref{app:d}.

\begin{theorem}
\label{thm:1}
Consider an integer-valued and binary chemical reaction network $(\mathcal{S},\mathcal{R})$ satisfying \textbf{(K1)}--\textbf{(K5)} with siphon $\mathcal{X} \subseteq \mathcal{S}$ and boundary steady state $c^* = (x^*, y^*)$. Suppose that the ALT-graph $\mathcal{G} = (\mathcal{V},\mathcal{E}, \Omega, \phi)$ of the network has an autocatalytic decomposition $\mathcal{G} = \mathcal{G}_F \oplus \mathcal{G}_{-V}$ with $\mathcal{G}_F = (\mathcal{V},\mathcal{E}_F, \Omega, \phi_F)$ and $\mathcal{G}_{-V} = (\mathcal{V},\mathcal{E}_{-V}, \Omega, \phi_{-V})$ where $\mathcal{E}_F \not= \emptyset$ and $\mathcal{G}_{-V}$ does not contain any critical autocatalytic reactions or critical strongly connected components. Then the Jacobian of \eqref{eq:crs2} with respect to $x$ and evaluated at $c^*$ can be split as $J = F-V$ where $J$ is the ALT-matrix of $\mathcal{G}$, $F$ is the ALT-matrix of $\mathcal{G}_F$, and $-V$ is the ALT-matrix of $\mathcal{G}_{-V}$ (Definition \ref{def:alt-matrix}). Furthermore, $F$ is nonnegative and nonzero, and $V$ is a $Z$-matrix with nonnegative inverse.\newline \indent It follows that, if $c^*$ is locally asymptotically stable with respect to the invariant subsystem at the intersection of the boundary face containing $c^*$ and the stoichiometric compatibility class $\mathsf{C}_{c_0}$, then $c^*$ is locally asymptotically stable with respect to $\mathsf{C}_{c_0}$ if $\rho(FV^{-1}) < 1$ and $c^*$ is unstable with respect to $\mathsf{C}_{c_0}$ if $\rho(FV^{-1}) > 1$.
\end{theorem}

Theorem \ref{thm:1} identifies the key process by which the ALT-graph of a chemical reaction network must be split in order to guarantee a valid splitting $J = F - V$ of the Jacobian. Specifically, we must: construct the ALT-graph; then identify the critical reactions and critical strongly connected components; then systematically remove autocatalytic and catalytic instances of autocatalytic reactions until no critical reactions or critical strongly connected components remain. The autocatalytic and catalytic instances which are removed from the graph are aggregated in $\mathcal{G}_F$ while $\mathcal{G}_{-V}$ contains all remaining instances. The matrices $F$ and $-V$ correspond to the ALT-matrices of the respective graphs. Removal of critical autocatalytic reactions and critical strongly connected components guarantees that the resulting $V$ has a nonnegative inverse by Theorem \ref{thm:ngm2}, which allows Theorem \ref{thm:ngm} to be applied.

The following result is a consequence of Theorem \ref{thm:1}. The proofs of each part are contained in Appendix \ref{app:e}.

\begin{corollary}
\label{cor:2}
Consider an integer-valued and binary chemical reaction network $(\mathcal{S},\mathcal{R})$ satisfying \textbf{(K1)}--\textbf{(K5)} with siphon $\mathcal{X} \subseteq \mathcal{S}$, boundary steady state $c^* = (x^*, y^*)$, and ALT-graph $\mathcal{G} = (\mathcal{V},\mathcal{E}, \Omega, \phi)$. Suppose that there is a directed path from every vertex $X_i \in \mathcal{X}$ to $0$. Suppose that, for any choice of kinetic rate functions satisfying \textbf{(K1)}--\textbf{(K5)}, the corresponding boundary steady state $c^*$ is locally asymptotically stable with respect to the invariant subsystem at the intersection of the boundary face containing $c^*$ and the stoichiometric compatibility class $\mathsf{C}_{c_0}$.
\begin{enumerate}
\item Suppose that $\mathcal{G}$ does not have any critical autocatalytic reactions or critical strongly connected components. Then the boundary steady state $c^*$ is locally asymptotically stable with respect to its stoichiometric compatibility class $\mathsf{C}_{c^*}$ for any choice of kinetic rate functions satisfying \textbf{(K1)}--\textbf{(K5)}. In particular, for mass-action systems, stability holds for all choices of rate constants $\kappa_k > 0$, $k= 1, \ldots, r$.
\item Suppose that $\mathcal{G}$ has at least one critical autocatalytic reaction or critical strongly connected component (i.e., $\mathcal{R}_C \not= \emptyset$ or $\mathcal{Z}_C \not= \emptyset$). Then there is an autocatalytic decomposition $\mathcal{G} = \mathcal{G}_F \oplus \mathcal{G}_{-V}$ with corresponding matrices $F = L(\mathcal{G}_{F})$ and $-V = L(\mathcal{G}_{-V})$ where $F$ is a nonnegative nonzero matrix, $V$ is a $Z$-matrix with nonnegative inverse, and $\mbox{rank}(FV^{-1}) \leq |\mathcal{R}_C| + | \mathcal{Z}_C |$.
\end{enumerate}
\end{corollary}

\noindent Corollary \ref{cor:2} part 2. gives upper bounds on the rank of the next-generation matrix $FV^{-1}$, which simplifies computation of the reproduction number $\rho(FV^{-1})$.

\begin{example}
\label{ex:1006}
Reconsider the chemical reaction network \eqref{eq:bifunctional} in Examples \ref{ex:1000}, \ref{ex:1001}, \ref{ex:1002}, \ref{ex:1003}, \ref{ex:1004}, and \ref{ex:1005}. In Example \ref{ex:1005}, we presented three autocatalytic decompositions (\eqref{eq:1006-1}, \eqref{eq:1006-2}, and \eqref{eq:1006-3}) where the resulting $\mathcal{G}_{-V}$ does not contain any autocatalytic reactions. Consequently, there were no critical autocatalytic reactions or critical strongly connected components and therefore Theorem \ref{thm:1} may be applied. It follows that $F$ is a nonnegative nonzero matrix, $V$ is a $Z$-matrix with nonnegative inverse, the next-generation matrix $FV^{-1}$ is well-defined and nonnegative, and $\rho(FV^{-1})$ can be computed.

For the autocatalytic splitting \eqref{eq:1006-1}, we have the following reproduction number:
\begin{equation}
\label{brn1}
\rho(FV^{-1}) = \displaystyle\frac{\partial_3 R_6 \partial_2 (R_3 + R_4)}{\partial_2 R_3\partial_3 (R_5 + 2R_6)}.
\end{equation}
We interpret \eqref{brn1} in the following way. For \eqref{eq:1006-1} we have split the $R_6$ into three components: the autocatalytic instance $R_6 \in \phi_F(e_{33})$ in $\mathcal{G}_F$ and the catalytic instances $R_6 \in \phi_{-V}(e_{31})$ and $R_6 \in \phi_{-V}(e_{32})$ in $\mathcal{G}_{-V}$. The autocatalytic instance $R_6 \in \phi_F(e_{33})$ produces $X_3$ at the rate given by the $(3,3)$-entry of $F$ while the $(3,3)$-entry of $V^{-1}$ corresponds to the expected cumulative time spent in state $X_3$ before transitioning to $0$. Consequently, the $(3,3)$-entry of $FV^{-1}$---which corresponds to $\rho(FV^{-1})$ \eqref{brn1} since $FV^{-1}$ has a single nonzero row---gives the expected number of $X_3$ produced per $X_3$ before exiting the system. If this number is greater than one (i.e., $\rho(FV^{-1}) > 1$), then more $X_3$ are produced than depleted for each new $X_3$ and hence the steady state with $x_3 = 0$ is unstable. Consequently, in this splitting we may interpret $F$ as the ``force of autocatalysis'' for $X_3$ and $V^{-1}$ as the ``dwell time'' for autocatalytically-created $X_3$.

In contrast, for the catalytic splittings \eqref{eq:1006-2} and \eqref{eq:1006-3}, we have the following reproduction number:
\begin{equation}
\label{brn2}
\rho(FV^{-1}) = \displaystyle\frac{\partial_2 R_4 \partial_3 R_6}{\partial_2 R_3 \partial_3 (R_5 + R_6)}.
\end{equation}
For \eqref{eq:1006-2}, we interpret \eqref{brn2} in the following way. We have split the $R_6$ into three components: the autocatalytic instance $R_6 \in \phi_F(e_{33})$ and catalytic instance $R_6 \in \phi_{F}(e_{32})$ in $\mathcal{G}_F$ and the catalytic instance $R_6 \in \phi_{-V}(e_{31})$ in $\mathcal{G}_{-V}$. The combined effect of the autocatalytic instance $R_6 \in \phi_F(e_{33})$ and catalytic instance $R_6 \in \phi_{F}(e_{32})$ is to produce $X_2$ from $X_3$ at the rate given by the $(2,3)$-entry of $F$. Note that the $(3,3)$-entry of $F$ is zero. In order to produce more $X_3$, we must transition back to $X_3$ which is directly possible through the instance $R_4 \in \phi_{-V}(e_{23})$. The $(3,2)$-entry of $V^{-1}$ corresponds to the expected cumulative time spent in state $X_3$ when starting in state $X_2$ before transitioning to $0$. The $(2,2)$-entry of $FV^{-1}$---which corresponds to $\rho(FV^{-1})$ \eqref{brn2} since $FV^{-1}$ has a single nonzero row---gives the expected number of $X_2$ produced per $X_2$ before it exits the system. If this number is greater than one (i.e., $\rho(FV^{-1}) > 1$), then more $X_2$ are produced than depleted for each new $X_2$ and hence the steady state with $x_2 = 0$ is unstable. Consequently, in this splitting we may interpret $F$ as the ``force of catalysis'' for $X_2$ from $X_3$ and $V^{-1}$ as the ``dwell time'' for catalytically-created $X_2$. A similar interpretation follows for \eqref{eq:1006-3} but with $X_1$ replacing $X_2$ throughout.

We note that different splittings may produce different reproduction numbers; however, $\rho(FV^{-1})=1$ corresponds to the same bifurcation set in each case. Specifically, \eqref{brn1} and \eqref{brn2} both yield the bifurcation threshold
\[\partial_2 R_4 \partial_3 R_6 = \partial_2 R_3 \partial_3 (R_5 + R_6).\]
The difference in the value of the reproduction numbers \eqref{brn1} and \eqref{brn2} results from whether we interpret the catalytic reaction $X_3 \to X_1 + X_2$ as creating an additional $X_3$, $X_2$, or $X_1$ through the three possible splittings \eqref{eq:1006-1}, \eqref{eq:1006-2}, or \eqref{eq:1006-3}, respectively.

\end{example}

\subsection{Algorithm}
\label{sec:algorithm}

Theorem \ref{thm:1} and Corollary \ref{cor:2} suggest a path for constructing valid and rank-reducing next-generation splittings $J = F - V$. After constructing the ALT-graph $\mathcal{G}$ of the chemical reaction network $(\mathcal{S},\mathcal{R})$, we need to iteratively perform autocatalytic and catalytic splittings $\mathcal{G}$ until we have an autocatalytic decomposition $\mathcal{G} = \mathcal{G}_{F} \oplus \mathcal{G}_{-V}$ \eqref{eq:decomposition} where $\mathcal{G}_{-V}$ has no critical autocatalytic reactions or critical strongly connected components. If successful, we have that $L(\mathcal{G}_F) = F$ is nonnegative and nonzero while $V = - L(\mathcal{G}_{-V})$ satisfies the conditions of Theorem \ref{thm:ngm2}.

Corollary \ref{cor:2} gives an upper bound on the rank of the next-generation matrix $FV^{-1}$. In practice, however, these bounds are not necessarily achieved for a careful selection of autocatalytic and catalytic instances for removal. We present the following heuristics---which we apply in the following order---for selecting instances to remove from $\mathcal{G}_{-V}$ and add to $\mathcal{G}_F$.
\begin{enumerate}
\item \textbf{Common target vertex preference:}
\begin{enumerate}
    \item Prioritize autocatalytic and catalytic splittings which remove instances from $\mathcal{G}_{-V}$ whose target vertex already corresponds to the target vertex of an edge in $\mathcal{G}_F$, unless doing so would disconnect the source vertex from $0$.
    \item If multiple autocatalytic reactions each have an instance leading to the same target vertex in $\mathcal{G}_{-V}$, remove one of these instances, prioritizing catalytic splittings first, unless doing so would disconnect a source vertex from $0$.
\end{enumerate}
\item \textbf{Catalytic preference:} When resolving a critical autocatalytic reaction from $\mathcal{G}_{-V}$ that has both an autocatalytic splitting and catalytic splitting available, perform the catalytic splitting unless doing so would cut off the source vertex from $0$.
\end{enumerate}

\noindent Intuitively, selecting instances for removal by \emph{common target vertex preference} keeps the rank of $F$ low by selecting terms which contribute to a common row of $F$ (i.e., the row $j$ corresponding to the common target vertex $X_j$). Selecting instances for removal by \emph{catalytic preference} uses the observation that catalytic splittings may change the connectivity of the graph while autocatalytic splittings leave connectivity properties unchanged. Catalytic splittings consequently help keep the rank of $F$ low by increasing the potential to disconnect the strongly connected components, which makes it less likely to have critical strongly connected components (examples are contained in Appendix \ref{app:g}). After applying heuristics 1. and 2., if there are remaining autocatalytic reactions associated with $\mathcal{R}_C$ and $\mathcal{Z}_C$, there is no preference in which are chosen for the next splitting. The process is outlined in Algorithm \ref{alg:1}.

\begin{algorithm} \small
\begin{algorithmic}[1] 
\Require Chemical reaction network $(\mathcal{S}, \mathcal{R})$ with boundary steady state $c^*$
\State Identify a siphon $\mathcal{X} \subseteq \mathcal{S}$ corresponding to $c^*$ (Definition \ref{def:siphon})
\State Construct ALT-graph $\mathcal{G} = (\mathcal{V}, \mathcal{E}, \Omega, \phi)$ of $(\mathcal{S},\mathcal{R})$ on siphon $\mathcal{X}$ (Definition \ref{def:alt-graph})
\If{there is a vertex with no path in $\mathcal{G}$ to $0$}
    \State Stability of $c^*$ is inconclusive
\Else{}
    \State Initialize $\mathcal{G}_{-V} = (\mathcal{V}, \mathcal{E}_{-V}, \Omega, \phi_{-V})$ with $\mathcal{E}_{-V} = \mathcal{E}$, $\phi_{-V}=\phi$ and $\mathcal{G}_F = (\mathcal{V}, \mathcal{E}_F, \Omega, \phi_F)$ with $\mathcal{E}_F = \emptyset$, $\phi_F=\emptyset$
    \State Identify autocatalytic and critical sets $\mathcal{R}_A$, $\mathcal{R}_C$, and $\mathcal{Z}_C$ of $\mathcal{G}_{-V}$ (Definition \ref{def:autocatalytic})
    \If{$\mathcal{R}_C = \emptyset$ and $\mathcal{Z}_C = \emptyset$}
        \State $c^*$ is universally locally asymptotically stable if boundary subsystem is stable (Corollary \ref{cor:2}, part 1.)
    \Else
        \While{$\mathcal{R}_C \not= \emptyset$ \textbf{or} $\mathcal{Z}_C \not= \emptyset$}
        \renewcommand{\algorithmicthen}{}
\If{\parbox[t]{\dimexpr\linewidth-5em}{%
    there are instances $R_k \in \phi_{-V}(e_{i\ell}) \cap \mathcal{R}_C$, $i \neq \ell$, and $R_{\tilde{k}} \in \phi_F(e_{j\ell})$, $R_{\tilde k} \not= R_k$, \textbf{and} \\
    \hspace*{-1.0em}removal of $R_k \in \phi_{-V}(e_{i\ell})$ from $\mathcal{G}_{-V}$ does not disconnect $X_i$ from $0$ \textbf{then}}}
    \State Remove instances $R_k \in \phi_{-V}(e_{ii})$ and $R_k\in\phi_{-V}(e_{i\ell})$ from $\mathcal{G}_{-V}$, add to $\mathcal{G}_F$
            \ElsIf{there are instances $R_k \in \phi_{-V}(e_{ii}) \cap \mathcal{R}_C$ and $R_{\tilde{k}} \in \phi_F(e_{ji})$, $R_{\tilde k} \not= R_k$, \textbf{then}}
                \State Remove instance $R_k \in \phi_{-V}(e_{ii})$ from $\mathcal{G}_{-V}$, add to $\mathcal{G}_F$
            \ElsIf{\parbox[t]{\dimexpr\linewidth-5em}{there are instances $R_k \in \phi_{-V}(e_{i\ell}) \cap \mathcal{R}_C$, $i \not= \ell,$ and $R_{\tilde{k}} \in \phi_{-V}(e_{j\ell}) \cap \mathcal{R}_C$, $R_{\tilde k} \not= R_k$, \textbf{and} \\
            \hspace*{-3.0em}removal of $R_k \in \phi_{-V}(e_{i\ell})$ from $\mathcal{G}_{-V}$ does not disconnect $X_i$ from $0$ \textbf{then}}}
                    \State Remove instances $R_k \in \phi_{-V}(e_{ii})$ and $R_k \in \phi_{-V}(e_{i\ell})$ from $\mathcal{G}_{-V}$, add to $\mathcal{G}_F$
            \ElsIf{there are instances $R_k \in \phi_{-V}(e_{ii}) \cap \mathcal{R}_C$ and $R_{\tilde{k}} \in \phi_{-V}(e_{ii}) \cap \mathcal{R}_C$, $R_{\tilde k} \not= R_k$,}
                \State Remove instance $R_k \in \phi_{-V}(e_{ii})$ from $\mathcal{G}_{-V}$, add to $\mathcal{G}_F$

            \ElsIf{\parbox[t]{\dimexpr\linewidth-5em}{$\mathcal{R}_C \not= \emptyset$}

                \If{\parbox[t]{\dimexpr\linewidth-5em}{there is instance $R_k\in\phi_{-V}(e_{i\ell})\cap\mathcal{R}_C$, $i\neq\ell$ \textbf{and} \\
            \hspace*{-1.0em}removal of $R_k\in\phi_{-V}(e_{i\ell})$ from $\mathcal{G}_{-V}$ does not disconnect $X_i$ from $0$ \textbf{then}}}
                    \State Remove instances $R_k \in \phi_{-V}(e_{ii})$ and $R_k \in \phi_{-V}(e_{i\ell})$ from $\mathcal{G}_{-V}$, add to $\mathcal{G}_F$}
                \Else{}
                \State Select any $R_k\in\mathcal{R}_C$ and remove instance $R_k \in \phi_{-V}(e_{ii})\cap\mathcal{R}_C$ from $\mathcal{G}_{-V}$, add to $\mathcal{G}_F$
                \EndIf
            \ElsIf{$\mathcal{Z}_C \not= \emptyset$}
                \State Select $\mathcal{Z}_h \in \mathcal{Z}_C$ and $R_k \in \mathcal{R}_A$ with $R_k \in\phi_{-V}(e_{ii})\cap \phi_{-V}(e_{ij})$, $X_i \in \mathcal{Z}_h$, $X_j \not\in \mathcal{Z}_h$
                \State Remove instance $R_k \in \phi_{-V}(e_{ii})$ from $\mathcal{G}_{-V}$, add to $\mathcal{G}_F$
            \EndIf
            \State Recompute $\mathcal{R}_A$, $\mathcal{R}_C$ and $\mathcal{Z}_C$ for $\mathcal{G}_{-V}$
        \EndWhile
        \State Compute ALT-matrices $F = L(\mathcal{G}_F)$ and $-V = L(\mathcal{G}_{-V})$
        \State Compute next-generation matrix $FV^{-1}$ and reproduction number $\rho(FV^{-1})$
        \State Establish thresholds of stability of $c^*$ according to Theorem \ref{thm:ngm}
    \EndIf
\EndIf
\end{algorithmic}
\caption{Procedure for assessing the stability of a boundary steady state of an integer-valued and binary chemical reaction network satisfying \textbf{(K1)}---\textbf{(K5)} via siphon-based ALT-graph decomposition. }
\label{alg:1}
\end{algorithm}

\section{Examples}
\label{sec:examples}

In this section, we present four examples of the ALT-graph approach to constructing the next-generation matrix $FV^{-1}$ drawn from the mathematical biochemistry and epidemiology literature.

\subsection{Biochemical Examples}
\label{sec:biochemexamples}

In this section, we present two further examples of biochemical models for which the ALT-graph method introduced in Section \ref{sec:mainresult} succeeds in efficiently generating a valid Jacobian splitting $J = F - V$.

\begin{example}
\label{ex:envzompr}
Consider the following EnvZ-OmpR osmoregulation mechanism, which was introduced in \cite{shinar2010structural} and has been studied in the context of boundary stability via the next-generation matrix method in the papers \cite{avram2024advancing,johnston2026boundary}:
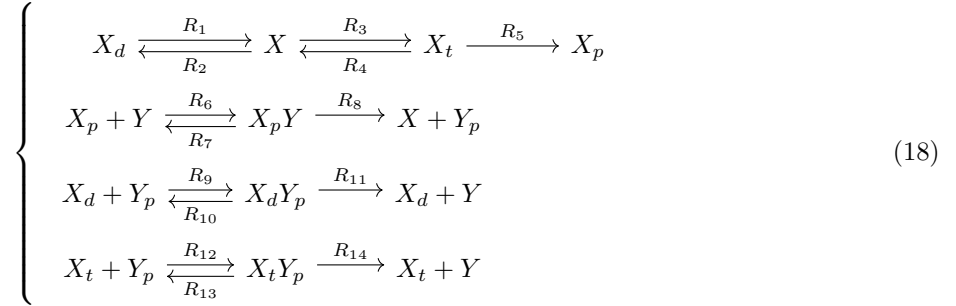
\begin{equation}
\label{envzompr}
\left\{ \; \; 
      \begin{tikzcd}
X_d \arrow[r,yshift=+0.5ex,"R_1"] & X \arrow[l,yshift=-0.5ex,"R_2"] \arrow[r,yshift=+0.5ex,"R_3"]& X_t \arrow[l,yshift=-0.5ex,"R_4"] \arrow[r,"R_5"]& X_p\\[-0.1in]
X_p + Y \arrow[r,yshift=+0.5ex,"R_6"] & X_pY \arrow[l,yshift=-0.5ex,"R_7"] \arrow[r,yshift=+0.5ex,"R_8"]& X + Y_p & \\[-0.1in]
X_d + Y_p \arrow[r,yshift=+0.5ex,"R_9"] & X_dY_p \arrow[l,yshift=-0.5ex,"R_{10}"] \arrow[r,yshift=+0.5ex,"R_{11}"]& X_d + Y &  \\[-0.1in]
X_t + Y_p \arrow[r,yshift=+0.5ex,"R_{12}"] & X_tY_p \arrow[l,yshift=-0.5ex,"R_{13}"] \arrow[r,yshift=+0.5ex,"R_{14}"]& X_t + Y &
      \end{tikzcd} \right.
\end{equation}
The network \eqref{envzompr} has the siphon $\mathcal{X} = \{ X_d, X, X_t, Y, X_pY, X_dY_p, X_tY_p \}$. The ALT-graph corresponding to $\mathcal{X}$ is given in Figure \ref{fig:ALT-envzompr}, and the assignment of the individual transitions is contained in Table \ref{tab:jacobian2} in Appendix \ref{app:h}. The autocatalytic reaction set is $\mathcal{R}_A = \{ R_{11}, R_{14}\}$ and the strongly connected components are given by $\mathcal{Z}_1 = \{ X_d, X, X_t, Y, X_pY, X_dY_p, X_tY_p \}$ and $\mathcal{Z}_2 = \{ 0 \}$. Since all catalytic instances of $R_{11}$ and $R_{14}$ lie within $\mathcal{Z}_1$, we have that all autocatalytic reactions are critical so that $\mathcal{R}_C = \{ R_{11}, R_{14} \}$. Since no instances of the autocatalytic reactions $\mathcal{R}_A = \{ R_{11}, R_{14} \}$ leave $\mathcal{Z}_1$, we have no critical strongly connected components so that $\mathcal{Z}_C = \emptyset$. Consequently, we need to perform autocatalytic/catalytic splittings according to Algorithm \ref{alg:1} until $\mathcal{R}_C = \emptyset$.

Applying Algorithm \ref{alg:1}, we initialize $\mathcal{G}_F$ with $\mathcal{E}_F = \emptyset$ and $\mathcal{G}_{-V} = \mathcal{G}$ with $\mathcal{E}_{-V} = \mathcal{E}$. We notice that the vertex $Y$ is the target of catalytic instances of $R_{11}$ and $R_{14}$. Applying lines 16--17 and then 12--13 (\emph{common target vertex} preference) of Algorithm \ref{alg:1}, we iterative remove these instances from $\mathcal{G}_{-V}$ and place them in $\mathcal{G}_F$. This yields the following autocatalytic decomposition $\mathcal{G} = \mathcal{G}_F \oplus \mathcal{G}_{-V}$ given in Figure \ref{fig:ALT-envzompr}.

Note that $\mathcal{G}_{-V}$ does not have any critical autocatalytic reactions or critical strongly connected components so that, by Theorem \ref{thm:1}, $F$ is nonnegative and nonzero, $V$ is a $Z$-matrix with a nonnegative inverse, the next-generation matrix $FV^{-1}$ is valid, and $\rho(FV^{-1})$ can be computed. We note that the same splitting $J=F-V$ was ultimately utilized in \cite{johnston2026boundary} but that the matrices were determined by trial and error rather than an algorithmic approach. Further details on the splitting procedure and computation of the reproduction number are contained in Appendix \ref{app:h}.

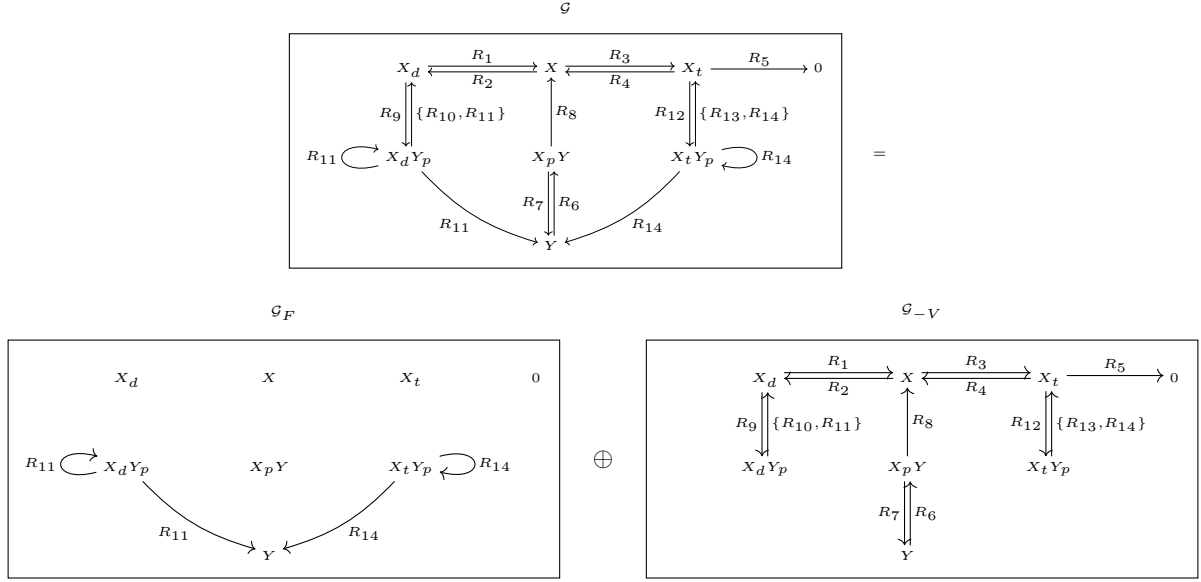
\begin{figure}[t]
\centering
\begin{minipage}{\textwidth}
\centering
{\tiny
\begin{tikzcd}[ampersand replacement=\&,execute at end picture={
        \draw[black, thin] 
            ([shift={(-5pt,-5pt)}]current bounding box.south west) 
            rectangle 
            ([shift={(5pt,5pt)}]current bounding box.north east);                  
            \node[above] at ([shift={(0,5pt)}]current bounding box.north) 
            {$\mathcal{G}$};
    }
]
X_d \arrow[rr,yshift=0.5ex,"R_1"] \arrow[dd,xshift=-0.5ex,"R_9"'] 
\& \& 
X \arrow[ll,yshift=-0.5ex,"R_2"] \arrow[rr,yshift=0.5ex,"R_3"] 
\& \& 
X_t \arrow[ll,yshift=-0.5ex,"R_4"] \arrow[rr,"R_5"] 
\arrow[dd,xshift=-0.5ex,"R_{12}"'] 
\& \& 0 \\
\\
X_dY_p \arrow[loop left,"R_{11}"] 
\arrow[uu,xshift=0.5ex,"{\{R_{10},R_{11}\}}"'] 
\arrow[ddrr,bend right=15,"R_{11}"'] 
\& \& 
X_pY \arrow[uu,"R_8"'] \arrow[dd,xshift=-0.5ex,"R_7"'] 
\& \& 
X_tY_p \arrow[loop right,"R_{14}"] 
\arrow[uu,xshift=0.5ex,"{\{R_{13},R_{14}\}}"'] 
\arrow[ddll,bend left=15,"R_{14}"] 
\& \& \\
\\
\& \& Y \arrow[uu,xshift=0.5ex,"R_{6}"'] \& \& \& \&
\end{tikzcd} 
\hspace{1em}
=
\hspace{1em}
}
\[
{\tiny
\begin{tikzcd}[ampersand replacement=\&,execute at end picture={
        \draw[black, thin] 
            ([shift={(-5pt,-5pt)}]current bounding box.south west) 
            rectangle 
            ([shift={(5pt,5pt)}]current bounding box.north east);                  
            \node[above] at ([shift={(0,5pt)}]current bounding box.north) 
            {$\mathcal{G}_F$};
    }
]
X_d \arrow[rr,yshift=0.5ex,white,"R_1"] \arrow[dd,white,xshift=-0.5ex,"R_9"'] 
\& \& 
X \arrow[ll,yshift=-0.5ex,white,"R_2"] \arrow[rr,yshift=0.5ex,white,"R_3"] 
\& \& 
X_t \arrow[ll,yshift=-0.5ex,white,"R_4"] \arrow[rr,white,"R_5"] 
\arrow[dd,xshift=-0.5ex,white,"R_{12}"'] 
\& \& 0 \\
\\
X_dY_p \arrow[loop left,"R_{11}"] 
\arrow[uu,xshift=0.5ex,white,"{\{R_{10},R_{11}\}}"'] 
\arrow[ddrr,bend right=15,"R_{11}"'] 
\& \& 
X_pY \arrow[uu,white,"R_8"'] \arrow[dd,xshift=-0.5ex,white,"R_7"'] 
\& \& 
X_tY_p \arrow[loop right,"R_{14}"] 
\arrow[uu,xshift=0.5ex,white,"{\{R_{13},R_{14}\}}"'] 
\arrow[ddll,bend left=15,"R_{14}"] 
\& \& \\
\\
\& \& Y \arrow[uu,xshift=0.5ex,white,"R_{6}"'] \& \& \& \&
\end{tikzcd}
}
\hspace{1em}
\oplus
\hspace{1em}
{\tiny
\begin{tikzcd}[ampersand replacement=\&,execute at end picture={
        \draw[black, thin] 
            ([shift={(-5pt,-5pt)}]current bounding box.south west) 
            rectangle 
            ([shift={(5pt,5pt)}]current bounding box.north east);                  
            \node[above] at ([shift={(0,5pt)}]current bounding box.north) 
            {$\mathcal{G}_{-V}$};
    }
]
X_d \arrow[rr,yshift=0.5ex,"R_1"] \arrow[dd,xshift=-0.5ex,"R_9"'] 
\& \& 
X \arrow[ll,yshift=-0.5ex,"R_2"] \arrow[rr,yshift=0.5ex,"R_3"] 
\& \& 
X_t \arrow[ll,yshift=-0.5ex,"R_4"] \arrow[rr,"R_5"] 
\arrow[dd,xshift=-0.5ex,"R_{12}"'] 
\& \& 0 \\
\\
X_dY_p \arrow[loop left,white,"R_{11}"] 
\arrow[uu,xshift=0.5ex,"{\{R_{10},R_{11}\}}"'] 
\arrow[ddrr,bend right=15,white,"R_{11}"'] 
\& \& 
X_pY \arrow[uu,"R_8"'] \arrow[dd,xshift=-0.5ex,"R_7"'] 
\& \& 
X_tY_p \arrow[loop right,white,"R_{14}"] 
\arrow[uu,xshift=0.5ex,"{\{R_{13},R_{14}\}}"'] 
\arrow[ddll,bend left=15,white,"R_{14}"] 
\& \& \\
\\
\& \& Y \arrow[uu,xshift=0.5ex,"R_{6}"'] \& \& \& \&
\end{tikzcd}
}
\]
\end{minipage}
\caption{The ALT-graph $\mathcal{G}$ of the EnvZ-OmpR network \eqref{envzompr} and an associated autocatalytic decomposition $\mathcal{G}_F\oplus\mathcal{G}_{-V}$.  $\mathcal{G}_F$ contains two instances each of the autocatalytic reactions $R_{11}$ and $R_{14}$ and their associated transitions, while $\mathcal{G}_{-V}$ contains the remaining transitions.}
\label{fig:ALT-envzompr}
\end{figure}

\end{example}

\begin{example}
\label{ex:twosite}
Consider the two-site bifunctional enzyme mechanism
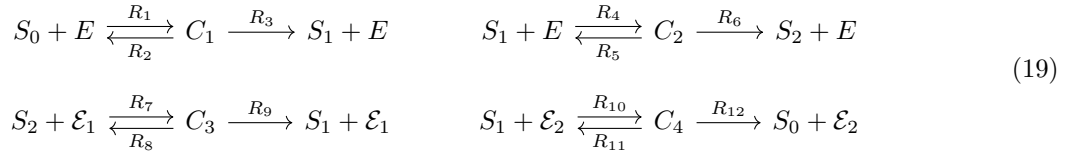
\begin{equation}
\label{eq:twosite2}
\begin{tikzcd}[ampersand replacement=\&]
            S_0 + E \arrow[r,yshift=0.5ex,"R_1"] \& C_1 \arrow[l,yshift=-0.5ex,"R_2"] \arrow[r,"R_3"] \& S_1 + E \&  S_1 + E \arrow[r,yshift=0.5ex,"R_4"] \& C_2 \arrow[l,yshift=-0.5ex,"R_5"] \arrow[r,"R_6"] \& S_2 + E \\
            S_2 + \mathcal{E}_1 \arrow[r,yshift=0.5ex,"R_7"] \& C_3 \arrow[l,yshift=-0.5ex,"R_8"] \arrow[r,"R_9"] \& S_1 + \mathcal{E}_1 \& S_1 + \mathcal{E}_2 \arrow[r,yshift=0.5ex,"R_{10}"] \& C_4 \arrow[l,yshift=-0.5ex,"R_{11}"] \arrow[r,"R_{12}"] \& S_0 + \mathcal{E}_2
            \end{tikzcd}
\end{equation}
where $\mathcal{E}_1, \mathcal{E}_2 \in \{ C_1, C_2 \}$. The intermediates from the forward chain act as the enzymes for the backward chain. The mechanism was previously studied with a focus on the capacity for absolute concentration robustness in \cite{joshi2024bifunctional,joshi2026bifunctional} and with a focus on boundary steady state stability and multistationarity in \cite{joshi2026bistability}.

In each of the four cases ($(\mathcal{E}_1, \mathcal{E}_2) = (C_1, C_1)$, $(\mathcal{E}_1, \mathcal{E}_2) = (C_1, C_2)$, $(\mathcal{E}_1, \mathcal{E}_2) = (C_2, C_1)$, $(\mathcal{E}_1, \mathcal{E}_2) = (C_2, C_2)$), the network \eqref{eq:twosite2} has the siphon $\mathcal{X} = \{S_0, S_1, C_1, C_2, C_3, C_4\}$. The ALT-graphs corresponding to $\mathcal{X}$ are given in Figure \ref{fig:alt-four-environments} and the mapping of the individual transitions is contained in Table \ref{tab:jacobian3} in Appendix \ref{app:h}. For all four mechanisms, we have that the autocatalytic reaction set is $\mathcal{R}_A = \{ R_9, R_{11}, R_{12}\}$. It remains to determine the critical autocatalytic reaction set $\mathcal{R}_C$ and the critical strongly connected component set $\mathcal{Z}_C$. We note that the strongly connected components of the ALT-graph vary by mechanism, with $(\mathcal{E}_1, \mathcal{E}_2) = (C_1, C_1)$ and $(\mathcal{E}_1, \mathcal{E}_2) = (C_1, C_2)$ admitting the strongly connected components $\mathcal{Z}_1 = \{ S_0, C_1, C_3\}$,  $\mathcal{Z}_2 = \{ S_1, C_2\}$, $\mathcal{Z}_3 = \{ C_4 \}$, and $\mathcal{Z}_4 = \{ 0 \}$, while the mechanisms with $(\mathcal{E}_1, \mathcal{E}_2) = (C_2, C_1)$ and $(\mathcal{E}_1, \mathcal{E}_2) = (C_2, C_2)$ admit the strongly connected components $\mathcal{Z}_1 = \{ S_0, C_1\}$,  $\mathcal{Z}_2 = \{ S_1, C_2, C_3\}$, $\mathcal{Z}_3 = \{ C_4 \}$, and $\mathcal{Z}_4 = \{ 0 \}$. We have the following analysis.

\begin{figure}[t]
\centering
{\tiny
\begin{tikzcd}[ampersand replacement=\&,execute at end picture={
        \draw[black, thin] 
            ([shift={(-5pt,-5pt)}]current bounding box.south west) 
            rectangle 
            ([shift={(5pt,5pt)}]current bounding box.north east);                  
            \node[above] at ([shift={(0,5pt)}]current bounding box.north) 
            {$(\mathcal{E}_1, \mathcal{E}_2) = (C_1, C_1)$};
    }
]
\& \& \& \& 
C_4 \arrow[loop above, "{\{R_{11},R_{12}\}}"] 
\arrow[lldd,"{\{R_{11},R_{12}\}}"'] 
\arrow[lllldd,bend right=30,"R_{12}"'] 
\arrow[dd,"R_{11}"] 
\& \& \& \& \\ \\
S_0 \arrow[rr,yshift=0.5ex,"R_1"] 
\& \& 
C_1 \arrow[ll,yshift=-0.5ex,"R_2"] 
\arrow[rr,"R_3"] 
\arrow[rrdd,xshift=0.3ex,yshift=0.3ex,"R_7"] 
\& \& 
S_1 \arrow[rr,yshift=0.5ex,"R_4"] 
\& \& 
C_2 \arrow[ll,yshift=-0.5ex,"R_5"] 
\arrow[rr,"R_6"] 
\& \& 0 \\ \\
\& \& \& \& 
C_3 \arrow[loop below,"R_9"] 
\arrow[lluu,xshift=-0.3ex,yshift=-0.3ex,"{\{R_8,R_9\}}"] 
\arrow[uu,"R_9"'] 
\& \& \& \&
\end{tikzcd}
}
\hspace{1em}
{\tiny
\begin{tikzcd}[ampersand replacement=\&,execute at end picture={
        \draw[black, thin] 
            ([shift={(-5pt,-5pt)}]current bounding box.south west) 
            rectangle 
            ([shift={(5pt,5pt)}]current bounding box.north east);                  
            \node[above] at ([shift={(0,5pt)}]current bounding box.north) 
            {$(\mathcal{E}_1, \mathcal{E}_2) = (C_1, C_2)$};
    }
]
\& \& \& \& 
C_4 \arrow[loop above,"{\{R_{11},R_{12}\}}"] 
\arrow[rrdd,"{\{R_{11},R_{12}\}}"] 
\arrow[lllldd,bend right=30,"R_{12}"'] 
\arrow[dd,"R_{11}"] 
\& \& \& \& \\ \\
S_0 \arrow[rr,yshift=0.5ex,"R_1"] 
\& \& 
C_1 \arrow[ll,yshift=-0.5ex,"R_2"] 
\arrow[rr,"R_3"] 
\arrow[rrdd,xshift=0.3ex,yshift=0.3ex,"R_7"] 
\& \& 
S_1 \arrow[rr,yshift=0.5ex,"R_4"] 
\& \& 
C_2 \arrow[ll,yshift=-0.5ex,"R_5"] 
\arrow[rr,"R_6"] 
\& \& 0 \\ \\
\& \& \& \& 
C_3 \arrow[loop below,"R_9"] 
\arrow[lluu,xshift=-0.3ex,yshift=-0.3ex,"{\{R_8,R_9\}}"] 
\arrow[uu,"R_9"'] 
\& \& \& \&
\end{tikzcd}
}
\vspace{1em}
{\tiny
\begin{tikzcd}[ampersand replacement=\&,execute at end picture={
        \draw[black, thin] 
            ([shift={(-5pt,-5pt)}]current bounding box.south west) 
            rectangle 
            ([shift={(5pt,5pt)}]current bounding box.north east);                  
            \node[above] at ([shift={(0,5pt)}]current bounding box.north) 
            {$(\mathcal{E}_1, \mathcal{E}_2) = (C_2, C_1)$};
    }
]
\& \& \& \& 
C_4 \arrow[loop above,"{\{R_{11},R_{12}\}}"] 
\arrow[lldd,"{\{R_{11},R_{12}\}}"'] 
\arrow[lllldd,bend right=30,"R_{12}"'] 
\arrow[dd,"R_{11}"] 
\& \& \& \& \\ \\
S_0 \arrow[rr,yshift=0.5ex,"R_1"] 
\& \& 
C_1 \arrow[ll,yshift=-0.5ex,"R_2"] 
\arrow[rr,"R_3"] 
\& \& 
S_1 \arrow[rr,yshift=0.5ex,"R_4"] 
\& \& 
C_2 \arrow[ll,yshift=-0.5ex,"R_5"] 
\arrow[lldd,xshift=-0.3ex,yshift=0.3ex,"R_7"'] 
\arrow[rr,"R_6"] 
\& \& 0 \\ \\
\& \& \& \& 
C_3 \arrow[loop below,"R_9"] 
\arrow[rruu,xshift=0.3ex,yshift=-0.3ex,"{\{R_8,R_9\}}"'] 
\arrow[uu,"R_9"'] 
\& \& \& \&
\end{tikzcd}
}
\hspace{1em}
{\tiny
\begin{tikzcd}[ampersand replacement=\&,execute at end picture={
        \draw[black, thin] 
            ([shift={(-5pt,-5pt)}]current bounding box.south west) 
            rectangle 
            ([shift={(5pt,5pt)}]current bounding box.north east);                  
            \node[above] at ([shift={(0,5pt)}]current bounding box.north) 
            {$(\mathcal{E}_1, \mathcal{E}_2) = (C_2, C_2)$};
    }
]
\& \& \& \& 
C_4 \arrow[loop above,"{\{R_{11},R_{12}\}}"] 
\arrow[rrdd,"{\{R_{11},R_{12}\}}"] 
\arrow[lllldd,bend right=30,"R_{12}"'] 
\arrow[dd,"R_{11}"] 
\& \& \& \& \\ \\
S_0 \arrow[rr,yshift=0.5ex,"R_1"] 
\& \& 
C_1 \arrow[ll,yshift=-0.5ex,"R_2"] 
\arrow[rr,"R_3"] 
\& \& 
S_1 \arrow[rr,yshift=0.5ex,"R_4"] 
\& \& 
C_2 \arrow[ll,yshift=-0.5ex,"R_5"] 
\arrow[lldd,xshift=-0.3ex,yshift=0.3ex,"R_7"'] 
\arrow[rr,"R_6"] 
\& \& 0 \\ \\
\& \& \& \& 
C_3 \arrow[loop below,"R_9"] 
\arrow[rruu,xshift=0.3ex,yshift=-0.3ex,"{\{R_8,R_9\}}"'] 
\arrow[uu,"R_9"'] 
\& \& \& \&
\end{tikzcd}
}
\caption{The ALT-graphs for the dual-site bifunctional enzyme network \eqref{eq:twosite2} associated with the four possible choices of species $\mathcal{E}_1$ and $\mathcal{E}_2$: $(\mathcal{E}_1,\mathcal{E}_2)=(C_1,C_1)$, $(C_1,C_2)$, $(C_2,C_1)$, and $(C_2,C_2)$.}
\label{fig:alt-four-environments}
\end{figure}
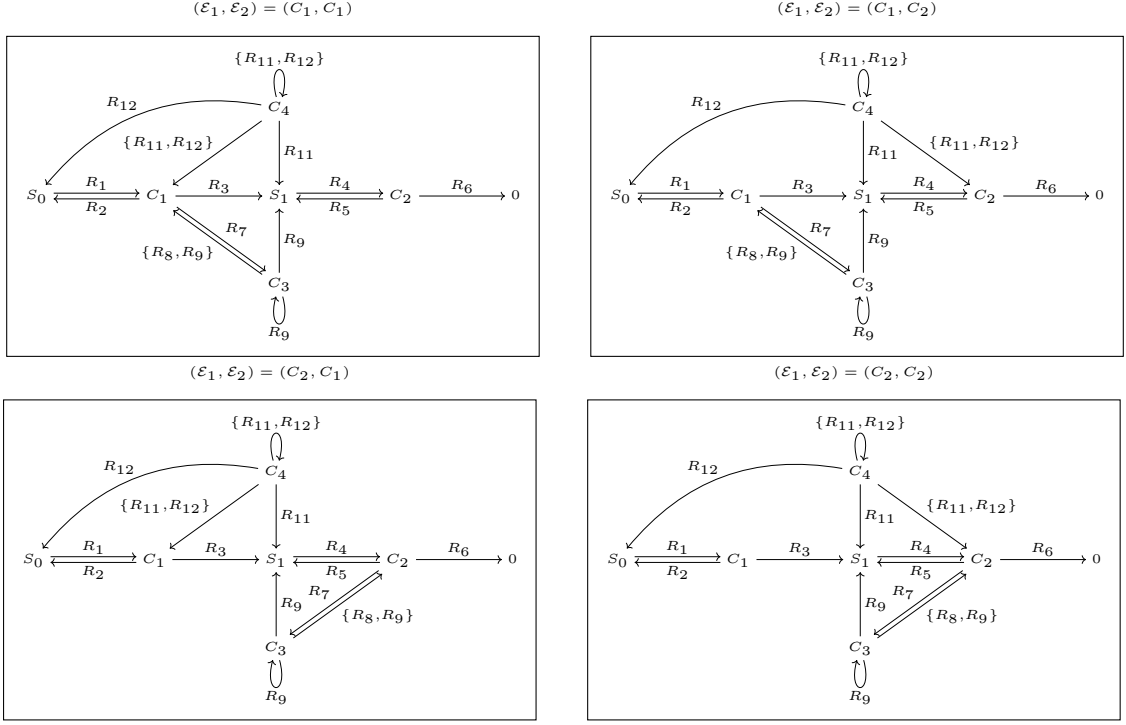

\paragraph{Mechanisms with $(\mathcal{E}_1, \mathcal{E}_2) = (C_1, C_1)$ and $(\mathcal{E}_1, \mathcal{E}_2) = (C_1, C_2)$:} We start by considering critical autocatalytic reactions. The reaction $R_9 \in \mathcal{R}_A$ has exactly one catalytic instance which leaves the strongly connected component $\mathcal{Z}_1$. Similarly, all catalytic instances of the autocatalytic reactions $R_{11}, R_{12} \in \mathcal{R}_A$ leave the strongly connected component $\mathcal{Z}_3$. Consequently, no autocatalytic reactions are critical so that $\mathcal{R}_C = \emptyset$. 

Now consider critical strongly connected components. Consider $\mathcal{Z}_1 = \{ S_0, C_1, C_3\}$. There is a catalytic instance of $R_9 \in \mathcal{R}_A$ which leaves $\mathcal{Z}_1$; however, $\mathcal{Z}_1$ may not be disconnected from $0$ by the removal of this instance because $R_3$ connects $\mathcal{Z}_1$ to $\mathcal{Z}_2$ and $R_3$ is not autocatalytic. Consequently, $\mathcal{Z}_1$ is not a critical strongly connected component in these mechanisms. Now consider $\mathcal{Z}_3 = \{ C_4 \}$. In all four mechanisms, two catalytic instances of $R_{11} \in \mathcal{R}_A$ and two catalytic instances of $R_{12} \in \mathcal{R}_A$ leave $\mathcal{Z}_3$. Despite this, there is no way to disconnect $\mathcal{Z}_3$ from $0$ by the removal of at most one catalytic instance per autocatalytic reaction ($R_{11}$ or $R_{12}$) since the remaining instances connect $\mathcal{Z}_3$ to $0$. Consequently, $\mathcal{Z}_3$ is not critical in any of the four mechanisms. It follows that $\mathcal{Z}_C = \emptyset$.

Since $\mathcal{R}_C = \emptyset$ and $\mathcal{Z}_C = \emptyset$, it follows from Corollary \ref{cor:2} part 1. that the boundary steady state is universally locally asymptotically stable. This reaffirms the result established in \cite{joshi2026bistability}.

\paragraph{Mechanisms with $(\mathcal{E}_1, \mathcal{E}_2) = (C_2, C_1)$ and $(\mathcal{E}_1, \mathcal{E}_2) = (C_2, C_2)$:} $R_9 \in \mathcal{R}_A$ has no catalytic instances leaving the strongly connected component $\mathcal{Z}_2$ while $R_{11}$ and $R_{12}$ have instances which leave $\mathcal{Z}_3$. It follows that $\mathcal{R}_C = \{ R_9\}$. Identical analysis from the previous case shows that $\mathcal{Z}_C = \emptyset$. 

We now apply Algorithm \ref{alg:1} to remove $R_9$ from $\mathcal{R}_C$. By Heuristic 2., we favor a catalytic splitting over an autocatalytic splitting. Performing the catalytic splitting on the instance of $R_9$ from $C_3$ to $S_1$ yields the splitting contained in Figure \ref{fig:alt-decomposition-GF-GV} (top) for the $(\mathcal{E}_1, \mathcal{E}_2) = (C_2, C_1)$ mechanism and the splitting contained in Figure \ref{fig:alt-decomposition-GF-GV} (bottom) for the $(\mathcal{E}_1, \mathcal{E}_2) = (C_2, C_2)$ mechanism. In both cases, $\mathcal{G}_{-V}$ does not have any critical autocatalytic reactions or critical strongly connected components. Consequently, the conclusions of Theorem \ref{thm:1} may be applied, and the next-generation matrix may be computed using the associated ALT-matrices $F$ and $-V$. Further computational details are provided in Appendix \ref{app:h}.

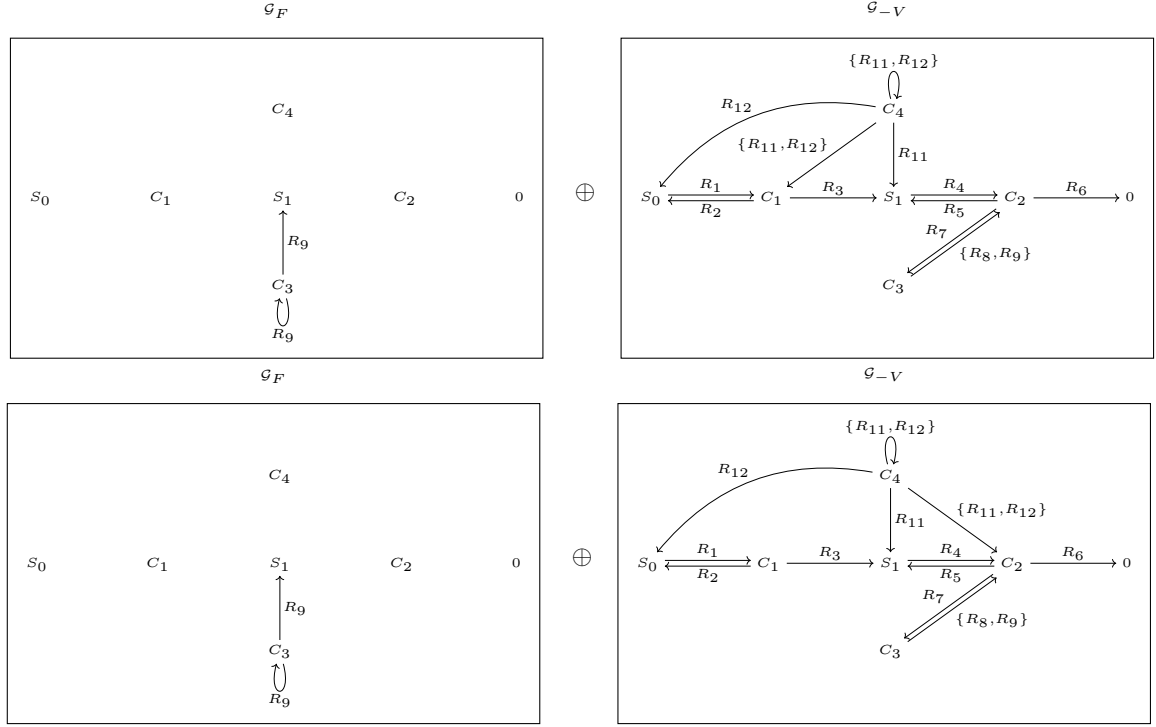
\begin{figure}[t]
\centering
{\tiny
\begin{tikzcd}[ampersand replacement=\&,execute at end picture={
        \draw[black, thin] 
            ([shift={(-5pt,-5pt)}]current bounding box.south west) 
            rectangle 
            ([shift={(5pt,5pt)}]current bounding box.north east);                  
            \node[above] at ([shift={(0,5pt)}]current bounding box.north) {$\mathcal{G}_F$};
    }
]
\& \& \& \& 
C_4 \arrow[loop above,white,"{\{R_{11},R_{12}\}}"] 
\& \& \& \& \\ \\
S_0 \& \& C_1 \& \& S_1 \& \& C_2 \& \& 0 \\ \\
\& \& \& \& 
C_3 \arrow[loop below,"R_9"] \arrow[uu,"R_9"'] 
\& \& \& \&
\end{tikzcd}
}
\;
$\oplus$
\;
{\tiny
\begin{tikzcd}[ampersand replacement=\&,execute at end picture={
        \draw[black, thin] 
            ([shift={(-5pt,-5pt)}]current bounding box.south west) 
            rectangle 
            ([shift={(5pt,5pt)}]current bounding box.north east);                  
            \node[above] at ([shift={(0,5pt)}]current bounding box.north) {$\mathcal{G}_{-V}$};
    }
]
\& \& \& \& 
C_4 \arrow[loop above,"{\{R_{11},R_{12}\}}"] 
\arrow[lldd,"{\{R_{11},R_{12}\}}"'] 
\arrow[lllldd,bend right=30,"R_{12}"'] 
\arrow[dd,"R_{11}"] 
\& \& \& \& \\ \\
S_0 \arrow[rr,yshift=0.5ex,"R_1"] 
\& \& 
C_1 \arrow[ll,yshift=-0.5ex,"R_2"] 
\arrow[rr,"R_3"] 
\& \& 
S_1 \arrow[rr,yshift=0.5ex,"R_4"] 
\& \& 
C_2 \arrow[ll,yshift=-0.5ex,"R_5"] 
\arrow[lldd,xshift=-0.3ex,yshift=0.3ex,"R_7"'] 
\arrow[rr,"R_6"] 
\& \& 0 \\ \\
\& \& \& \& 
C_3 \arrow[loop below,white,"R_9"] 
\arrow[rruu,xshift=0.3ex,yshift=-0.3ex,"{\{R_8,R_9\}}"'] 
\& \& \& \&
\end{tikzcd}
}\hspace{1em}
{\tiny
\begin{tikzcd}[ampersand replacement=\&,execute at end picture={
        \draw[black, thin] 
            ([shift={(-5pt,-5pt)}]current bounding box.south west) 
            rectangle 
            ([shift={(5pt,5pt)}]current bounding box.north east);                  
            \node[above] at ([shift={(0,5pt)}]current bounding box.north) {$\mathcal{G}_F$};
    }
]
\& \& \& \& 
C_4 \arrow[loop above,white,"{\{R_{11},R_{12}\}}"] 
\& \& \& \& \\ \\
S_0 \& \& C_1 \& \& S_1 \& \& C_2 \& \& 0 \\ \\
\& \& \& \& 
C_3 \arrow[loop below,"R_9"] \arrow[uu,"R_9"'] 
\& \& \& \&
\end{tikzcd}
}
\;
$\oplus$
\;
{\tiny
\begin{tikzcd}[ampersand replacement=\&,execute at end picture={
        \draw[black, thin] 
            ([shift={(-5pt,-5pt)}]current bounding box.south west) 
            rectangle 
            ([shift={(5pt,5pt)}]current bounding box.north east);                  
            \node[above] at ([shift={(0,5pt)}]current bounding box.north) {$\mathcal{G}_{-V}$};
    }
]
\& \& \& \& 
C_4 \arrow[loop above,"{\{R_{11},R_{12}\}}"] 
\arrow[rrdd,"{\{R_{11},R_{12}\}}"] 
\arrow[lllldd,bend right=30,"R_{12}"'] 
\arrow[dd,"R_{11}"] 
\& \& \& \& \\ \\
S_0 \arrow[rr,yshift=0.5ex,"R_1"] 
\& \& 
C_1 \arrow[ll,yshift=-0.5ex,"R_2"] 
\arrow[rr,"R_3"] 
\& \& 
S_1 \arrow[rr,yshift=0.5ex,"R_4"] 
\& \& 
C_2 \arrow[ll,yshift=-0.5ex,"R_5"] 
\arrow[lldd,xshift=-0.3ex,yshift=0.3ex,"R_7"'] 
\arrow[rr,"R_6"] 
\& \& 0 \\ \\
\& \& \& \& 
C_3 \arrow[loop below,white,"R_9"] 
\arrow[rruu,xshift=0.3ex,yshift=-0.3ex,"{\{R_8,R_9\}}"'] 
\& \& \& \&
\end{tikzcd}
}
\caption{Autocatalytic decompositions of the ALT-graphs into $\mathcal{G}_F\oplus\mathcal{G}_{-V}$ for $(\mathcal{E}_1, \mathcal{E}_2) = (C_2, C_1)$ (top) and $(\mathcal{E}_1, \mathcal{E}_2) = (C_2, C_2)$ (bottom). In each case, $\mathcal{G}_F$ contains the autocatalytic transitions associated with $R_9$, while $\mathcal{G}_{-V}$ contains the remaining transitions. The two decompositions correspond to the two possible orientations of the transitions associated with the environment-dependent autocatalytic sets $\{R_{11},R_{12}\}$ and $\{R_8,R_9\}$.}
\label{fig:alt-decomposition-GF-GV}
\end{figure}

\end{example}

\subsection{Epidemiological Examples}
\label{sec:epiexamples}

\begin{table}[t]
\centering
\begin{NiceTabular}{ccccccc}[code-before = \rowcolors{2}{white}{gray!12}]
\toprule
\begin{tabular}{c}Epidemiological \\ interaction\end{tabular} & & \begin{tabular}{c}Chemical \\ reaction \end{tabular} & & \begin{tabular}{c}Infection-reduced\\reaction\end{tabular} & & \begin{tabular}{c} ALT-graph\\edges\end{tabular} \\
\midrule
Direct Infection & & \begin{tikzcd}[ampersand replacement=\&, baseline=(current bounding box.center)] I_i+S_i \arrow[r,"\beta/N"] \& 2I_i \end{tikzcd} & $\Longrightarrow$ & \begin{tikzcd}[ampersand replacement=\&, baseline=(current bounding box.center)] I_i \arrow[r] \& 2I_i \end{tikzcd} & $\Longrightarrow$ & \begin{tikzcd}[ampersand replacement=\&, baseline=(current bounding box.center)] I_i \arrow[loop left,"\beta/N"]  \end{tikzcd} \\[0pt]
Recovery & & \begin{tikzcd}[ampersand replacement=\&, baseline=(current bounding box.center)] I_i \arrow[r,"\gamma"] \& S_i \end{tikzcd} & $\Longrightarrow$ & \begin{tikzcd}[ampersand replacement=\&, baseline=(current bounding box.center)] I_i \arrow[r] \& 0 \end{tikzcd} & $\Longrightarrow$ & \begin{tikzcd}[ampersand replacement=\&, baseline=(current bounding box.center)] I_i \arrow[r,"\gamma"] \& 0 \end{tikzcd} \\[0pt]
Transition & & \begin{tikzcd}[ampersand replacement=\&, baseline=(current bounding box.center)] I_i \arrow[r,"\lambda"] \& I_j \end{tikzcd} & $\Longrightarrow$ & \begin{tikzcd}[ampersand replacement=\&, baseline=(current bounding box.center)] I_i \arrow[r] \& I_j \end{tikzcd} & $\Longrightarrow$ & \begin{tikzcd}[ampersand replacement=\&, baseline=(current bounding box.center)] I_i \arrow[r,"\lambda"] \& I_j \end{tikzcd} \\[0pt]
Cross Infection & & \begin{tikzcd}[ampersand replacement=\&, baseline=(current bounding box.center)] I_i + S_j \arrow[r,"\beta/N"] \& I_i + I_j \end{tikzcd} & $\Longrightarrow$ & \begin{tikzcd}[ampersand replacement=\&, baseline=(current bounding box.center)] I_i \arrow[r] \& I_i + I_j \end{tikzcd} & $\Longrightarrow$ & \begin{tikzcd}[ampersand replacement=\&, baseline=(current bounding box.center)] I_i \arrow[loop left,"\beta/N"] \arrow[r,"\beta/N"] \& I_j \end{tikzcd} \\
\bottomrule
\end{NiceTabular}
\caption{Correspondences between common interaction types in epidemiology and edges in an ALT-graph. We generically denote the ``infected'' states by $I_i$ and the ``non-infected'' states by $S_i$ but note that other classifications which go to zero when the infection vanishes may be included in this set (e.g., hospitalized [$I_i = H$], temporarily immune [$I_i = R$]). All four types listed have direct correspondences in an ALT-graph (see Table \ref{tab:jacobian}). Specifically, we have: (1) Direct infection events correspond to autocatalytic edges in the ALT-graph. Direct infection events are typical in single population models where infected individuals turn susceptible individual into the same type of infected individual. (2) Recovery events correspond to leak edges in the ALT-graph. These events correspond to the final transition where an individual enters a state which could exist without the infection (e.g., the final step in a multi-stage loss of temporary immunity). (3) Transition events correspond to transition edges in the ALT-graph. (4) Cross infection events correspond to Catalytic Transitions (Type I), which consist of an autocatalytic edge and a transition edge with the same label. Cross infection events occur when an infected individual turns a susceptible individual into a different type of infected individual, which is common for vector-borne illnesses (e.g., $S_H+I_V \to I_H+I_V$, where $H$ is human and $V$ is vector) or stratified population models (e.g., $I_i + S_j \to I_i + I_j$ where $i$ and $j$ represent distinct strata of a population such age or geographical location). Notably, the remaining two classifications---Catalytic Transition (Type II) and Catalytic Transition (Type III)---do not occur naturally in mathematical epidemiology. These types of edges are unique to biochemical applications.}
\label{tab:epidemiology}
\end{table}

Transitions in epidemiological models can be represented as interactions between species in chemical reaction networks. Consequently, common compartmental epidemiological models can be considered a subclass of biochemical reaction networks. When the resulting reaction networks are integer-valued and binary, the developments we have made on the next-generation matrix method in the biochemical setting can be applied in the epidemiological setting. In particular, we notice the following correspondence between infection events, interpreted as chemical reactions, and ALT-graph edges:
\begin{enumerate}
\item \textbf{Direct Infection:} Consider a transmission event between a susceptible individual and an infectious individual which converts a susceptible individual into \emph{the same type} of infectious individual. This corresponds to the interaction
\[\begin{tikzcd}
S_i + I_i \arrow[r,"\beta/N"] & 2I_i
\end{tikzcd}\]
This type of interaction is common in homogeneous population settings. Such interactions contribute the following to the ALT-graph
\[
\begin{tikzcd}
I_i \arrow[loop left, "\beta/N"]
\end{tikzcd}
\]
\item \textbf{Cross Infection:} Consider a transmission event between a susceptible individual and an infectious individual which converts a susceptible individual into \emph{a different type} of infectious individual. This corresponds to the interaction
\[\begin{tikzcd}
S_j + I_i \arrow[r,"\beta/N"] & I_i + I_j
\end{tikzcd}\]
This type of interaction is common in nonhomogeneous population (e.g., regionally- or age-stratified) and vector-borne illness models. In these settings, the interaction may correspond to an infectious individual in region $i$ infecting an individual in region $j$, or an individual of one species (e.g., a mosquito) infecting an individual of another species (e.g., a human). Such interactions contribute the following to the ALT-graph
\[
\begin{tikzcd}
I_i \arrow[loop left, "\beta/N"] \arrow[r, "\beta/N"] & I_j
\end{tikzcd}
\]
\end{enumerate}
The full correspondence between typical transition types in mathematical epidemiology and ALT-graph edges, including leak and transition edges, is made in Table \ref{tab:epidemiology}. Note that, since infection events necessarily retain the original infected individual and produce exactly one new infected individual, two common biochemical reaction types---Catalytic Transition (Type II) and Catalytic Transition (Type III) in Table \ref{tab:jacobian}---do not occur naturally in mathematical epidemiology. These notions of autocatalysis are consequently unique to biochemical applications.

\begin{example}
\label{ex:sir}
Consider the standard SIR (Susceptible-Infectious-Removed) network with reinfection, births, and deaths. This is typically represented by a network diagram of the form
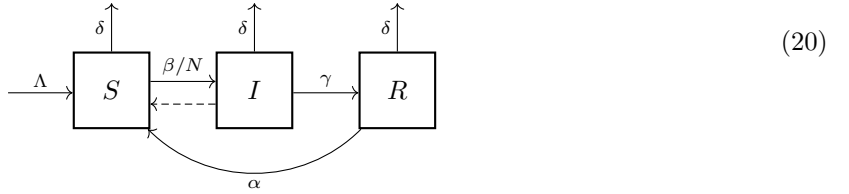
\begin{equation}
\label{sir}
\begin{tikzcd}
& {\color{white}o} & {\color{white}o} & {\color{white}o} \\
{\color{white}o} \arrow[r,"\Lambda"] & |[draw, thick, black, minimum size=1cm]| S \arrow[r,yshift=0.15cm,"\beta/N"] \arrow[u,"\delta"] & |[draw, thick, black, minimum size=1cm]| I \arrow[r,"\gamma"] \arrow[u,"\delta"] \arrow[l,yshift=-0.15cm,dashed] & |[draw, thick, black, minimum size=1cm]| R \arrow[u,"\delta"] \arrow[ll,bend left=45,"\alpha"] 
\end{tikzcd}
\end{equation}
where $\Lambda > 0$ is the birth rate, $\delta$ is the death rate, $\beta$ is the transmission rate, $1/\gamma$ is the mean infectious period, and $N$ is the population size.

Network representations like \eqref{sir} conveniently map the transitions of individuals through their varied infectious/noninfectious states; however, they obscure the interaction form of the transmission edge, which requires both $S$ and $I$ to proceed. To represent \eqref{sir} as an interaction network on the species set $\{S, I, R \}$, we reformulate the interactions as:\\
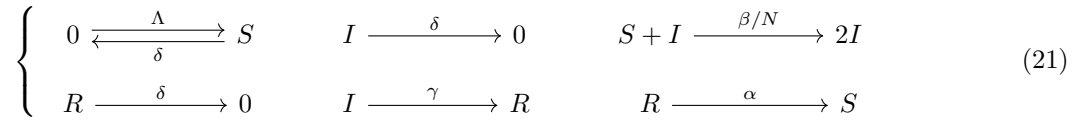
\begin{equation}
\label{sir-interaction}
\left\{ \; \;
\begin{tikzcd}
0 \arrow[rr,yshift=0.5ex,"\Lambda"] & & S \arrow[ll,yshift=-0.5ex,"\delta"] & I \arrow[rr,"\delta"] & & 0 & S + I \arrow[rr,"\beta/N"] & & 2I \\[-0.1in]
R \arrow[rr,"\delta"] & & 0 & I \arrow[rr,"\gamma"] & & R & R \arrow[rr,"\alpha"] & & S 
\end{tikzcd}
\right.
\end{equation}

The interaction network \eqref{sir-interaction} has the siphon $\mathcal{X} = \{ I, R \}$ corresponding to the set of infected states. The ALT-graph of \eqref{sir-interaction} on $\{ I, R \}$ is given in Figure \ref{fig:alt-graph-sir} (left). The strongly connected components are $\mathcal{Z}_1 = \{ I \}$, $\mathcal{Z}_2 = \{ R \}$, and $\mathcal{Z}_3 = \{ 0 \}$.  There is a single instance of the autocatalytic reaction labeled $\beta/N$ which corresponds to ``new infections.'' This reaction is critical since its only instance is autocatalytic (i.e., the self-loop on $I$) and it has no catalytic instance leaving the strongly connected component containing $I$. Consequently, by Algorithm \ref{alg:1}, we must remove the reaction labeled $\beta/N$ from $\mathcal{G}$ and add it to $\mathcal{G}_F$. The resulting autocatalytic decomposition is shown in Figure \ref{fig:alt-graph-sir} (right). Note that $\mathcal{G}_{-V}$ does not contain any autocatalytic reactions so we are finished and the basic reproduction number may be computed. The matrices $F$ and $-V$ corresponding to $\mathcal{G}_F$ and $\mathcal{G}_{-V}$ are the same as those provided by the standard next-generation matrix method of \cite{van2002reproduction}. Further computational details are provided in Appendix \ref{app:h}.

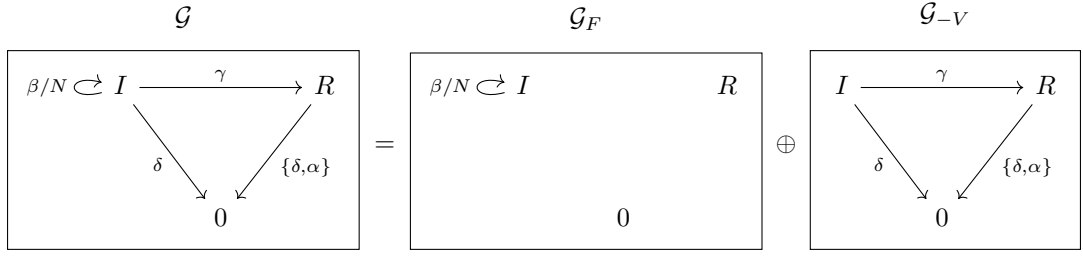
\begin{figure}[t]
    \centering
    \[
    \begin{tikzcd}[
        ampersand replacement=\&,
        execute at end picture={
            \draw[black, thin]
                ([shift={(-5pt,-5pt)}]current bounding box.south west)
                rectangle
                ([shift={(5pt,5pt)}]current bounding box.north east);
            \node[above] at
                ([shift={(0,5pt)}]current bounding box.north) {$\mathcal{G}$};
        }
    ]
        I \arrow[loop left,"\beta/N"]
          \arrow[rr,"\gamma"]
          \arrow[rdd,"\delta"']
        \& \&
        R \arrow[ldd,"{\{\delta,\alpha\}}"] \\
        \& \& \\
        \& 0 \&
    \end{tikzcd}
    \;=\;
    \begin{tikzcd}[
        ampersand replacement=\&,
        execute at end picture={
            \draw[black, thin]
                ([shift={(-5pt,-5pt)}]current bounding box.south west)
                rectangle
                ([shift={(5pt,5pt)}]current bounding box.north east);
            \node[above] at
                ([shift={(0,5pt)}]current bounding box.north) {$\mathcal{G}_F$};
        }
    ]
        I \arrow[loop left,"\beta/N"]
        \& \&
        R \\
        \& \& \\
        \& 0 \&
    \end{tikzcd}
    \;\oplus\;
    \begin{tikzcd}[
        ampersand replacement=\&,
        execute at end picture={
            \draw[black, thin]
                ([shift={(-5pt,-5pt)}]current bounding box.south west)
                rectangle
                ([shift={(5pt,5pt)}]current bounding box.north east);
            \node[above] at
                ([shift={(0,5pt)}]current bounding box.north) {$\mathcal{G}_{-V}$};
        }
    ]
        I \arrow[rr,"\gamma"]
          \arrow[rdd,"\delta"']
        \& \&
        R \arrow[ldd,"{\{\delta,\alpha\}}"] \\
        \& \& \\
        \& 0 \&
    \end{tikzcd}
    \]
    \caption{The ALT-graph $\mathcal{G}$ for the SIR network \eqref{sir-interaction} and its associated autocatalytic decomposition $\mathcal{G}_F \oplus \mathcal{G}_{-V}$. The graph decomposes as the direct sum $\mathcal{G}=\mathcal{G}_F\oplus\mathcal{G}_{-V}$, where $\mathcal{G}_F$ contains the $\beta/N$ self-interaction on $I$, while $\mathcal{G}_{-V}$ contains the remaining transitions involving $\gamma$, $\delta$, and $\{\delta,\alpha\}$.}
    \label{fig:alt-graph-sir}
\end{figure}

\end{example}

\begin{example}
\label{ex:vectorborne}
Consider the following vector-borne disease spread model proposed for dengue fever transmission between mosquitoes and humans in \cite{feng1997competitive} and further studied in \cite{van2002reproduction}:
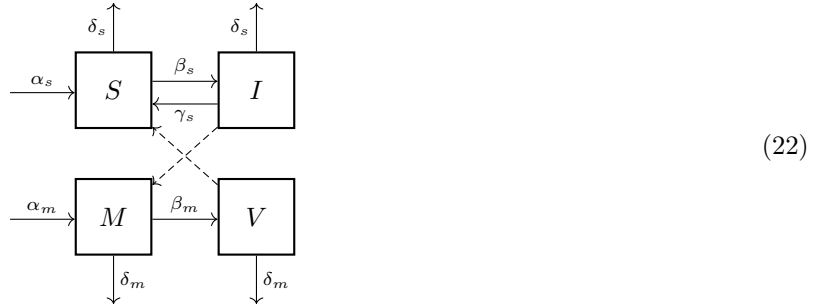
\begin{equation}
\label{sir-vector}
\begin{tikzcd}
& {\color{white}o} & {\color{white}o} & {\color{white}o} \\
{\color{white}o} \arrow[r,"\alpha_s"] & |[draw, thick, black, minimum size=1cm]| S \arrow[r,yshift=0.15cm,"\beta_s"] \arrow[u,"\delta_s"] & |[draw, thick, black, minimum size=1cm]| I \arrow[ld,dashed] \arrow[l,yshift=-0.15cm,"\gamma_s"] \arrow[u,"\delta_s"] \\
{\color{white}o} \arrow[r,"\alpha_m"] & |[draw, thick, black, minimum size=1cm]| M \arrow[r,"\beta_m"] \arrow[d,"\delta_m"] & |[draw, thick, black, minimum size=1cm]| V \arrow[lu,dashed] \arrow[d,"\delta_m"] \\
& {\color{white}o} & {\color{white}o} & {\color{white}o}
\end{tikzcd}
\end{equation}
where $\alpha_s> 0$ is the human birth rate, $\delta_s$ is the human death rate, $\beta_s > 0$ is the human infection rate, $1/\gamma_s$ is the human mean infectious period, $\alpha_m > 0$ is the mosquito birth rate, $\delta_m > 0$ is the mosquito death rate, and $\beta_m > 0$ is the mosquito infection rate.

The system \eqref{sir-vector} can be equivalently represented by the following chemical interaction network:
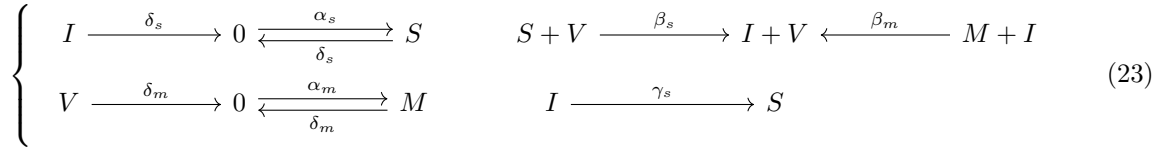
\begin{equation}
\label{vector}
\left\{ \; \;
\begin{tikzcd}
I \arrow[rr,"\delta_s"] & & 0 \arrow[rr,yshift=0.5ex,"\alpha_s"] & & S \arrow[ll,yshift=-0.5ex,"\delta_s"] & S + V \arrow[rr,"\beta_s"] & & I + V & & M + I \arrow[ll,"\beta_m"'] \\[-0.1in]
V \arrow[rr,"\delta_m"] & & 0 \arrow[rr,yshift=0.5ex,"\alpha_m"] & & M \arrow[ll,yshift=-0.5ex,"\delta_m"] & I \arrow[rr,"\gamma_s"] & & S & & \\[-0.1in]
\end{tikzcd}
\right.
\end{equation}

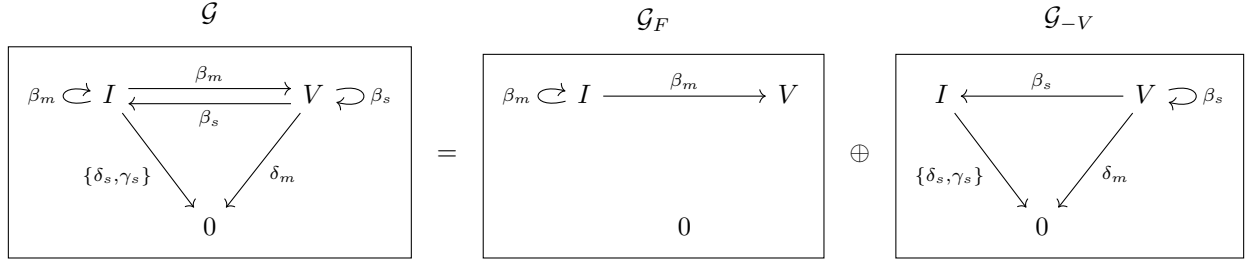
\begin{figure}[t]
\centering
\begin{tikzcd}[ampersand replacement=\&,execute at end picture={
        \draw[black, thin] 
            ([shift={(-5pt,-5pt)}]current bounding box.south west) 
            rectangle 
            ([shift={(5pt,5pt)}]current bounding box.north east);                  
            \node[above] at ([shift={(0,5pt)}]current bounding box.north) {$\mathcal{G}$};
    }
]
I \arrow[loop left,"\beta_m"] \arrow[rr,yshift=0.1cm,"\beta_m"] \arrow[rdd,"{\{\delta_s,\gamma_s\}}"'] 
\& \& 
V \arrow[loop right,"\beta_s"] \arrow[ll,yshift=-0.1cm,"\beta_s"] \arrow[ldd,"\delta_m"] \\
\& \& \\
\& 0 \&
\end{tikzcd}
\; = \;
\begin{tikzcd}[ampersand replacement=\&,execute at end picture={
        \draw[black, thin] 
            ([shift={(-5pt,-5pt)}]current bounding box.south west) 
            rectangle 
            ([shift={(5pt,5pt)}]current bounding box.north east);                  
            \node[above] at ([shift={(0,5pt)}]current bounding box.north) {$\mathcal{G}_F$};
    }
]
I \arrow[loop left,"\beta_m"] \arrow[rr,"\beta_m"]  
\& \& 
V \\
\& \& \\
\& 0 \&
\end{tikzcd}
\; $\oplus$ \;
\begin{tikzcd}[ampersand replacement=\&,execute at end picture={
        \draw[black, thin] 
            ([shift={(-5pt,-5pt)}]current bounding box.south west) 
            rectangle 
            ([shift={(5pt,5pt)}]current bounding box.north east);                  
            \node[above] at ([shift={(0,5pt)}]current bounding box.north) {$\mathcal{G}_{-V}$};
    }
]
I \arrow[rdd,"{\{\delta_s,\gamma_s\}}"'] 
\& \& 
V \arrow[loop right,"\beta_s"] \arrow[ll,"\beta_s"'] \arrow[ldd,"\delta_m"] \\
\& \& \\
\& 0 \&
\end{tikzcd}

\caption{The ALT-graph $\mathcal{G}$ for the vector-borne disease spread network \eqref{vector} and its associated autocatalytic decomposition
$\mathcal{G}_F \oplus \mathcal{G}_{-V}$. The graph decomposes as the direct sum
$\mathcal{G}=\mathcal{G}_F\oplus\mathcal{G}_{-V}$, where $\mathcal{G}_F$ contains
and autocatalytic and catalytic instance of $\beta_m$ and an autocatalytic instance of $\beta_s$ , while $\mathcal{G}_{-V}$ contains the remaining transitions involving $\gamma_s$, $\delta_m$, and
$\{\delta_s,\gamma_s\}$.}
\label{fig:alt-vector}
\end{figure}

The interaction network \eqref{vector} has the siphon $\mathcal{X} = \{ I, V \}$ corresponding to the set of infected states for both the human and mosquito populations. The ALT-graph of \eqref{vector} on $\{ I, V \}$ is given in Figure \ref{fig:alt-vector} (left). We have the strongly connected components $\mathcal{Z}_1 = \{ I, V \}$ and $\mathcal{Z}_2 = \{ 0 \}$. The reactions labeled $\beta_m$ and $\beta_s$ are both autocatalytic, with one autocatalytic and one catalytic instance each. Since there are no catalytic instances which leave $\mathcal{Z}_1$, both autocatalytic reactions are critical so that $\mathcal{R}_C \not= \emptyset$. Furthermore, since there is a directed path from $\mathcal{Z}_1$ to $0$, there are no critical strongly connected components so that $\mathcal{Z}_C = \emptyset$.

Applying Algorithm \ref{alg:1}, we identify that both $I$ and $V$ are the targets of multiple instances of autocatalytic reactions. There is no preference between the target vertices $I$ and $V$; however, Heuristic 1(b) dictates that we perform the available catalytic splitting rather than the autocatalytic splitting. Performing the catalytic splitting of $\beta_m$ leading to $V$ gives the autocatalytic decomposition $\mathcal{G}_F \oplus \mathcal{G}_{-V}$ given in Figure \ref{fig:alt-vector} (right). The reformulated $\mathcal{G}_{-V}$ has the autocatalytic reaction $\beta_s$; however, this autocatalytic reaction is not critical in $\mathcal{G}_{-V}$ because the strong linkage classes have changed to $\mathcal{Z}_1 = \{ I \}$, $\mathcal{Z}_2 = \{V\}$, and $\mathcal{Z}_3 = \{ 0 \}$ and there is an instance from $\mathcal{Z}_2$ to $\mathcal{Z}_1$. Consequently, $\mathcal{R}_C = \emptyset$. Similarly, we have that there are no critical strongly connected components since the edge labeled $\delta_m$ leads from $\mathcal{Z}_2$ to $0$, making disconnection from $0$ impossible. Since $\mathcal{R}_C = \emptyset$ and $\mathcal{Z}_C = \emptyset$, we are done. The resulting matrix $F$ is rank 1. Further details on the computations are contained in Appendix \ref{app:h}.

We make the following note about the application of Algorithm \ref{alg:1} to this example. Suppose we had prioritized removing all autocatalytic reactions in one step, rather than iteratively as in Algorithm \ref{alg:1}. If we had simultaneously performed two autocatalytic splittings of $\beta_m$ and $\beta_s$, or simultaneously performed two catalytic splittings of $\beta_m$ and $\beta_s$, the resulting matrix $F$ would have rank 2. The results in greater computational work in determining $\rho(FV^{-1})$ than the rank-1 splitting produced in Figure \ref{fig:alt-vector}. This emphasizes the importance of following the algorithm iteratively and applying the heuristics, so the rank rises minimally at each step. 

\end{example}

\section{Conclusions}
\label{sec:conclusions}

This paper addresses a significant recent focus in mathematical biology: how methods from one subdiscipline (e.g., mathematical epidemiology) can be applied to another (e.g., mathematical biochemistry). Recent work has shown that the next-generation matrix method, classically used to compute the basic reproduction number in mathematical epidemiology, can be used to evaluate the stability of boundary steady states in mathematical biochemistry. Widespread application, however, has been hindered by the significant care which must be taken in the process to ensure mathematical validity and computational feasibility \cite{avram2024advancing,johnston2026boundary}. Furthermore, the reproduction number produced in the biochemical context did not carry the same direct biological interpretation as in the epidemiological one, where the value corresponds to the number of secondary infections produced by a single typical primary infection.

In this paper, we have significantly addressed these deficiencies. We have presented  a method (Theorem \ref{thm:1} and Algorithm \ref{alg:1}) for performing Jacobian splittings $J = F - V$ that produces valid next-generation matrices $FV^{-1}$ for computing reproduction numbers $\rho(FV^{-1})$ in both mathematical epidemiology and mathematical biochemistry. The method involves constructing the ALT-graph (autocatalysis-leak-transition graph) (Definition \ref{def:alt-graph}) and removing critical autocatalytic reactions and critical strongly connected components according to specific splitting rules (Definition \ref{def:splitting}). The resulting decomposition defines the matrices $F$ and $V$ (Definition \ref{def:alt-matrix}) so that a valid next-generation matrix may be computed (Theorem \ref{thm:1}). The method also relates the key feature of boundary stability and instability in biochemical reaction systems to the growing literature on autocatalysis. We demonstrated the method on several examples from mathematical biochemistry and mathematical epidemiology. By producing low-rank matrices $F$, the method reduces the computational cost of evaluating $\rho(FV^{-1})$ relative to splitting methods that do not minimize the rank of $F$, while guaranteeing validity of the next-generation matrix $FV^{-1}$.

Future work will focus on the following ongoing projects:
\begin{enumerate}
\item[] \textbf{Computing Rank-$1$ Reproduction Numbers:} Even in complex biochemical reaction systems with boundary steady states, empirical evidence from the examples in this paper and related work \cite{avram2024advancing,johnston2026boundary} suggests that splittings $J = F - V$ leading to a valid next-generation matrix $FV^{-1}$ where $F$ is rank-$1$ are frequently attainable. In such cases, the computation of $\rho(FV^{-1})$ simplifies to $\operatorname{tr}(FV^{-1})$, which often reduces to a single nonzero entry. 
Future work will focus on:
\begin{enumerate}
\item optimizing Algorithm \ref{alg:1} to yield minimal-rank next-generation matrices $FV^{-1}$; and
\item developing specialized algorithms for computing $\rho(FV^{-1})$ for rank-$1$ matrices $F$.
\end{enumerate}
\item[] \textbf{Applications to Bifunctional Enzyme Models:} Recent work in biochemical reaction networks, particularly work related to absolute concentration robustness \cite{joshi2024bifunctional,avram2024advancing,joshi2026bifunctional,joshi2026bistability} and discrete extinction events \cite{agarwal2026extinction,johnston2017computational,johnston2018conditions}, has uncovered numerous examples of mechanisms with nontrivial boundary steady states. The results of this paper provide a systematic approach to assessing the parameter-dependent stability regimes of the corresponding boundary steady states. Immediate future work will focus on applications of Theorem \ref{thm:1} to $n$-site bifunctional enzyme motifs, an extension of those covered in \cite{joshi2026bistability}, and distributed---rather than processive---dual-site enzyme motifs \cite{markevich2004signaling}.
\item[] \textbf{Alternative Kinetic Assumptions:} In this work, we have restricted ourselves to chemical reaction networks that are integer-valued and binary; however, many modeling frameworks have been recently studied which allow more general stoichiometric and kinetic-order coefficients, such as power-law kinetics \cite{fogler1999elements,fortun2019deficiency}, generalized mass-action kinetics \cite{muller2012generalized,johnston2014translated}, and S-systems \cite{savageau1969biochemical,knorre1979ma}. While these kinetic frameworks differ from those considered here---and the ALT-graph rules (a)--(f) of Definition \ref{def:alt-graph} would need to be extended to accommodate non-binary stoichiometries and more general kinetic orders---the same principles of splitting the Jacobian and attaining a matrix $V$ satisfying the conditions of Theorem \ref{thm:ngm2} remain valid. Future work will focus on adapting the ALT-graph approach to networks with these more general kinetic frameworks.
\end{enumerate}

\paragraph{Acknowledgments} MDJ and PJC are supported by National Science Foundation Grant No. DMS-2213390.

\bibliographystyle{unsrt}

\appendix
\section{Proof of Theorem \ref{thm:2}}
\label{app:a}

\begin{proof}[Proof of Theorem \ref{thm:2}]
Let $c^* \in \partial \mathbb{R}_{\geq 0}^n$ be a boundary steady state corresponding to the siphon $\mathcal{X}$. We reindex according to $\mathcal{X}$ so that $x^*_i = 0$ and $y^*_i > 0$. To prove that $-J$ is a $Z$-matrix, it suffices to show that the off-diagonal entries of the restricted Jacobian are nonnegative, i.e., $J_{ij} \geq 0$ for all distinct $X_i, X_j \in \mathcal{X}$. Differentiating the $i$th equation of \eqref{eq:crs2} with respect to $x_j$ yields:
\[
J_{ij} = \sum_{R_k \in \mathcal{R}} (\beta'_{ik} - \alpha'_{ik}) \partial_j R_k.
\]
We evaluate the partial derivative for each reaction $R_k \in \mathcal{R}$ using assumptions \textbf{(K1)}--\textbf{(K5)}:
\begin{itemize}
    \item If $\alpha'_{jk} = 0$, then the species $X_j$ is not a reactant, so $\partial_j R_k = 0$ by \textbf{(K3)}.
    \item If $\alpha_{jk}'>0$ and $x^*_\ell=0$ for some $\ell$ with $\alpha_{\ell k}'-\delta_{jl}>0$, then reaction $R_k$ either has another reactant $X_\ell$ in the siphon whose concentration is zero, or the multiplicity of the reactant $X_j$ is 2. Thus, $\partial_jR_k=0$ by \textbf{(K5)}.
    \item If $\alpha'_{jk} = 1$ and $X_j$ is the unique reactant of $R_k$ within $\mathcal{X}$, then all other required reactants lie outside the siphon and have positive concentrations (i.e., $y^*_\ell > 0$) which implies $\partial_j R_k > 0$ by \textbf{(K4)}. Since the network is binary and $X_j$ is already a reactant, the siphon species $X_i$ cannot be a reactant, i.e., $\alpha'_{ik} = 0$. This ensures that the stoichiometric coefficient is nonnegative so that $\beta'_{ik} - \alpha'_{ik} = \beta'_{ik} \geq 0$.
\end{itemize}
Every non-zero term in the sum for $J_{ij}$ is therefore the product of a nonnegative coefficient $\beta_{ik} \geq 0$ and a strictly positive derivative. It follows that $J_{ij} \geq 0$ for all $i \neq j$, confirming that $-J$ has non-positive off-diagonal entries and is a $Z$-matrix.
\end{proof}

\section{Proof of Theorem \ref{thm:jacobian}}
\label{app:b}

\begin{proof}[Proof of Theorem \ref{thm:jacobian}]
Let $c^* \in \partial \mathbb{R}_{\geq 0}^n$ be the boundary steady state corresponding to the siphon $\mathcal{X}$. We reindex according to $\mathcal{X}$ so that $x^*_i = 0$ and $y^*_i > 0$, and $c^* = (x^*,y^*) \in \partial \mathbb{R}_{\geq 0}^n$. Differentiating the $j$th equation of \eqref{eq:crs2} with respect to $x_i$ yields:
\[
J_{ji} = \sum_{R_k \in \mathcal{R}} (\beta'_{jk} - \alpha'_{jk}) \partial_i R_k.
\]
We show that $[L(\mathcal{G})]_{ji} = J_{ji}$ for all $i, j = 1, \dots, m$ by considering the off-diagonal and diagonal entries separately. Note that reactions whose $\mathcal{X}$-reduced source contains two reactant species or one reactant with multiplicity 2 make no contribution to $J$ by \textbf{(K5)}.

\paragraph{Off-diagonal entries:} For $i \neq j$, $X_j$ is not a reactant (i.e., $\alpha'_{jk} = 0$), so $J_{ji} = \sum_{R_k} \beta'_{jk} \partial_i R_k$. According to Definition \ref{def:alt-graph}, each contributing reaction $R_k$ is assigned to the multiset $\phi(e_{ij})$ with a multiplicity equal to its product stoichiometric coefficient $\beta'_{jk}$ (specifically, $1$ for types (c), (d), and (f); $2$ for type (e); and $0$ for types (a) and (b)). Under the multiset sum interpretation of Definition \ref{def:alt-matrix}, we obtain:
\[
[L(\mathcal{G})]_{ji} = \sum_{R_k \in \phi(e_{ij})} \partial_i R_k = \sum_{R_k \in \mathcal{R}} \beta'_{jk} \partial_i R_k = J_{ji}.
\]

\paragraph{Diagonal entries:} For $i=j$, we have $\alpha'_{ik} = 1$ so that $J_{ii} = \sum_{R_k} (\beta'_{ik} - 1) \partial_i R_k$. We evaluate the coefficient $\beta'_{ik} - 1$ for each reaction type and compare it with the terms of the ALT-matrix diagonal element:
\[
[L(\mathcal{G})]_{ii} = \sum_{R_k \in \phi(e_{ii})} \partial_i R_k - \sum_{R_k \in \phi(e_{i0})} \partial_i R_k - \sum_{\substack{p=1 \\ p\neq i}}^m \sum_{R_k \in \phi(e_{ip})} \partial_i R_k
\]
Using the multiset assignment rules of Definition \ref{def:alt-graph}:
\begin{itemize}
    \item If $R_k': X_i \to 2X_i$: We have that $\beta'_{ik} - 1 = 1$ while $R_k \in \phi(e_{ii})$ contributes $+1$ to $[L(\mathcal{G})]_{ii}$.
    \item If $R_k': X_i \to 0$: We have that $\beta'_{ik} - 1 = -1$ while $R_k \in \phi(e_{i0})$ contributes $-1$ to $[L(\mathcal{G})]_{ii}$.
    \item If $R_k': X_i \to X_j$: We have that $\beta'_{ik} - 1 = -1$ while $R_k \in \phi(e_{ij})$ contributes $-1$ to $[L(\mathcal{G})]_{ii}$.
    \item If $R_k': X_i \to X_i + X_j$: We have that $\beta'_{ik} - 1 = 0$ while $R_k \in \phi(e_{ii})$ contributes $+1$ to $[L(\mathcal{G})]_{ii}$ and $R_k \in \phi(e_{ij})$ contributes $-1$ to $[L(\mathcal{G})]_{ii}$, which sums to $0$.
    \item If $R_k': X_i \to 2X_j$: We have that $\beta'_{ik} - 1 = -1$ while $R_k \in \phi(e_{ii})$ contributes $+1$ to $[L(\mathcal{G})]_{ii}$ and $\{R_k,R_k\} \subseteq \phi(e_{ij})$ contributes $-2$, which sums to $-1$.
    \item If $R_k': X_i \to X_j + X_{\ell}$: We have that $\beta'_{ik} - 1 = -1$ while $R_k \in \phi(e_{ii})$ contributes $+1$ to $L_{ii}$, $R_k \in \phi(e_{ij})$ contributes $-1$ to $[L(\mathcal{G})]_{ii}$, and $R_k \in \phi(e_{i\ell})$ contributes $-1$ to $L_{ii}$, which sums to $-1$.
\end{itemize}
Since the net coefficient of $\partial_i R_k$ in $[L(\mathcal{G})]_{ii}$ matches $\beta'_{ik} - 1$ in all possible cases, we have $[L(\mathcal{G})]_{ii} = J_{ii}$. Thus, $L(\mathcal{G}) = J$.
\end{proof}

\section{Proof of Lemma \ref{lemma:multiset}}
\label{app:c}

\begin{proof}[Proof of Lemma \ref{lemma:multiset}]
Let $L(\mathcal{G}_1)$, $L(\mathcal{G}_2)$ and $L(\mathcal{G}_1\oplus\mathcal{G}_2)$ be the ALT-matrices corresponding to $\mathcal{G}_1$, $\mathcal{G}_2$ and $\mathcal{G}_1\oplus\mathcal{G}_2$ respectively. We prove that $[L(\mathcal{G}_1\oplus\mathcal{G}_2)]_{ji}=[L(\mathcal{G}_1)]_{ji}+[L(\mathcal{G}_2)]_{ji}$ for all $j,i=1,\ldots,m$.
\paragraph{Off-diagonal entries:}
For $i\neq j$, by Definition \ref{def:alt-matrix}, we have:
\[
[L(\mathcal{G}_1\oplus\mathcal{G}_2)]_{ji}=\sum_{R_k\in\phi_{\oplus}(e_{ij})}\partial_iR_k.
\]
Since $\phi_{\oplus}(e_{ij})=\phi_1(e_{ij})\oplus\phi_2(e_{ij})$ and multiset summation counts repeated labels with multiplicity, we have:
\[
\sum_{R_k\in\phi_{\oplus}(e_{ij})}\partial_iR_k=\sum_{R_k\in\phi_{1}(e_{ij})}\partial_iR_k+\sum_{R_k\in\phi_{2}(e_{ij})}\partial_iR_k.
\]
Therefore, $[L(\mathcal{G}_1\oplus\mathcal{G}_2)]_{ji}=[L(\mathcal{G}_1)]_{ji}+[L(\mathcal{G}_2)]_{ji}$.
\paragraph{Diagonal entries:}
For $i=j$, again by Definition \ref{def:alt-matrix}, we have:
\[
[L(\mathcal{G}_1\oplus\mathcal{G}_2)]_{ii}=\sum_{R_k \in \phi_{\oplus}(e_{ii})} \partial_i R_k - \sum_{R_k \in \phi_{\oplus}(e_{i0})} \partial_i R_k - \sum_{\substack{p=1 \\ p\neq i}}^m \sum_{R_k \in \phi_{\oplus}(e_{ip})} \partial_i R_k
\]
Again, we can use the multiset summation property of counting repeated labels with multiplicity to divide the contributions of $\mathcal{G}_1$ and $\mathcal{G}_2$. Hence:
\[
[L(\mathcal{G}_1\oplus\mathcal{G}_2)]_{ii}=\left[\sum_{R_k \in \phi_{1}(e_{ii})} \partial_i R_k - \sum_{R_k \in \phi_{1}(e_{i0})} \partial_i R_k - \sum_{\substack{p=1 \\ p\neq i}}^m \sum_{R_k \in \phi_{1}(e_{ip})} \partial_i R_k\right]\]

\[
+\left[\sum_{R_k \in \phi_{2}(e_{ii})} \partial_i R_k - \sum_{R_k \in \phi_{2}(e_{i0})} \partial_i R_k - \sum_{\substack{p=1 \\ p\neq i}}^m \sum_{R_k \in \phi_{2}(e_{ip})} \partial_i R_k\right]
\]
The first bracket is $[L(\mathcal{G}_1)]_{ii}$ and the second one is $[L(\mathcal{G}_2)]_{ii}$. Therefore, $[L(\mathcal{G}_1\oplus\mathcal{G}_2)]_{ii}=[L(\mathcal{G}_1)]_{ii}+[L(\mathcal{G}_2)]_{ii}$. Since this is true for every entry, we can conclude that $L(\mathcal{G}_1\oplus\mathcal{G}_2)=L(\mathcal{G}_1)+L(\mathcal{G}_2).$
\end{proof}

\section{Proof of Theorem \ref{thm:1}}
\label{app:d}

\begin{proof}[Proof of Theorem \ref{thm:1}]
Let $J$ be the Jacobian of the reduced system with respect to $\mathcal{X}$ and evaluated at $c^*=(x^*,y^*)$. Also, let $\mathcal{G}=(\mathcal{V},\mathcal{E},\Omega,\phi)$ be the ALT-graph of the network, which has an autocatalytic decomposition $\mathcal{G}=\mathcal{G}_F\oplus \mathcal{G}_{-V}$, with $\mathcal{G}_F=(\mathcal{V},\mathcal{E}_F,\Omega,\phi_F)$ and $\mathcal{G}_{-V}=(\mathcal{V},\mathcal{E}_{-V},\Omega,\phi_{-V})$. Now let $L(\mathcal{G})$, $L(\mathcal{G}_F)$ and $L(\mathcal{G}_{-V})$ be the ALT-matrices of the full graph, the $F$ graph and the $-V$ graph respectively. By Theorem \ref{thm:jacobian}, the ALT-matrix of $\mathcal{G}$ coincides with $J$, therefore, $L(\mathcal{G})=J$.
Because the terms of an ALT-matrix are linear with respect to multiset summation, by Lemma \ref{lemma:multiset} we get:
\[
L(\mathcal{G})=L(\mathcal{G}_F\oplus\mathcal{G}_{-V})=L(\mathcal{G}_F)+L(\mathcal{G}_{-V})
\]
Defining $F=L(\mathcal{G}_F)$ and $-V=L(\mathcal{G}_{-V})$, we have:
\[
J=L(\mathcal{G})=F+(-V)=F-V
\]

To prove that $F$ is nonnegative, we use the assumption that the splitting $\mathcal{G}=\mathcal{G}_F\oplus\mathcal{G}_{-V}$ is an autocatalytic decomposition. By Definition \ref{def:splitting}, this implies that $\mathcal{G}_F$ is built from autocatalytic splittings and catalytic splittings so that 
\[\mathcal{G} = \mathcal{G}_{F_1} \oplus \mathcal{G}_{-V_1} = \mathcal{G}_{F_1} \oplus ( \mathcal{G}_{F_2} \oplus \mathcal{G}_{-V_2}) = \cdots = (\mathcal{G}_{F_1} \oplus \cdots \oplus \mathcal{G}_{F_p} ) \oplus \mathcal{G}_{-V_p} =\mathcal{G}_F \oplus \mathcal{G}_{-V}\]
where $\mathcal{G}_F = \mathcal{G}_{F_1} \oplus \cdots \oplus \mathcal{G}_{F_p}$, $\mathcal{G}_{-V} = \mathcal{G}_{-V_p}$, and each $\mathcal{G}_{-V_{\ell-1}} = \mathcal{G}_{F_{\ell}} \oplus \mathcal{G}_{-V_{\ell}}$ is either an autocatalytic or catalytic splitting. Note that an autocatalytic splitting places one autocatalytic instance into $\mathcal{G}_{F_{\ell}}$, while a catalytic splitting places one pair of autocatalytic and catalytic instances into $\mathcal{G}_{F_{\ell}}$. We have the following cases:
\begin{enumerate}
\item \textbf{Autocatalytic splitting $\mathcal{G}_{F_{\ell}} \oplus \mathcal{G}_{-V_{\ell}}$:} $\mathcal{G}_{F_{\ell}}$ only contains one self-loop $(X_i\rightarrow2X_i)$ and, therefore, the contribution to $L(\mathcal{G}_{F_{\ell}})$ is only $[L(\mathcal{G}_{F_{\ell}})]_{ii}=\partial_iR_k>0$ by Definition \ref{def:alt-matrix}.
\item \textbf{Catalytic splitting $\mathcal{G}_{F_{\ell}} \oplus \mathcal{G}_{-V_{\ell}}$:} $\mathcal{G}_{F_{\ell}}$ only contains one self-loop and one transition $(X_i\rightarrow 2X_i$ and $X_i\rightarrow X_j)$ and, therefore, the contribution to $L(\mathcal{G}_{F_{\ell}})$ is $[L(\mathcal{G}_{F_{\ell}})]_{ii}=\partial_iR_k-\partial_iR_k=0$ and $[L(\mathcal{G}_{F_{\ell}})]_{ji} = \partial_iR_k>0$ by Definition \ref{def:alt-matrix}.
\end{enumerate}
Since every possible component of the autocatalytic decomposition contributes to a nonnegative $L(\mathcal{G}_{F_{\ell}})$ matrix and $\mathcal{E}_F\neq \emptyset$ guarantees that there is at least one strictly positive entry, we have 
\[F = L(\mathcal{G}_F) = L(\mathcal{G}_{F_1} \oplus \cdots \oplus \mathcal{G}_{F_p}) = L(\mathcal{G}_{F_1}) + \cdots + L(\mathcal{G}_{F_p}) \]
is nonnegative and nonzero.

We now consider $L(\mathcal{G}_{-V}) = -V$. 
The off-diagonal entries of $L(\mathcal{G}_{-V})$ are:
\[[L(\mathcal{G}_{-V})]_{ji}=\sum_{R_k\in \phi_{-V}(e_{ij})}\partial_iR_k,\]
for $i\neq j$. Each of the terms of the sum is nonnegative, which implies that $[L(\mathcal{G}_{-V})]_{ji}\geq 0$. Because $-V=L(\mathcal{G}_{-V})$, we have that $V_{ji}\leq 0$ for $i\neq j$. Hence, $V$ is a $Z$-matrix.

Now let \(\mathcal{Z}_1,\ldots,\mathcal{Z}_q\) be the strongly connected components of \(\mathcal{G}_{-V}\) restricted to the vertices in \(\mathcal X\) (i.e., we do not include the vertex $0$). The condensation graph of \(\mathcal{G}_{-V}\) is acyclic, and therefore its strongly connected components can be ordered topologically (Section 6.2, \cite{horn2012matrix}). Permuting the species according to this ordering gives a permutation matrix \(P\) such that
\[
PVP^T=
\begin{pmatrix}
A_1 & B_{12} & \cdots & B_{1q}\\
0 & A_2 & \cdots & B_{2q}\\
\vdots & \vdots & \ddots & \vdots\\
0 & 0 & \cdots & A_q
\end{pmatrix}.
\]
Each diagonal block $A_h$ is an irreducible $Z$-matrix which corresponds to a single strongly connected component $\mathcal{Z}_h$. The off-diagonal blocks $B_{ij}$ are nonpositive because $V$ is a $Z$-matrix. To satisfy the conditions of Theorem \ref{thm:ngm2}, we prove 
that $\mathbbm{1}^TA_h$ is nonnegative and nonzero, i.e., each column sum of $A_h$ is nonnegative and at least one column sum is strictly positive. 

We now consider the contribution of each reaction type in Table \ref{tab:jacobian} to $\mathbbm{1}^T A_h$. We note that the possible contributions of the reactions $R_k \in \mathcal{R}$ to $\mathbbm{1}^T A_h$ are given by the negatives of the Jacobian column in Table \ref{tab:jacobian} restricted to the component $\mathcal{Z}_h$. Specifically, when restricting to $A_h$, if $X_j\not\in \mathcal{Z}_h$ or $X_{\ell} \not\in \mathcal{Z}_h$, we may ignore the corresponding rows in the Jacobian. We have the following:
\begin{enumerate}
\item \textbf{Autocatalytic or Catalytic Transition (Types I, II, or III) with $X_j, X_{\ell} \in \mathcal{Z}_h$:} Since there are no catalytic instances associated with the reaction $R_k \in \mathcal{R}$ which leave $\mathcal{Z}_h$, we may sum over the negative of the whole vector in the Jacobian entry column of Table \ref{tab:jacobian}. In all cases, the contribution of the reaction to the $i$th column of $\mathbbm{1}^T A_h$ is $-\partial_i R_k < 0$.
\item \textbf{Catalytic Transition (Type I) with $X_j \not\in \mathcal{Z}_h$:} The only relevant entry is the $[A_h]_{ii} = 0$ entry (corresponding to $J_{ii} = 0$ in Table \ref{tab:jacobian}) so that the contribution of the reaction to the $i$th column of $\mathbbm{1}^T A_h$ is $0$.
\item \textbf{Transition with $X_j \in \mathcal{Z}_h$ or Catalytic Transition (Type III) with $X_j \in \mathcal{Z}_h$ and $X_\ell \not\in \mathcal{Z}_h$:} The only relevant nonzero entries are the $[A_h]_{ii} = \partial_i R_k$ entry (corresponding to $J_{ii} = - \partial_i R_k$ in Table \ref{tab:jacobian}) and the $[A_h]_{ji} = - \partial_i R_k$  entry (corresponding to $J_{ji} = \partial_i R_k$) so that the contribution of the reaction to the $i$th column of $\mathbbm{1}^T A_h$ is $\partial_i R_k - \partial_i R_k = 0$.
\item \textbf{Leak or Transition with $X_j \not\in \mathcal{Z}_h$ or Catalytic Transition (Type II or Type III) with $X_j, X_\ell \not\in \mathcal{Z}_h$:} The only relevant nonzero entry is the $[A_h]_{ii} = \partial_i R_k$ entry (corresponding to $J_{ii} = - \partial_i R_k$ in Table \ref{tab:jacobian}) so that the contribution of the reaction to the $i$th column of $\mathbbm{1}^T A_h$ is $\partial_i R_k > 0$.
\end{enumerate}

We first prove that $\mathbbm{1}^T A_h$ is nonnegative. Since $\mathcal{G}_{-V}$ has no critical autocatalytic reactions by assumption, every autocatalytic instance in $\mathcal{Z}_h$ must have at least one catalytic instance that leaves the strongly connected component $\mathcal{Z}_h$. Consequently, the reaction may not be any of the cases listed in item 1. above. Since the remaining cases listed in items 2.-4. are nonnegative, we have that $\mathbbm{1}^T A_h$ is nonnegative.

We now prove that $\mathbbm{1}^T A_h$ is nonzero. It is clear from the cases above that $\mathcal{Z}_h$ must have at least one of the cases indicated in item $4.$ above. Suppose instead that $\mathcal{Z}_h$ only contains the cases indicated in item $2.$ and $3.$ (we may disregard the cases in item $1.$ because $\mathcal{G}_{-V}$ does not contain any critical autocatalytic reactions). In these cases, the only edges which leave $\mathcal{Z}_h$ are singleton catalytic instances of autocatalytic reactions. Consequently, we can disconnect $\mathcal{Z}_h$ from $0$ by removing the catalytic instances which leave $\mathcal{Z}_h$. Since this is not allowed because there are no critical strongly connected components, it must be the case that $\mathcal{Z}_h$ has at least one of the cases indicated in item $4.$ above. In such cases we have that the contribution to the $i$th column of $\mathbbm{1}^T A_h$ is $\partial_i R_k > 0$. Since all of the reaction possibilities in items $2.$, $3.$, and $4.$ above contribute either positive or zero terms, it follows that $\mathbbm{1}^T A_h$ is nonzero. 

The conditions of Theorem \ref{thm:ngm2} are satisfied and therefore $V^{-1}$ is nonnegative. Because $F$ is nonnegative and nonzero, and $V$ is a Z-matrix with a nonnegative inverse, the conditions of Theorem \ref{thm:ngm} are satisfied and $c^*$ is locally asymptotically stable with respect to $\mathsf{C}_{c_0}$ if $\rho(FV^{-1}) < 1$ and $c^*$ is unstable with respect to $\mathsf{C}_{c_0}$ if $\rho(FV^{-1}) > 1$. 
\end{proof}

\section{Proof of Corollary \ref{cor:2}}
\label{app:e}

\begin{proof}[Proof of Corollary \ref{cor:2}, part 1] Since $\mathcal{G}$ does not have any critical autocatalytic reactions or critical strongly connected components, we may construct the splitting $J = F - V$ with $F=0$ so that $V = -J$. Since $V$ is a $Z$-matrix with a nonnegative inverse, it follows that $-J$ is as well. It follows that $-J$ is an $M$-matrix by Theorem 1 of \cite{plemmons1977m}, and consequently the eigenvalues of $-J$ have strictly positive real part. Consequently, the eigenvalues of $J$ have strictly negative real part. The result follows by the Hartman-Grobman Theorem \cite{perko2001}.
\end{proof}
\begin{proof}[Proof of Corollary \ref{cor:2}, part 2]

Suppose that $\mathcal{R}_C \not= \emptyset$ or $\mathcal{Z}_C \not= \emptyset$.

First, consider the set of critical autocatalytic reactions $\mathcal{R}_C$. For every reaction $R_k \in \mathcal{R}_C$, we perform an autocatalytic splitting on the instance $R_k \in \phi(e_{ii})$ (i.e., remove  $R_k \in \phi(e_{ii})$ from $\mathcal{G}_{-V}$ and add  $R_k \in \phi(e_{ii})$ to $\mathcal{G}_{F}$). We note that all catalytic instances of critical autocatalytic reactions in $\mathcal{G}$ are not autocatalytic in $\mathcal{G}_{-V}$ so that $\mathcal{G}_{-V}$ has no critical autocatalytic reactions. Furthermore, removing autocatalytic instances from an ALT-graph does not alter the connectivity of the graph so that $\mathcal{G}$ and $\mathcal{G}_{-V}$ have the same strongly connected components and there is still a path from every vertex in $\mathcal{G}_{-V}$ to $0$. Consequently, the critical strongly connected components are not altered by removing these autocatalytic instances. 

Now consider the critical strongly connected components $\mathcal{Z}_C$. For every strongly connected component $\mathcal{Z}_h\in\mathcal{Z}_C$, choose an autocatalytic reaction $R_k\in\mathcal{R}_A$ that has a catalytic instance $R_k \in \phi(e_{ij})$ where $X_i \in \mathcal{Z}_h$ and $X_j \not\in \mathcal{Z}_h$. We perform an autocatalytic splitting on the autocatalytic instance. Note that transferring an autocatalytic instance to $\mathcal{G}_F$ does not change the connectivity structure of the underlying graph so that $\mathcal{G}$ and $\mathcal{G}_{-V}$ have the same connectivity structure. Consequently, there is still a directed path in $\mathcal{G}_{-V}$ from every vertex to $0$. Furthermore, each selected reaction $R_k$ is not autocatalytic in $\mathcal{G}_{-V}$, which implies that $R_k \in \phi(e_{ij})$ is not a catalytic instance in $\mathcal{G}_{-V}$. Consequently, the strongly connected component $\mathcal{Z}_h$ cannot be disconnected from $0$ through removal of catalytic instances and is therefore no longer critical.

This process results in the autocatalytic splitting $\mathcal{G} = \mathcal{G}_F \oplus \mathcal{G}_{-V}$ where
\[\mathcal{G}_{F} = (\mathcal{G}_{F_1} \oplus \cdots \oplus \mathcal{G}_{F_{|\mathcal{R}_C|}}) \oplus (\mathcal{G}_{F_{{|\mathcal{R}_C|}+1}} \oplus \cdots \oplus \mathcal{G}_{F_{|\mathcal{R}_C|+|\mathcal{Z}_C|}})\]
and $\mathcal{G}_{F_j}$, $j=1, \ldots, {|\mathcal{R}_C|}$, contains a single autocatalytic instance $R_k \in \phi(e_{ii})$ for every critical autocatalytic reaction $R_k \in \mathcal{R}_C$, and $\mathcal{G}_{F_j}$, $j={|\mathcal{R}_C|}+1, \ldots, {|\mathcal{R}_C|+|\mathcal{Z}_C|}$, contains a single autocatalytic instance $R_k \in \phi(e_{ii})$ from every critical strongly connected components $\mathcal{Z}_h \in \mathcal{Z}_C$. The graph $\mathcal{G}_{-V}$ contains whatever instances are leftover and, by construction, contains no critical autocatalytic reactions or critical strongly connected components and has a path from every vertex to $0$. By Theorem \ref{thm:1}, we have $F = L(\mathcal{G}_{F})$ is a nonnegative nonzero matrix and $V$, where $-V = L(\mathcal{G}_{-V})$, is a $Z$-matrix with nonnegative inverse. Furthermore, by Lemma \ref{lemma:multiset} we have that
\[F = L(\mathcal{G}_F) = L(\mathcal{G}_{F_1}) + \cdots + L(\mathcal{G}_{F_{|\mathcal{R}_C|+|\mathcal{Z}_C|}}) = F_1 + \cdots + F_{|\mathcal{R}_C|+|\mathcal{Z}_C|}\]
where $F_j$, $j=1, \ldots, {|\mathcal{R}_C|+|\mathcal{Z}_C|},$ consists of a single nonzero entry on the diagonal. In the maximal case, there is a single nonzero entry on the diagonal for each critical autocatalytic reaction and each critical strongly connected components, although we note that it could be less than this number if multiple autocatalytic splittings share a common self-edge $e_{ii}$. Consequently, we have $\mbox{rank}(F) \leq |\mathcal{R}_C| + | \mathcal{Z}_C |$. Since $V$ is invertible, it follows that $\mbox{rank}(FV^{-1}) \leq |\mathcal{R}_C| + | \mathcal{Z}_C |$, which completes the proof.
\end{proof}

\section{Critical Set Examples}
\label{app:f}

In this Appendix, we present some examples which illustrate in detail the subtle role criticality of both autocatalytic reactions and strongly connected components plays in the invertibility of $V$.

\begin{example}
Consider the following general ALT-graph:
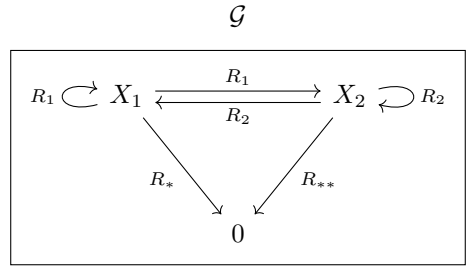
\[\begin{tikzcd}[ampersand replacement=\&,execute at end picture={
        \draw[black, thin] 
            ([shift={(-5pt,-5pt)}]current bounding box.south west) 
            rectangle 
            ([shift={(5pt,5pt)}]current bounding box.north east);                  
            \node[above] at ([shift={(0,5pt)}]current bounding box.north) {$\mathcal{G}$};
    }
]
X_1 \arrow[loop left,"R_1"] \arrow[rr,yshift=0.5ex,"R_1"] \arrow[rdd,"R_*"'] \& \& X_2 \arrow[ll,yshift=-0.5ex,"R_2"] \arrow[ldd,"R_{**}"] \arrow[loop right,"R_2"] \\
\& \& \\
\& 0 \&
\end{tikzcd}
\]

We analyze four distinct mechanisms parameterized by the edge labels $R_*$ (from $X_1$ to $0$) and $R_{**}$ (from $X_2$ to $0$). Despite the common network structure, each case leads to different criticality structures and stability conclusions, illustrating how the ALT-graph methodology adapts to varying reaction networks. The cases also illuminate how the critical autocatalytic reactions and the critical strongly connected components make it more challenging to guarantee the conclusions of Theorem \ref{thm:ngm} hold by preventing the structural arguments for Theorem \ref{thm:ngm2} from holding. Specifically, critical autocatalytic reactions make it more challenging to guarantee $V$ is diagonally column dominant while critical strongly connected components prevent $V$ from being invertible.

Note that, in all cases, we have the autocatalytic reaction set $\mathcal{R}_A = \{ R_1, R_2 \}$ and the strongly connected components $\mathcal{Z}_1 = \{ X_1, X_2 \}$ and $\mathcal{Z}_2 = \{ 0 \}$.

\paragraph{\textbf{Case 1: }$R_*=R_1$ and $R_{**}=R_2$:}

Since the catalytic instances $R_1 \in \phi(e_{10})$ and $R_2 \in \phi(e_{20})$ leave $\mathcal{Z}_1$, we have that there are no critical autocatalytic reactions, i.e., $\mathcal{R}_C = \emptyset$. However, since the removal of these two instances disconnects $\mathcal{Z}_1$ from $0$, we have that $\mathcal{Z}_1$ is a critical strongly connected component, i.e., $\mathcal{Z}_C = \{ \mathcal{Z}_1 \}.$ By Algorithm \ref{alg:1}, we must select an instance of an autocatalytic reaction with its source in $\mathcal{Z}_1$, remove its autocatalytic instance from $\mathcal{G}_{-V}$, and add it to $\mathcal{G}_F$. Choosing $R_2 \in \phi(e_{22})$ gives the following ALT-graph and decomposition $\mathcal{G} = \mathcal{G}_F \oplus \mathcal{G}_{-V}$:
    

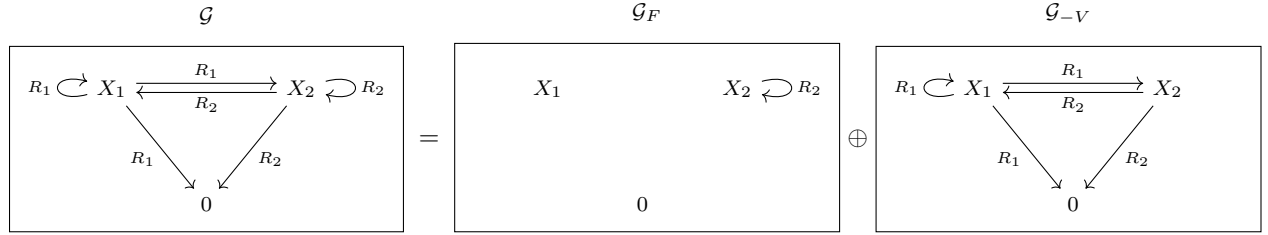
\[\footnotesize \begin{tikzcd}[ampersand replacement=\&,execute at end picture={
        \draw[black, thin] 
            ([shift={(-5pt,-5pt)}]current bounding box.south west) 
            rectangle 
            ([shift={(5pt,5pt)}]current bounding box.north east);                  
            \node[above] at ([shift={(0,5pt)}]current bounding box.north) {$\mathcal{G}$};
    }
]
X_1 \arrow[loop left,"R_1"] \arrow[rr,yshift=0.5ex,"R_1"] \arrow[rdd,"R_1"'] \& \& X_2 \arrow[ll,yshift=-0.5ex,"R_2"] \arrow[ldd,"R_2"] \arrow[loop right,"R_2"] \\
\& \& \\
\& 0 \&
\end{tikzcd}
\; = \;
            \begin{tikzcd}[ampersand replacement=\&,execute at end picture={
        \draw[black, thin] 
            ([shift={(-5pt,-5pt)}]current bounding box.south west) 
            rectangle 
            ([shift={(5pt,5pt)}]current bounding box.north east);
            \node[above] at ([shift={(0,5pt)}]current bounding box.north) {$\mathcal{G}_F$};
    }
]
           \&X_1\arrow[rr,white,yshift=0.1cm,"R_1"]    \& \& X_2 \arrow[loop right,"R_2"]\\
\& \& \\
\& \&0 \&
            \end{tikzcd}
\; \oplus \;
\begin{tikzcd}[ampersand replacement=\&,execute at end picture={
        \draw[black, thin] 
            ([shift={(-5pt,-5pt)}]current bounding box.south west) 
            rectangle 
            ([shift={(5pt,5pt)}]current bounding box.north east);
            \node[above] at ([shift={(0,5pt)}]current bounding box.north) {$\mathcal{G}_{-V}$};
    }]
            X_1 \arrow[loop left,"R_1"] \arrow[rr,yshift=0.5ex,"R_1"] \arrow[rdd,"R_1"'] \& \& X_2 \arrow[ll,yshift=-0.5ex,"R_2"] \arrow[ldd,"R_{2}"] \& \\
\& \& \\
\& 0 \&
            \end{tikzcd}
\]
The ALT-graph $\mathcal{G}_{-V}$ has no critical strongly connected component and no critical autocatalytic reactions. Consequently, we can apply Theorem \ref{thm:1} with $L(\mathcal{G}_{-V})=-V$ and $L(\mathcal{G}_F)=F$, which gives $\rho(FV^{-1}) = 1$. Consequently, the local stability of the boundary steady state is inconclusive.

To understand the importance of the critical strongly connected component $\mathcal{Z}_C = \{ \mathcal{Z}_1\}$ in this example, consider the corresponding ALT-matrices:
\[
L(\mathcal{G})=\begin{bmatrix}
    -\partial_1R_1 & \partial_2R_2 \\ \partial_1R_1 & -\partial_2R_2
\end{bmatrix}, \hspace{0.5in}
L(\mathcal{G_{F}})=\begin{bmatrix}
    0 & 0 \\ 0 & \partial_2R_2
\end{bmatrix}, \hspace{0.5in}
L(\mathcal{G_{-V}})=\begin{bmatrix}
    -\partial_{1}R_1 &\partial_2R_2 \\\partial_1R_1 & -2\partial_2R_2
\end{bmatrix}
\]
We note that the ALT-matrix of the network $\mathcal{G}$---which corresponds to the Jacobian---is not invertible and consequently Theorem \ref{thm:1} cannot be applied directly. This singularity is a direct consequence of the existence of the critical strongly connected component $\mathcal{Z}_1$ in $\mathcal{G}$. Specifically, the catalytic instances with $0$ as its target ($R_1 \in \phi(e_{10})$ and $R_2 \in \phi(e_{20})$) do not contribute to the $\mathcal{Z}_1$ block of $L(\mathcal{G})$ because their negative contribution is balanced exactly by the corresponding autocatalytic instances ($R_1 \in \phi(e_{11})$ and $R_2 \in \phi(e_{22})$). As a result, even though these instances appear as a leak in the graph $\mathcal{G}$, they do not contribute a leak term to the ALT-matrix. Removing either one of the autocatalytic instances $R_1 \in \phi(e_{11})$ or $R_2 \in \phi(e_{22})$ solves this problem by allowing the corresponding catalytic instance to behave as a leak in both the ALT-graph $\mathcal{G}_{-V}$ and ALT-matrix $L(\mathcal{G}_{-V})$. Note that the rules $\rho(FV^{-1})$ is consistent with the original matrix $J$ having an eigenvalue of zero, corresponding to inconclusive local stability. 

\paragraph{\textbf{Case 2: }$R_* = R_1$ and $R_{**} = R_3$:}

Since the catalytic instance $R_1 \in \phi(e_{10})$ leaves $\mathcal{Z}_1$, $R_1$ is not a critical autocatalytic reaction. However, since the sole catalytic instance $R_2 \in \phi(e_{21})$ of $R_2 \in \mathcal{R}_A$ traverses $\mathcal{Z}_1$, we have that $R_2$ is a critical autocatalytic reaction. Consequently, we have $\mathcal{R}_C = \{ R_2 \}$. Since the instance $R_3 \in \phi(e_{20})$ does not correspond to an autocatalytic reaction and therefore may not be removed, it is impossible to disconnect $\mathcal{Z}_1$ from $0$. Consequently, there are no critical strongly connected components so that $\mathcal{Z}_C = \emptyset$. by Algorithm \ref{alg:1}, we select the critical autocatalytic reaction $R_2 \in \mathcal{R}_C$, remove the catalytic instance $R_2 \in \phi(e_{21})$ and autocatalytic instance $R_2 \in \phi(e_{22})$ from $\mathcal{G}_{-V}$, and add them to $\mathcal{G}_F$. The resulting ALT-graph and decomposition $\mathcal{G} = \mathcal{G}_F \oplus \mathcal{G}_{-V}$ are:


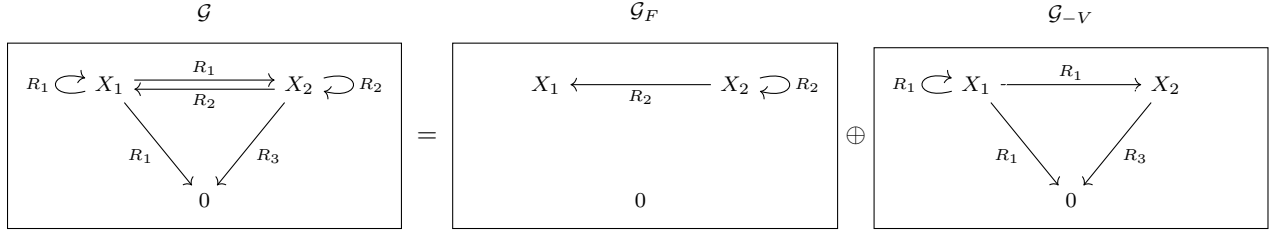
\[\footnotesize \begin{tikzcd}[ampersand replacement=\&,execute at end picture={
        \draw[black, thin] 
            ([shift={(-5pt,-5pt)}]current bounding box.south west) 
            rectangle 
            ([shift={(5pt,5pt)}]current bounding box.north east);                  
            \node[above] at ([shift={(0,5pt)}]current bounding box.north) {$\mathcal{G}$};
    }
]
X_1 \arrow[loop left,"R_1"] \arrow[rr,yshift=0.5ex,"R_1"] \arrow[rdd,"R_1"'] \& \& X_2 \arrow[ll,yshift=-0.5ex,"R_2"] \arrow[ldd,"R_3"] \arrow[loop right,"R_2"] \\
\& \& \\
\& 0 \&
\end{tikzcd}
\; = \;
            \begin{tikzcd}[ampersand replacement=\&,execute at end picture={
        \draw[black, thin] 
            ([shift={(-5pt,-5pt)}]current bounding box.south west) 
            rectangle 
            ([shift={(5pt,5pt)}]current bounding box.north east);
            \node[above] at ([shift={(0,5pt)}]current bounding box.north) {$\mathcal{G}_F$};
    }
]
           \&X_1   \arrow[rr,white,yshift=0.5ex,"R_1"]  \& \& X_2 \arrow[ll,"R_2"] \arrow[loop right,"R_2"]\\
\& \& \\
\& \&0 \&
            \end{tikzcd}
\; \oplus \;
 \begin{tikzcd}[ampersand replacement=\&,execute at end picture={
        \draw[black, thin] 
            ([shift={(-5pt,-5pt)}]current bounding box.south west) 
            rectangle 
            ([shift={(5pt,5pt)}]current bounding box.north east);
            \node[above] at ([shift={(0,5pt)}]current bounding box.north) {$\mathcal{G}_{-V}$};
    }]
            X_1 \arrow[loop left,"R_1"] \arrow[rr,"R_1"] \arrow[rdd,"R_1"'] \& \& X_2 \arrow[ll,white,yshift=-0.5ex,"R_2"] \arrow[loop right,white,"R_2"] \arrow[ldd,"R_{3}"] \& \\
\& \& \\
\& 0 \&
            \end{tikzcd}
\]
The resulting $\mathcal{G}_{-V}$ graph has no critical autocatalytic reactions or critical strongly connected components so Algorithm \ref{alg:1} terminates. By Theorem \ref{thm:1}, the ALT-matrices $L(\mathcal{G})=J$, $L(\mathcal{G}_{-V})=-V$ and $L(\mathcal{G}_F)=F$, therefore, we can compute the basic reproduction number, which is $\rho(FV^{-1})=\frac{\partial_2R_2}{\partial_2R_3}$ so that the boundary steady state will be locally asymptotically stable if $\partial_2R_3>\partial_2R_2$.

To consider the importance of the critical autocatalytic reaction set $\mathcal{R}_C = \{ R_2 \}$, consider the corresponding ALT-matrices:
\[ L(\mathcal{G}) = \begin{bmatrix}
-\partial_{1}R_1 & \partial_2R_2 \\ \partial_1R_1 & -\partial_2R_3
\end{bmatrix}, \hspace{0.5in}
L(\mathcal{G_{F}})=\begin{bmatrix}
    0 & \partial_2R_2 \\ 0 & 0
\end{bmatrix}, \hspace{0.5in}
L(\mathcal{G_{-V}})=\begin{bmatrix}
    -\partial_{1}R_1 & \partial_2R_2 \\ 0 & - \partial_2R_3
\end{bmatrix} 
\]
The ALT-matrix $L(\mathcal{G})$ is invertible unless $\partial_2R_3=\partial_2R_2$; however, if $\partial_2R_3<\partial_2R_2$, then it will not be diagonally column dominant and Theorem \ref{thm:1} cannot be applied. This lack of diagonal column dominance is a result of the critical autocatalytic reaction $R_2 \in \mathcal{R}_C$. Specifically, the combined catalytic instance $R_2 \in \phi(e_{21})$ and autocatalytic instance $R_2 \in \phi(e_{22})$ produce an off-diagonal positive entry in the $(1,2)$ entry of $L(\mathcal{G})$ but the entries along the diagonal cancel. Diagonal column dominance can therefore not be guaranteed. Moving both the catalytic and autocatalytic instances ($R_2 \in \phi(e_{21})$ and $R_2 \in \phi(e_{22})$) removes this problem by eliminating the possible violation to diagonal column dominance in $L(\mathcal{G}_{-V})$. We note that, if $\partial_2R_3>\partial_2R_2$, then the autocatalytic reaction $R_2$ is not a problem for invertibility or diagonal column dominance and Theorem \ref{thm:ngm} can be applied directly. 
The possibility of using conditional diagonal column dominance to save some computational time will be studied in future work.

\paragraph{\textbf{Case 3: }$R_*=R_3$ and $R_{**}=R_4$:}
    
All of the catalytic instances of autocatalytic reactions ($R_1 \in \phi(e_{12})$ and $R_2 \in \phi(e_{21})$) are contained in $\mathcal{Z}_1$, so that they are both critical, i.e., $\mathcal{R}_C = \{ R_1, R_2\}.$ Since the instances $R_3 \in \phi(e_{10})$ and $R_4 \in \phi(e_{20})$ connect $\mathcal{Z}_1$ to $0$ and the underlying reactions are not autocatalytic, there are no critical strongly connected components, i.e., $\mathcal{Z}_C = \emptyset$. By Algorithm \ref{alg:1}, we select the critical autocatalytic reaction $R_2 \in \mathcal{R}_C$ and remove the catalytic instance $R_2 \in \phi(e_{21})$ and the autocatalytic instance $R_2 \in \phi(e_{22})$ from $\mathcal{G}_{-V}$ and add them to $\mathcal{G}_F$. The resulting ALT-graph and decomposition $\mathcal{G} = \mathcal{G}_F \oplus \mathcal{G}_{-V}$ are:

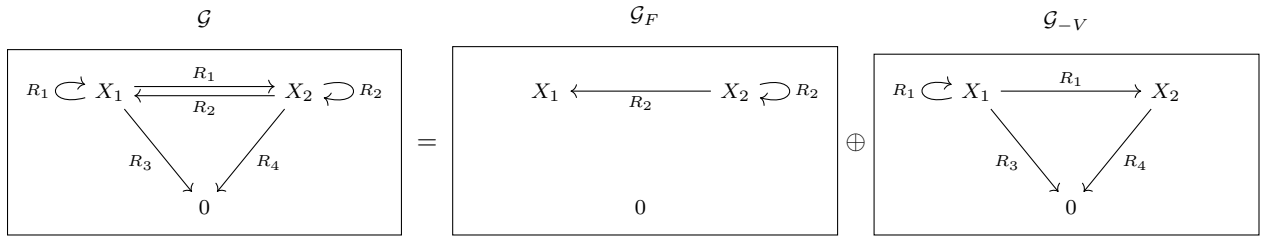
\[\footnotesize \begin{tikzcd}[ampersand replacement=\&,execute at end picture={
        \draw[black, thin] 
            ([shift={(-5pt,-5pt)}]current bounding box.south west) 
            rectangle 
            ([shift={(5pt,5pt)}]current bounding box.north east);                  
            \node[above] at ([shift={(0,5pt)}]current bounding box.north) {$\mathcal{G}$};
    }
]
X_1 \arrow[loop left,"R_1"] \arrow[rr,yshift=0.5ex,"R_1"] \arrow[rdd,"R_3"'] \& \& X_2 \arrow[ll,yshift=-0.5ex,"R_2"] \arrow[ldd,"R_4"] \arrow[loop right,"R_2"] \\
\& \& \\
\& 0 \&
\end{tikzcd}
\; = \;
\begin{tikzcd}[ampersand replacement=\&,execute at end picture={
        \draw[black, thin] 
            ([shift={(-5pt,-5pt)}]current bounding box.south west) 
            rectangle 
            ([shift={(5pt,5pt)}]current bounding box.north east);
            \node[above] at ([shift={(0,5pt)}]current bounding box.north) {$\mathcal{G}_F$};
    }
]
           \&X_1\arrow[rr,white,yshift=0.1cm,"R_1"]    \& \& X_2 \arrow[loop right,"R_2"] \arrow[ll,"R_2"]\\
\& \& \\
\& \&0 \&
            \end{tikzcd}
\; \oplus \;
\begin{tikzcd}[ampersand replacement=\&,execute at end picture={
        \draw[black, thin] 
            ([shift={(-5pt,-5pt)}]current bounding box.south west) 
            rectangle 
            ([shift={(5pt,5pt)}]current bounding box.north east);
            \node[above] at ([shift={(0,5pt)}]current bounding box.north) {$\mathcal{G}_{-V}$};
    }]
            \& X_1  \arrow[loop left,"R_1"]
    \arrow[rr,"R_1"] 
    \arrow[rdd,"R_3"'] 
\& \& 
X_2 \arrow[ldd,"R_{4}"] 
\& \\
\& \& \& \& \\
\& \& 0 \& \&
            \end{tikzcd}
\]
The resulting $\mathcal{G}_{-V}$ graph has no critical strongly connected components or critical autocatalytic reactions so Algorithm \ref{alg:1} terminates. Applying Theorem \ref{thm:1} with $L(\mathcal{G}_{-V})=-V$ and $L(\mathcal{G}_F)=F$ gives the reproduction number $\rho(FV^{-1})=\frac{\partial_1R_1\partial_2R_2}{\partial_1R_3\partial_2R_4}$. Therefore, the boundary steady state will be stable if $\partial_{1}R_{3}\partial_2R_4>\partial_1R_1\partial_2R_2$ and unstable if $\partial_{1}R_{3}\partial_2R_4<\partial_1R_1\partial_2R_2$.

To consider the importance of the critical autocatalytic reaction set $\mathcal{R}_C = \{ R_1, R_2 \}$, consider the corresponding ALT-matrices:
\[L(\mathcal{G})=\begin{bmatrix}
    -\partial_1R_3 & \partial_2R_2 \\ \partial_1R_1 & -\partial_2R_4
\end{bmatrix}, \hspace{0.5in}
L(\mathcal{G_{F}})=\begin{bmatrix}
    0 & \partial_2R_2 \\ 0 & 0
\end{bmatrix}, \hspace{0.5in}
L(\mathcal{G_{-V}})=\begin{bmatrix}
    -\partial_1R_3 & 0 \\\partial_1R_1 & -\partial_2R_4
\end{bmatrix}
\]
The ALT-matrix $L(\mathcal{G})$ is invertible unless $\partial_1R_3 \partial_2 R_4 = \partial_1R_1 \partial_2R_2$; however, diagonal column dominance is not guaranteed for either column. This is a result of the critical autocatalytic reactions $R_1 \in \mathcal{R}_C$ and $R_2 \in \mathcal{R}_C$. Specifically, the combination of catalytic and autocatalytic instances leaves one positive entry on the off-diagonal of $L(\mathcal{G})$ while the entries cancel on the diagonal of $L(\mathcal{G})$. Diagonal column dominance is consequently not guaranteed. We can see, however, that it is unnecessary to split both $R_1 \in \mathcal{R}_C$ and $R_2 \in \mathcal{R}_C$ from $\mathcal{G}$. After removing the instances $R_2 \in \phi(e_{21})$ and $R_2 \in \phi(e_{22})$, we have a new strongly connected component structure in $\mathcal{G}_{-V}$, namely $\mathcal{Z}_1 = \{ X_1 \}$, $\mathcal{Z}_2 = \{ X_2 \}$, and $\mathcal{Z}_3 = \{ 0 \}$. Although $R_1$ is still autocatalytic (i.e., $\mathcal{R}_A = \{ R_1 \}$), the catalytic instance $R_1 \in \phi(e_{12})$ traverses $\mathcal{Z}_1$ and $\mathcal{Z}_2$ so that $\mathcal{R}_C = \emptyset$. Furthermore, the instance $R_3 \in \phi(e_{10})$ of the nonautocatalytic reaction $R_3 \not\in \mathcal{R}_A$ prevents disconnection of $\mathcal{Z}_1$ and $0$ so that $\mathcal{Z}_C = \emptyset$. It follows that Theorem \ref{thm:1} may be applied. Importantly, we do not need to split all critical autocatalytic reactions as a careful choice of splitting may alter the network structure in a way which removes criticality. This case emphasizes the importance of performing Algorithm \ref{alg:1} iteratively, one splitting at a time.

\paragraph{\textbf{Case 4: }$R_*=\{R_1,R_3\}$ and $R_{**}=R_2$:}

Since there are catalytic instances $R_1 \in \phi(e_{10})$ and $R_2 \in \phi(e_{20})$ of each autocatalytic reaction which leave $\mathcal{Z}_1$, we have that there are no critical autocatalytic reactions, i.e., $\mathcal{R}_C = \emptyset$. Furthermore, since the instance $R_3 \in \phi(e_{10})$ connects $\mathcal{Z}_1$ to $0$ but $R_3$ is not an autocatalytic reaction, there is no way to disconnect $\mathcal{Z}_1$ from $0$ by removal of instances of autocatalytic reactions. Consequently, we have that there are no critical strongly connected components, i.e., $\mathcal{Z}_C = \emptyset$. 
It follows that the boundary steady state is universally locally asymptotically stable by Corollary \ref{cor:2}. The ALT-graph and corresponding ALT-matrix are given by the following:
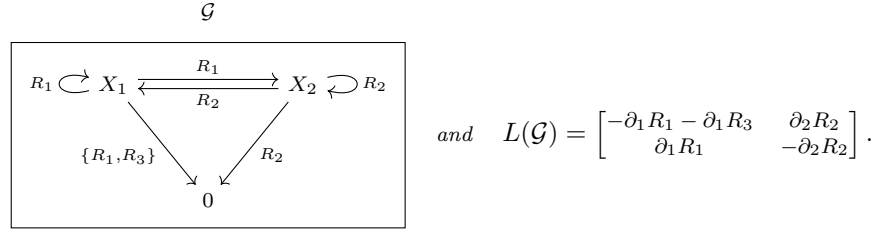
\[\footnotesize \begin{tikzcd}[ampersand replacement=\&,execute at end picture={
        \draw[black, thin] 
            ([shift={(-5pt,-5pt)}]current bounding box.south west) 
            rectangle 
            ([shift={(5pt,5pt)}]current bounding box.north east);                  
            \node[above] at ([shift={(0,5pt)}]current bounding box.north) {$\mathcal{G}$};
    }
]
X_1 \arrow[loop left,"R_1"] \arrow[rr,yshift=0.5ex,"R_1"] \arrow[rdd,"{\{R_{1}, R_3\}}"'] \& \& X_2 \arrow[ll,yshift=-0.5ex,"R_2"] \arrow[ldd,"R_2"] \arrow[loop right,"R_2"] \\
\& \& \\
\& 0 \&
\end{tikzcd} \; \; \; \mbox{ and } \; \; \; L(\mathcal{G})=\begin{bmatrix}
    -\partial_1R_1-\partial_1R_3 & \partial_2R_2 \\ \partial_1R_1 & -\partial_2R_2
\end{bmatrix}.
\]
We can see that, despite the autocatalytic reactions in $\mathcal{G}$, $J = L(\mathcal{G})$ is already invertible and diagonally column dominant so that no splittings are required before applying Theorem \ref{thm:ngm}.

\end{example}

\section{Splitting Sequence Examples}
\label{app:g}
\begin{example}
The following example shows the fact that Algorithm \ref{alg:1} allows different splitting options in the same step, meaning that the same ALT-graph $\mathcal{G}$ can be decomposed differently depending on what type of splitting we perform, or what specific autocatalytic reaction we choose first. One specific decomposition of $\mathcal{G}$ may yield a lower rank for $F$ than another one, which implies less computational complexity to calculate the basic reproduction number. Usually, choosing to perform a catalytic splitting rather than an autocatalytic one minimizes the number of steps required in the decomposition, as it can break the strongly connected components apart. See the following ALT-graph:
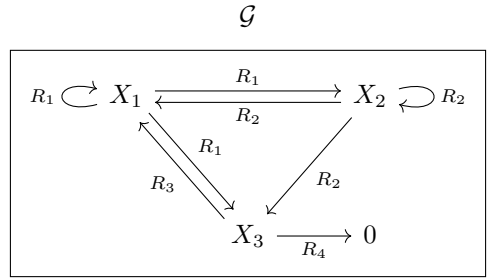
\[\begin{tikzcd}[ampersand replacement=\&,execute at end picture={
        \draw[black, thin] 
            ([shift={(-5pt,-5pt)}]current bounding box.south west) 
            rectangle 
            ([shift={(5pt,5pt)}]current bounding box.north east);                  
            \node[above] at ([shift={(0,5pt)}]current bounding box.north) {$\mathcal{G}$};
    }
]
X_1
    \arrow[loop left,"R_1"]
    \arrow[rr,yshift=0.5ex,"R_1"]
    \arrow[rdd,xshift=0.4ex,yshift=0.4ex,"R_1"]
\&
\&
X_2
    \arrow[ll,yshift=-0.5ex,"R_2"]
    \arrow[loop right,"R_2"]
    \arrow[ldd,"R_2"]
\\
\&
\& 
\\
\&
X_3
    \arrow[luu,xshift=-0.4ex,yshift=-0.4ex,"R_3"]
    \arrow[r,"R_4"']
\&
0
\end{tikzcd}
\]
There are the strongly connected components $\mathcal{Z}_1 = \{ X_1, X_2, X_3\}$ and $\mathcal{Z}_2 = \{ 0 \}$, neither of which are critical so that $\mathcal{Z}_C = \emptyset$. There are also two autocatalytic reactions $R_1$ and $R_2$, both of which are critical so that $\mathcal{R}_C = \{ R_1, R_2 \}$. According to Algorithm \ref{alg:1}, we remove a catalytic instance of a critical autocatalytic reaction from $\mathcal{G}$ and place it in $\mathcal{G}_F$. We now explore different scenarios based on what specific splitting we choose to perform in the first step of the decomposition.
\paragraph{\textbf{Autocatalytic Splitting:}}
The autocatalytic instance $R_1 \in \phi(e_{11})$ is removed from $\mathcal{G_{-V}}$ and added to $\mathcal{G}_F$. The resulting ALT-graphs are:
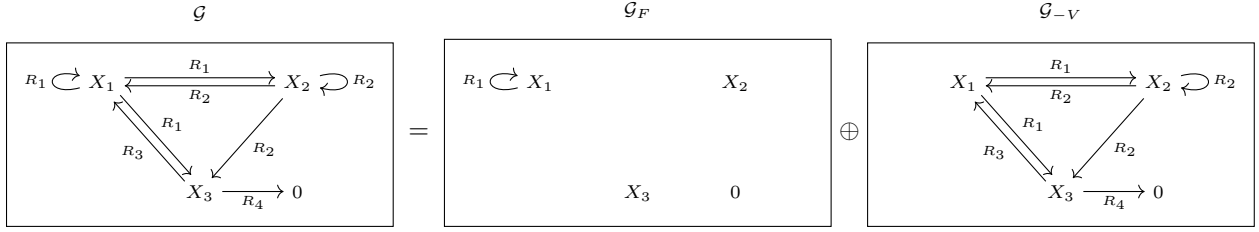
\[\scriptsize \begin{tikzcd}[ampersand replacement=\&,execute at end picture={
        \draw[black, thin] 
            ([shift={(-5pt,-5pt)}]current bounding box.south west) 
            rectangle 
            ([shift={(5pt,5pt)}]current bounding box.north east);                  
            \node[above] at ([shift={(0,5pt)}]current bounding box.north) {$\mathcal{G}$};
    }
]
X_1
    \arrow[loop left,"R_1"]
    \arrow[rr,yshift=0.5ex,"R_1"]
    \arrow[rdd,xshift=0.4ex,yshift=0.4ex,"R_1"]
\&
\&
X_2
    \arrow[ll,yshift=-0.5ex,"R_2"]
    \arrow[loop right,"R_2"]
    \arrow[ldd,"R_2"]
\\
\&
\& 
\\
\&
X_3
    \arrow[luu,xshift=-0.4ex,yshift=-0.4ex,"R_3"]
    \arrow[r,"R_4"']
\&
0
\end{tikzcd}
\; = \;
            \begin{tikzcd}[ampersand replacement=\&,execute at end picture={
        \draw[black, thin] 
            ([shift={(-5pt,-5pt)}]current bounding box.south west) 
            rectangle 
            ([shift={(5pt,5pt)}]current bounding box.north east);
            \node[above] at ([shift={(0,5pt)}]current bounding box.north) {$\mathcal{G}_F$};
    }
]
           X_1
    \arrow[loop left,"R_1"]
    \arrow[rr,yshift=0.1cm,white,"R_1"]
    \arrow[rdd,xshift=0cm,yshift=0cm,white,"R_1"]
\&
\&
X_2
    \arrow[ll,yshift=-0.1cm,white,"R_2"]
    \arrow[loop right,white,"R_2"]
    \arrow[ldd,white,"R_2"]
\\
\&
\& 
\\
\&
X_3
    \arrow[luu,xshift=-0.1cm,yshift=-0.1cm,white,"R_3"]
    \arrow[r,white,"R_4"']
\&
0
            \end{tikzcd}
\; \oplus \;
 \begin{tikzcd}[ampersand replacement=\&,execute at end picture={
        \draw[black, thin] 
            ([shift={(-5pt,-5pt)}]current bounding box.south west) 
            rectangle 
            ([shift={(5pt,5pt)}]current bounding box.north east);
            \node[above] at ([shift={(0,5pt)}]current bounding box.north) {$\mathcal{G}_{-V}$};
    }]
            X_1
    \arrow[loop left,white,"R_1"]
    \arrow[rr,yshift=0.5ex,"R_1"]
    \arrow[rdd,xshift=0.4ex,yshift=0.4ex,"R_1"]
\&
\&
X_2
    \arrow[ll,yshift=-0.5ex,"R_2"]
    \arrow[loop right,"R_2"]
    \arrow[ldd,"R_2"]
\\
\&
\& 
\\
\&
X_3
    \arrow[luu,xshift=-0.4ex,yshift=-0.4ex,"R_3"]
    \arrow[r,"R_4"']
\&
0
            \end{tikzcd}
\]
After this first autocatalytic splitting, $\mathcal{G}_{-V}$ still has the critical autocatalytic reaction $R_2$, and the single strongly connected component is not critical. For the next step of the decomposition, there are 3 different options: perform an autocatalytic splitting with $R_2 \in \phi_{-V}(e_{22})$, perform a catalytic splitting with $R_2\in\phi_{-V}(e_{22})\cap\phi_{-V}(e_{23})$, or perform a catalytic splitting with $R_2\in\phi_{-V}(e_{22})\cap\phi_{-V}(e_{21})$. The first two will produce a rank-2 $F$ matrix, while the third one will produce a rank-1 $F$ matrix (because both splittings contribute to the same row, as shown in the common target vertex heuristic of Algorithm \ref{alg:1}). After this second splitting, $\mathcal{G}_{-V}$ will have no critical autocatalytic reactions or critical strongly connected components. Therefore, for this specific scenario, at least two splittings need to be made in order to satisfy the requirements of Theorem \ref{thm:1}.
\paragraph{\textbf{Catalytic Splitting 1:}}
The autocatalytic instance $R_1$ and the catalytic instance $R_1\in\phi_{-V}(e_{13})$ are moved from $\mathcal{G_{-V}}$ to $\mathcal{G}_F$. The resulting ALT-graphs are:
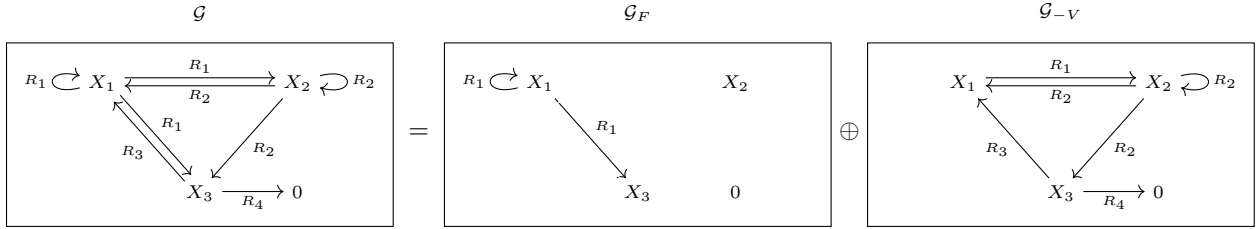
\[\scriptsize \begin{tikzcd}[ampersand replacement=\&,execute at end picture={
        \draw[black, thin] 
            ([shift={(-5pt,-5pt)}]current bounding box.south west) 
            rectangle 
            ([shift={(5pt,5pt)}]current bounding box.north east);                  
            \node[above] at ([shift={(0,5pt)}]current bounding box.north) {$\mathcal{G}$};
    }
]
X_1
    \arrow[loop left,"R_1"]
    \arrow[rr,yshift=0.5ex,"R_1"]
    \arrow[rdd,xshift=0.4ex,yshift=0.4ex,"R_1"]
\&
\&
X_2
    \arrow[ll,yshift=-0.5ex,"R_2"]
    \arrow[loop right,"R_2"]
    \arrow[ldd,"R_2"]
\\
\&
\& 
\\
\&
X_3
    \arrow[luu,xshift=-0.4ex,yshift=-0.4ex,"R_3"]
    \arrow[r,"R_4"']
\&
0
\end{tikzcd}
\; = \;
            \begin{tikzcd}[ampersand replacement=\&,execute at end picture={
        \draw[black, thin] 
            ([shift={(-5pt,-5pt)}]current bounding box.south west) 
            rectangle 
            ([shift={(5pt,5pt)}]current bounding box.north east);
            \node[above] at ([shift={(0,5pt)}]current bounding box.north) {$\mathcal{G}_F$};
    }
]
           X_1
    \arrow[loop left,"R_1"]
    \arrow[rr,yshift=0.5ex,white,"R_1"]
    \arrow[rdd,"R_1"]
\&
\&
X_2
    \arrow[ll,yshift=-0.5ex,white,"R_2"]
    \arrow[loop right,white,"R_2"]
    \arrow[ldd,white,"R_2"]
\\
\&
\& 
\\
\&
X_3
    \arrow[luu,xshift=-0.4ex,yshift=-0.4ex,white,"R_3"]
    \arrow[r,white,"R_4"']
\&
0
            \end{tikzcd}
\; \oplus \;
 \begin{tikzcd}[ampersand replacement=\&,execute at end picture={
        \draw[black, thin] 
            ([shift={(-5pt,-5pt)}]current bounding box.south west) 
            rectangle 
            ([shift={(5pt,5pt)}]current bounding box.north east);
            \node[above] at ([shift={(0,5pt)}]current bounding box.north) {$\mathcal{G}_{-V}$};
    }]
            X_1
    \arrow[loop left,white,"R_1"]
    \arrow[rr,yshift=0.5ex,"R_1"]
    \arrow[rdd,xshift=0.4ex,yshift=0.4ex,white,"R_1"]
\&
\&
X_2
    \arrow[ll,yshift=-0.5ex,"R_2"]
    \arrow[loop right,"R_2"]
    \arrow[ldd,"R_2"]
\\
\&
\& 
\\
\&
X_3
    \arrow[luu,"R_3"]
    \arrow[r,"R_4"']
\&
0
            \end{tikzcd}
\]
After this first catalytic splitting, $\mathcal{G}_{-V}$ still has the critical autocatalytic reaction $R_2$, and the single strongly connected component is not critical. For the next step of the decomposition, there are 3 different options: perform an autocatalytic splitting with $R_2 \in \phi_{-V}(e_{22})$, perform a catalytic splitting with $R_2\in\phi_{-V}(e_{22})\cap\phi_{-V}(e_{23})$ or perform a catalytic splitting with $R_2\in\phi_{-V}(e_{22})\cap\phi_{-V}(e_{21})$. The first one will produce a rank-2 $F$ matrix, the second one will disconnect the new strongly connected component $\mathcal{Z}_1=\{X_1,X_2\}$ from zero, meaning that this splitting is not valid, and the third one will also produce a rank-2 $F$ matrix. Only the first and third options will give a decomposition that guarantees connectivity and satisfies the requirements of Theorem \ref{thm:1}. We can see how the catalytic preference heuristic of Algorithm \ref{alg:1} is playing an important role in this scenario, as removing both autocatalytic and catalytic instances can be beneficial for the splitting, unless it disconnects the source from zero (as happens in the second option). Therefore, choosing which catalytic edge to remove is very important.
\paragraph{\textbf{Catalytic Splitting 2:}}
The autocatalytic instance $R_1$ and the catalytic instance $R_1\in\phi_{-V}(e_{12})$ are moved from $\mathcal{G_{-V}}$ to $\mathcal{G}_F$. The resulting ALT-graphs are:
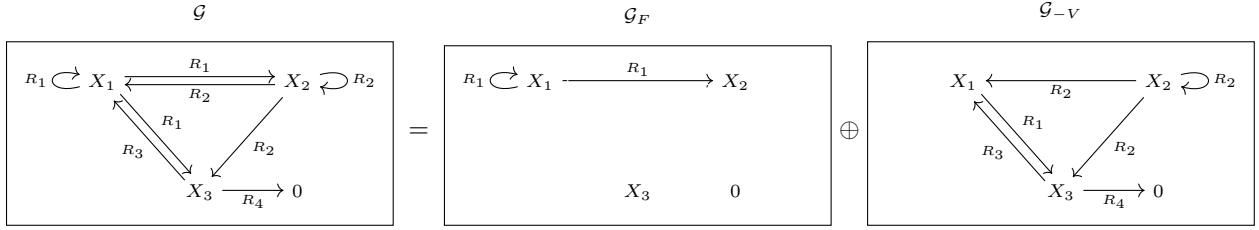
\[\scriptsize \begin{tikzcd}[ampersand replacement=\&,execute at end picture={
        \draw[black, thin] 
            ([shift={(-5pt,-5pt)}]current bounding box.south west) 
            rectangle 
            ([shift={(5pt,5pt)}]current bounding box.north east);                  
            \node[above] at ([shift={(0,5pt)}]current bounding box.north) {$\mathcal{G}$};
    }
]
X_1
    \arrow[loop left,"R_1"]
    \arrow[rr,yshift=0.5ex,"R_1"]
    \arrow[rdd,xshift=0.4ex,yshift=0.4ex,"R_1"]
\&
\&
X_2
    \arrow[ll,yshift=-0.5ex,"R_2"]
    \arrow[loop right,"R_2"]
    \arrow[ldd,"R_2"]
\\
\&
\& 
\\
\&
X_3
    \arrow[luu,xshift=-0.4ex,yshift=-0.4ex,"R_3"]
    \arrow[r,"R_4"']
\&
0
\end{tikzcd}
\; = \;
            \begin{tikzcd}[ampersand replacement=\&,execute at end picture={
        \draw[black, thin] 
            ([shift={(-5pt,-5pt)}]current bounding box.south west) 
            rectangle 
            ([shift={(5pt,5pt)}]current bounding box.north east);
            \node[above] at ([shift={(0,5pt)}]current bounding box.north) {$\mathcal{G}_F$};
    }
]
           X_1
    \arrow[loop left,"R_1"]
    \arrow[rr,"R_1"]
    \arrow[rdd,xshift=0.4ex,yshift=0.4ex,white,"R_1"]
\&
\&
X_2
    \arrow[ll,yshift=-0.5ex,white,"R_2"]
    \arrow[loop right,white,"R_2"]
    \arrow[ldd,white,"R_2"]
\\
\&
\& 
\\
\&
X_3
    \arrow[luu,xshift=-0.5ex,yshift=-0.5ex,white,"R_3"]
    \arrow[r,white,"R_4"']
\&
0
            \end{tikzcd}
\; \oplus \;
 \begin{tikzcd}[ampersand replacement=\&,execute at end picture={
        \draw[black, thin] 
            ([shift={(-5pt,-5pt)}]current bounding box.south west) 
            rectangle 
            ([shift={(5pt,5pt)}]current bounding box.north east);
            \node[above] at ([shift={(0,5pt)}]current bounding box.north) {$\mathcal{G}_{-V}$};
    }]
            X_1
    \arrow[loop left,white,"R_1"]
    \arrow[rr,yshift=0.5ex,white,"R_1"]
    \arrow[rdd,xshift=0.4ex,yshift=0.4ex,"R_1"]
\&
\&
X_2
    \arrow[ll,yshift=0cm,"R_2"]
    \arrow[loop right,"R_2"]
    \arrow[ldd,"R_2"]
\\
\&
\& 
\\
\&
X_3
    \arrow[luu,xshift=-0.5ex,yshift=-0.5ex,"R_3"]
    \arrow[r,"R_4"']
\&
0
            \end{tikzcd}
\]
After this first catalytic splitting, $\mathcal{G}_{-V}$ has broken its single strongly connected component into two, which are  $\mathcal{Z}_1=\{X_1,X_3\}$ and $\mathcal{Z}_2=\{X_2\}$. As we can see, none of them are critical, and the autocatalytic reaction $R_2$ is not critical anymore. Therefore, with only one catalytic splitting, we get a rank-1 $F$ matrix and a valid decomposition that satisfies the conditions of Theorem \ref{thm:1}. This is the decomposition that minimizes the rank of $F$ with the least amount of steps.

If instead of choosing $R_1$ for the first step, we choose $R_2$, there is no possible decomposition that could give us a rank-1 $F$ matrix in less than two steps. So choosing what autocatalytic reaction to remove first can play an important role in the decomposition, minimizing the steps of the algorithm and therefore some computational time.
\end{example}


\section{Details of Examples in Section \ref{sec:examples}}
\label{app:h}

In this Appendix, we present computational details of Examples \ref{ex:envzompr}, \ref{ex:twosite}, \ref{ex:sir}, and \ref{ex:vectorborne} from the main text. To clearly illustrate the ALT-graph approach, we adopt throughout the siphon-based reindexing convention introduced in Section \ref{sec:siphon}.

\begin{example}
\label{ex:envzompr2}
We now present full computational details for the EnvZ-OmpR osmoregularity mechanism \eqref{envzompr} studied in Example \ref{ex:envzompr}. Recall that the network was introduced in \cite{shinar2010structural} and has been studied in the context of boundary stability in the papers \cite{avram2024advancing,johnston2026boundary}. 
The network \eqref{envzompr} is governed by the following mass action system:
\begin{equation}
\label{envzompr-des}
\left\{ \; \; \begin{aligned}
X_d' &= -R_1(X_d) + R_2(X) - R_9(X_d,Y_p) + R_{10}(X_dY_p) + R_{11}(X_dY_p) \\
X' &= R_1(X_d) - R_2(X) - R_3(X) + R_4(X_t) + R_8(X_pY) \\
X_t' &= R_3(X) - R_4(X_t) - R_5(X_t) - R_{12}(X_t,Y_p) + R_{13}(X_tY_p) + R_{14}(X_tY_p) \\
X_p' &= R_5(X_t) - R_6(X_p,Y) + R_7(X_pY) \\
Y' &= -R_6(X_p,Y) + R_7(X_pY) + R_{11}(X_dY_p) + R_{14}(X_tY_p) \\
X_pY' &= R_6(X_p,Y) - R_7(X_pY) - R_8(X_pY) \\
Y_p' &= R_8(X_pY) - R_9(X_d,Y_p) + R_{10}(X_dY_p) - R_{12}(X_t,Y_p) + R_{13}(X_tY_p) \\
X_dY_p' &= R_9(X_d,Y_p) - R_{10}(X_dY_p) - R_{11}(X_dY_p) \\
X_tY_p' &= R_{12}(X_t,Y_p) - R_{13}(X_tY_p) - R_{14}(X_tY_p).
\end{aligned} \right.
\end{equation} 
The system \eqref{envzompr-des} has the conservation laws
\[\begin{aligned}
        &X_{tot} = X_d+X+X_t+X_p+X_pY+X_dY_p+X_tY_p, \\
        &Y_{tot} = Y+X_pY+Y_p+X_dY_p+X_tY_p. 
\end{aligned}\]
and the boundary steady state $(X_d,X,X_t,X_p,Y,X_pY,Y_p,X_dY_p,X_tY_p)=(0,0,0,X_{tot},0,0,Y_{tot},0,0)$ corresponding to the siphon $\mathcal{X} = \{ X_d, X, X_t, Y, X_pY, X_dY_p, X_tY_p \}$.

We reindex the species as $X_1 = X_d, X_2 = X, X_3 = X_t, X_4 = Y, X_5 = X_pY, X_6 = X_dY_p, X_7 = X_tY_p, Y_1 = X_p,$ and $Y_2 = Y_p$. The network \eqref{envzompr} then becomes:
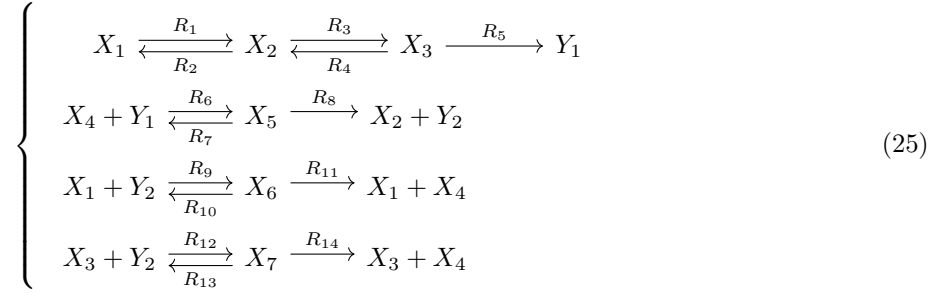
\begin{equation}
\label{envzompr-reindex}
\left\{ \; \; 
\begin{tikzcd}
X_1 \arrow[r,yshift=+0.5ex,"R_1"] & X_2 \arrow[l,yshift=-0.5ex,"R_2"] \arrow[r,yshift=+0.5ex,"R_3"]& X_3 \arrow[l,yshift=-0.5ex,"R_4"] \arrow[r,"R_5"]& Y_1\\[-0.1in]
X_4 + Y_1 \arrow[r,yshift=+0.5ex,"R_6"] & X_5 \arrow[l,yshift=-0.5ex,"R_7"] \arrow[r,yshift=+0.5ex,"R_8"]& X_2 + Y_2 & \\[-0.1in]
X_1 + Y_2 \arrow[r,yshift=+0.5ex,"R_9"] & X_6 \arrow[l,yshift=-0.5ex,"R_{10}"] \arrow[r,yshift=+0.5ex,"R_{11}"]& X_1 + X_4 &  \\[-0.1in]
X_3 + Y_2 \arrow[r,yshift=+0.5ex,"R_{12}"] & X_7 \arrow[l,yshift=-0.5ex,"R_{13}"] \arrow[r,yshift=+0.5ex,"R_{14}"]& X_3 + X_4 &
      \end{tikzcd} \right.
\end{equation}
The corresponding system \eqref{eq:crs2} restricted to $\mathcal{X}$ is given by
\[
\left\{ \; \; \begin{aligned}
X_1' &= -R_1(X_1) + R_2(X_2) - R_9(X_1,Y_2) + R_{10}(X_6) + R_{11}(X_6) \\
X_2' &= R_1(X_1) - R_2(X_2) - R_3(X_2) + R_4(X_3) + R_8(X_5) \\
X_3' &= R_3(X_2) - R_4(X_3) - R_5(X_3) - R_{12}(X_3,Y_2) + R_{13}(X_7) + R_{14}(X_7) \\
X_4' &= -R_6(X_4,Y_1) + R_7(X_5) + R_{11}(X_6) + R_{14}(X_7) \\
X_5' &= R_6(X_4,Y_1) - R_7(X_5) - R_8(X_5) \\
X_6' &= R_9(X_1,Y_2) - R_{10}(X_6) - R_{11}(X_6) \\
X_7' &= R_{12}(X_3,Y_2) - R_{13}(X_7) - R_{14}(X_7).
\end{aligned} \right.
\]

\begin{table}[t]
\centering
\begin{NiceTabular}{ccccccc}[code-before = \rowcolors{2}{white}{gray!12}]
\toprule
\begin{tabular}{c}Reaction\\ label \end{tabular} & & \begin{tabular}{c}Chemical \\ reaction \end{tabular} & & \begin{tabular}{c}$\mathcal{X}$-reduced\\reaction\end{tabular} & & \begin{tabular}{c} ALT-graph\\edge(s)\end{tabular} \\
\midrule
$R_1:$ & & \begin{tikzcd}[ampersand replacement=\&, baseline=(current bounding box.center)] X_1 \arrow[r] \& X_2 \end{tikzcd} & $\Longrightarrow$ & \begin{tikzcd}[ampersand replacement=\&, baseline=(current bounding box.center)] X_1 \arrow[r] \& X_2 \end{tikzcd} & $\Longrightarrow$ & \begin{tikzcd}[ampersand replacement=\&, baseline=(current bounding box.center)] X_1 \arrow[r,"{}_{R_1}"] \& X_2 \end{tikzcd} \\[0pt]
$R_2:$ & & \begin{tikzcd}[ampersand replacement=\&, baseline=(current bounding box.center)] X_2 \arrow[r] \& X_1 \end{tikzcd} & $\Longrightarrow$ & \begin{tikzcd}[ampersand replacement=\&, baseline=(current bounding box.center)] X_2 \arrow[r] \& X_1 \end{tikzcd} & $\Longrightarrow$ & \begin{tikzcd}[ampersand replacement=\&, baseline=(current bounding box.center)] X_2 \arrow[r,"{}_{R_2}"] \& X_1 \end{tikzcd} \\[0pt]
$R_3:$ & & \begin{tikzcd}[ampersand replacement=\&, baseline=(current bounding box.center)] X_2 \arrow[r] \& X_3 \end{tikzcd} & $\Longrightarrow$ & \begin{tikzcd}[ampersand replacement=\&, baseline=(current bounding box.center)] X_2 \arrow[r] \& X_3 \end{tikzcd} & $\Longrightarrow$ & \begin{tikzcd}[ampersand replacement=\&, baseline=(current bounding box.center)] X_2 \arrow[r,"{}_{R_3}"] \& X_3 \end{tikzcd} \\[0pt]
$R_4:$ & & \begin{tikzcd}[ampersand replacement=\&, baseline=(current bounding box.center)] X_3 \arrow[r] \& X_2 \end{tikzcd} & $\Longrightarrow$ & \begin{tikzcd}[ampersand replacement=\&, baseline=(current bounding box.center)] X_3 \arrow[r] \& X_2 \end{tikzcd} & $\Longrightarrow$ & \begin{tikzcd}[ampersand replacement=\&, baseline=(current bounding box.center)] X_3 \arrow[r,"{}_{R_4}"] \& X_2 \end{tikzcd} \\[0pt]
$R_5:$ & & \begin{tikzcd}[ampersand replacement=\&, baseline=(current bounding box.center)] X_3 \arrow[r] \& Y_1 \end{tikzcd} & $\Longrightarrow$ & \begin{tikzcd}[ampersand replacement=\&, baseline=(current bounding box.center)] X_3 \arrow[r] \& 0 \end{tikzcd} & $\Longrightarrow$ & \begin{tikzcd}[ampersand replacement=\&, baseline=(current bounding box.center)] X_3 \arrow[r,"{}_{R_5}"] \& 0 \end{tikzcd} \\[0pt]
$R_6:$ & & \begin{tikzcd}[ampersand replacement=\&, baseline=(current bounding box.center)] Y_1 + X_4 \arrow[r] \& X_5 \end{tikzcd} & $\Longrightarrow$ & \begin{tikzcd}[ampersand replacement=\&, baseline=(current bounding box.center)] X_4 \arrow[r] \& X_5 \end{tikzcd} & $\Longrightarrow$ & \begin{tikzcd}[ampersand replacement=\&, baseline=(current bounding box.center)] X_4 \arrow[r,"{}_{R_6}"] \& X_5 \end{tikzcd} \\[0pt]
$R_7:$ & & \begin{tikzcd}[ampersand replacement=\&, baseline=(current bounding box.center)] X_5 \arrow[r] \& X_4 + Y_1 \end{tikzcd} & $\Longrightarrow$ & \begin{tikzcd}[ampersand replacement=\&, baseline=(current bounding box.center)] X_5 \arrow[r] \& X_4 \end{tikzcd} & $\Longrightarrow$ & \begin{tikzcd}[ampersand replacement=\&, baseline=(current bounding box.center)] X_5 \arrow[r,"{}_{R_7}"] \& X_4 \end{tikzcd} \\[0pt]
$R_8:$ & & \begin{tikzcd}[ampersand replacement=\&, baseline=(current bounding box.center)] X_5 \arrow[r] \& X_2 + Y_2 \end{tikzcd} & $\Longrightarrow$ & \begin{tikzcd}[ampersand replacement=\&, baseline=(current bounding box.center)] X_5 \arrow[r] \& X_2 \end{tikzcd} & $\Longrightarrow$ & \begin{tikzcd}[ampersand replacement=\&, baseline=(current bounding box.center)] X_5 \arrow[r,"{}_{R_8}"] \& X_2 \end{tikzcd} \\[0pt]
$R_9:$ & & \begin{tikzcd}[ampersand replacement=\&, baseline=(current bounding box.center)] X_1 + Y_2 \arrow[r] \& X_6 \end{tikzcd} & $\Longrightarrow$ & \begin{tikzcd}[ampersand replacement=\&, baseline=(current bounding box.center)] X_1 \arrow[r] \& X_6 \end{tikzcd} & $\Longrightarrow$ & \begin{tikzcd}[ampersand replacement=\&, baseline=(current bounding box.center)] X_1 \arrow[r,"{}_{R_9}"] \& X_6 \end{tikzcd} \\[0pt]
$R_{10}:$ & & \begin{tikzcd}[ampersand replacement=\&, baseline=(current bounding box.center)] X_6 \arrow[r] \& X_1 + Y_2 \end{tikzcd} & $\Longrightarrow$ & \begin{tikzcd}[ampersand replacement=\&, baseline=(current bounding box.center)] X_6 \arrow[r] \& X_1 \end{tikzcd} & $\Longrightarrow$ & \begin{tikzcd}[ampersand replacement=\&, baseline=(current bounding box.center)] X_6 \arrow[r,"{}_{R_{10}}"] \& X_1 \end{tikzcd} \\[0pt]
$R_{11}:$ & & \begin{tikzcd}[ampersand replacement=\&, baseline=(current bounding box.center)] X_6 \arrow[r] \& X_1 + X_4 \end{tikzcd} & $\Longrightarrow$ & \begin{tikzcd}[ampersand replacement=\&, baseline=(current bounding box.center)] X_6 \arrow[r] \& X_1 + X_4 \end{tikzcd} & $\Longrightarrow$ & \begin{tikzcd}[ampersand replacement=\&, baseline=(current bounding box.center)] X_1 \& X_6 \arrow[l,"{}_{R_{11}}"'] \arrow[loop below,"{}_{R_{11}}" right] \arrow[r,"{}_{R_{11}}"] \& X_4 \end{tikzcd} \\[0pt]
$R_{12}:$ & & \begin{tikzcd}[ampersand replacement=\&, baseline=(current bounding box.center)] X_3 + Y_2 \arrow[r] \& X_7 \end{tikzcd} & $\Longrightarrow$ & \begin{tikzcd}[ampersand replacement=\&, baseline=(current bounding box.center)] X_3 \arrow[r] \& X_7 \end{tikzcd} & $\Longrightarrow$ & \begin{tikzcd}[ampersand replacement=\&, baseline=(current bounding box.center)] X_3 \arrow[r,"{}_{R_{12}}"] \& X_7 \end{tikzcd} \\[0pt]
$R_{13}:$ & & \begin{tikzcd}[ampersand replacement=\&, baseline=(current bounding box.center)] X_7 \arrow[r] \& X_3 + Y_2 \end{tikzcd} & $\Longrightarrow$ & \begin{tikzcd}[ampersand replacement=\&, baseline=(current bounding box.center)] X_7 \arrow[r] \& X_3 \end{tikzcd} & $\Longrightarrow$ & \begin{tikzcd}[ampersand replacement=\&, baseline=(current bounding box.center)] X_7 \arrow[r,"{}_{R_{13}}"] \& X_3 \end{tikzcd} \\[0pt]
$R_{14}:$ & & \begin{tikzcd}[ampersand replacement=\&, baseline=(current bounding box.center)] X_7 \arrow[r] \& X_3 + X_4 \end{tikzcd} & $\Longrightarrow$ & \begin{tikzcd}[ampersand replacement=\&, baseline=(current bounding box.center)] X_7 \arrow[r] \& X_3 + X_4 \end{tikzcd} & $\Longrightarrow$ & \begin{tikzcd}[ampersand replacement=\&, baseline=(current bounding box.center)] X_3 \& X_7 \arrow[l,"{}_{R_{14}}"'] \arrow[loop below,"{}_{R_{14}}" right] \arrow[r,"{}_{R_{14}}"] \& X_4 \end{tikzcd} \\[0pt]
\bottomrule
\end{NiceTabular}
\caption{Shown above are the steps which take each reaction (left) from the chemical reaction network \eqref{envzompr-reindex} to the corresponding ALT-graph edges (right). First, we construct the $\mathcal{X}$-reduced edge (middle) by removing the species not in the set $\mathcal{X} = \{ X_1, \ldots, X_7 \}$. We then construct the ALT-graph edges (right) by noting that the $R_5$ is a leak edge, $R_{11}$ and $R_{14}$ are catalytic transition (type III) edges, and the rest are transition edges. We apply the corresponding rules given in Table \ref{tab:jacobian}. Combining the ALT-graph edges (right) to a single graph on the vertices $\mathcal{X} \cup \{ 0 \}$ gives the ALT-graph \eqref{eq:ALT-envzompr}.}
\label{tab:jacobian2}
\end{table}

The ALT-graph of \eqref{envzompr-reindex} is given by the following, where the assignment of the individual transitions is contained in Table \ref{tab:jacobian2}:
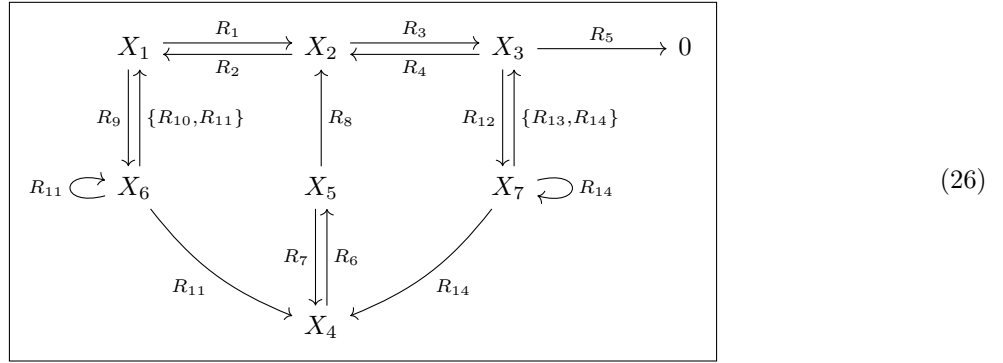
\begin{equation}
\label{eq:ALT-envzompr}
\begin{tikzcd}[ampersand replacement=\&,execute at end picture={
        \draw[black, thin] 
            ([shift={(-5pt,-5pt)}]current bounding box.south west) 
            rectangle 
            ([shift={(5pt,5pt)}]current bounding box.north east);                  
            \node[above] at ([shift={(0,5pt)}]current bounding box.north) {$\mathcal{G} = (\mathcal{V}, \mathcal{E}, \Omega, \phi)$};
    }
]
X_1 \arrow[rr,yshift=0.5ex,"R_1"] \arrow[dd,xshift=-0.5ex,"R_9"'] \& \& X_2 \arrow[ll,yshift=-0.5ex,"R_2"] \arrow[rr,yshift=0.5ex,"R_3"] \& \& X_3 \arrow[ll,yshift=-0.5ex,"R_4"] \arrow[rr,"R_5"] \arrow[dd,xshift=-0.5ex,"R_{12}"'] \& \& 0 \\
\\
 X_6 \arrow[loop left,"R_{11}"] \arrow[uu,xshift=0.5ex,"{\{R_{10},R_{11}\}}"'] \arrow[ddrr,bend right=15,"R_{11}"'] \& \& X_5 \arrow[uu,"R_8"'] \arrow[dd,xshift=-0.5ex,"R_7"'] \& \& X_7 \arrow[loop right,"R_{14}"] \arrow[uu,xshift=0.5ex,"{\{R_{13},R_{14}\}}"'] \arrow[ddll,bend left=15,"R_{14}"] \& \& \\
\\
\& \& X_4 \arrow[uu,xshift=0.5ex,"R_{6}"'] \& \& \& \& 
\end{tikzcd}
\end{equation}
The autocatalytic reaction set is $\mathcal{R}_A = \{ R_{11}, R_{14}\}$ and we correspondingly have autocatalytic instances $R_{11} \in \phi(e_{66})$ and $R_{14} \in \phi(e_{77})$, and catalytic instances $R_{11} \in \phi(e_{61})$, $R_{11} \in \phi(e_{64})$, $R_{14} \in \phi(e_{73})$, and $R_{14} \in \phi(e_{74})$. The strongly connected components are given by $\mathcal{Z}_1 = \{ X_1, X_2, X_3, X_4, X_5, X_6, X_7 \}$ and $\mathcal{Z}_2 = \{ 0 \}$. Since all catalytic instances of $R_{11}$ and $R_{14}$ lie within $\mathcal{Z}_1$, we have that all autocatalytic reactions are critical so that $\mathcal{R}_C = \{ R_{11}, R_{14} \}$. Since no instances of the autocatalytic reactions $\mathcal{R}_A = \{ R_{11}, R_{14} \}$ leave $\mathcal{Z}_1$, we have no critical strongly connected components so that $\mathcal{Z}_C = \emptyset$. Consequently, we need to perform autocatalytic/catalytic splittings according to Algorithm \ref{alg:1} until $\mathcal{R}_C = \emptyset$.

Applying Algorithm \ref{alg:1}, we initialize $\mathcal{G}_F$ with $\mathcal{E}_F = \emptyset$ and $\mathcal{G}_{-V} = \mathcal{G}$ with $\mathcal{E}_{-V} = \mathcal{E}$. We notice that the vertex $X_4$ is the target of catalytic instances $R_{11} \in \phi(e_{64})$ and $R_{14} \in \phi(e_{74})$ so that we apply lines 16--17 of Algorithm \ref{alg:1}. Performing a catalytic splitting on the catalytic instance $R_{11} \in \phi(e_{64})$ gives the following decomposition $\mathcal{G} = \mathcal{G}_{F_1} \oplus \mathcal{G}_{-V_1}$:
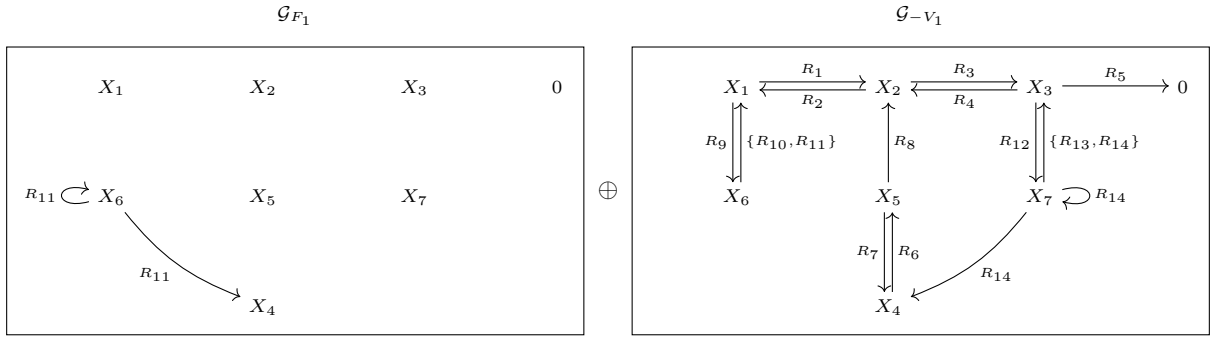
\[ \scriptsize
\begin{tikzcd}[ampersand replacement=\&,execute at end picture={
        \draw[black, thin] 
            ([shift={(-5pt,-5pt)}]current bounding box.south west) 
            rectangle 
            ([shift={(5pt,5pt)}]current bounding box.north east);                  
            \node[above] at ([shift={(0,5pt)}]current bounding box.north) {$\mathcal{G}_{F_1}$};
    }
]
X_1 \arrow[rr,white,yshift=0.5ex,"R_1"] \arrow[dd,white,xshift=-0.5ex,"R_9"'] \& \& X_2 \arrow[ll,white,yshift=-0.5ex,"R_2"] \arrow[rr,white,yshift=0.5ex,"R_3"] \& \& X_3 \arrow[ll,white,yshift=-0.5ex,"R_4"] \arrow[rr,white,"R_5"] \arrow[dd,white,xshift=-0.5ex,"R_{12}"'] \& \& 0 \\
\\
 X_6 \arrow[loop left,"R_{11}"] \arrow[ddrr,bend right=15,"R_{11}"'] \& \& X_5 \& \& X_7 \arrow[loop right,white,"R_{14}"] \arrow[ddll,white,bend left=15,"R_{14}"] \& \& \\
\\
\& \& X_4 \& \& \& \& 
\end{tikzcd} \; \oplus \;
\begin{tikzcd}[ampersand replacement=\&,execute at end picture={
        \draw[black, thin] 
            ([shift={(-5pt,-5pt)}]current bounding box.south west) 
            rectangle 
            ([shift={(5pt,5pt)}]current bounding box.north east);                  
            \node[above] at ([shift={(0,5pt)}]current bounding box.north) {$\mathcal{G}_{-V_1}$};
    }
]
X_1 \arrow[rr,yshift=0.5ex,"R_1"] \arrow[dd,xshift=-0.5ex,"R_9"'] \& \& X_2 \arrow[ll,yshift=-0.5ex,"R_2"] \arrow[rr,yshift=0.5ex,"R_3"] \& \& X_3 \arrow[ll,yshift=-0.5ex,"R_4"] \arrow[rr,"R_5"] \arrow[dd,xshift=-0.5ex,"R_{12}"'] \& \& 0 \\
\\
 X_6 \arrow[loop left,white,"R_{11}"] \arrow[uu,xshift=0.5ex,"{\{R_{10},R_{11}\}}"'] \& \& X_5 \arrow[uu,"R_8"'] \arrow[dd,xshift=-0.5ex,"R_7"'] \& \& X_7 \arrow[loop right,"R_{14}"] \arrow[uu,xshift=0.5ex,"{\{R_{13},R_{14}\}}"']  \arrow[ddll,bend left=15,"R_{14}"] \& \& \\
\\
\& \& X_4 \arrow[uu,xshift=0.5ex,"R_{6}"'] \& \& \& \& 
\end{tikzcd}
\]

We now recompute the sets $\mathcal{R}_A$, $\mathcal{R}_C$, and $\mathcal{Z}_C$ on the graph $\mathcal{G}_{-V_1}$. We notice that $R_{11}$ is not autocatalytic in $\mathcal{G}_{-V_1}$ but $R_{14}$ is, so that $\mathcal{R}_A = \{ R_{14} \}$. As before, we have the strongly connected components $\mathcal{Z}_1 = \{ X_1, X_2, X_3, X_4, X_5, X_6, X_7\}$ and $\mathcal{Z}_2 = \{ 0 \}$. The catalytic instances of $R_{14}$ lie within $\mathcal{Z}_1$ so that $R_{14}$ is critical so we have $\mathcal{R}_C = \{ R_{14} \}$. No catalytic instances leave $\mathcal{Z}_1$ so that $\mathcal{Z}_C = \emptyset$. Since $\mathcal{R}_C \not= \emptyset$, Algorithm \ref{alg:1} continues. We notice that the vertex $X_4$ is the target of catalytic instances $R_{74} \in \phi_{-V_1}(e_{64})$ and the instance $R_{14} \in \phi_{F_1}(e_{14})$ so that we apply lines 12--13 of Algorithm \ref{alg:1}. Performing a catalytic splitting on the catalytic instance $R_{14} \in \phi_{-V_1}(e_{74})$ gives the following decomposition $\mathcal{G}_{-V_1} = \mathcal{G}_{F_2} \oplus \mathcal{G}_{-V_2}$:
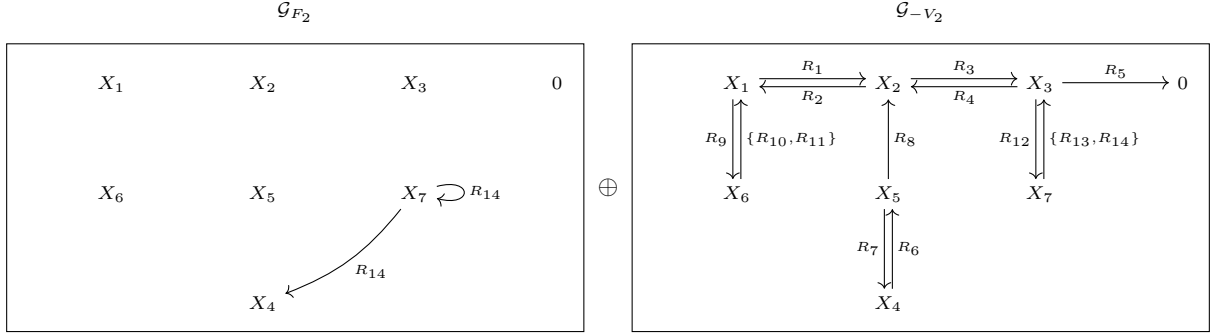
\[ \scriptsize
\begin{tikzcd}[ampersand replacement=\&,execute at end picture={
        \draw[black, thin] 
            ([shift={(-5pt,-5pt)}]current bounding box.south west) 
            rectangle 
            ([shift={(5pt,5pt)}]current bounding box.north east);                  
            \node[above] at ([shift={(0,5pt)}]current bounding box.north) {$\mathcal{G}_{F_2}$};
    }
]
X_1 \arrow[rr,white,yshift=0.5ex,"R_1"] \arrow[dd,white,xshift=-0.5ex,"R_9"'] \& \& X_2 \arrow[ll,white,yshift=-0.5ex,"R_2"] \arrow[rr,white,yshift=0.5ex,"R_3"] \& \& X_3 \arrow[ll,white,yshift=-0.5ex,"R_4"] \arrow[rr,white,"R_5"] \arrow[dd,white,xshift=-0.5ex,"R_{12}"'] \& \& 0 \\
\\
 X_6 \arrow[loop left,white,"R_{11}"] \arrow[ddrr,bend right=15,white,"R_{11}"'] \& \& X_5 \& \& X_7 \arrow[loop right,"R_{14}"] \arrow[ddll,bend left=15,"R_{14}"] \& \& \\
\\
\& \& X_4 \& \& \& \& 
\end{tikzcd} \; \oplus \;
\begin{tikzcd}[ampersand replacement=\&,execute at end picture={
        \draw[black, thin] 
            ([shift={(-5pt,-5pt)}]current bounding box.south west) 
            rectangle 
            ([shift={(5pt,5pt)}]current bounding box.north east);                  
            \node[above] at ([shift={(0,5pt)}]current bounding box.north) {$\mathcal{G}_{-V_2}$};
    }
]
X_1 \arrow[rr,yshift=0.5ex,"R_1"] \arrow[dd,xshift=-0.5ex,"R_9"'] \& \& X_2 \arrow[ll,yshift=-0.5ex,"R_2"] \arrow[rr,yshift=0.5ex,"R_3"] \& \& X_3 \arrow[ll,yshift=-0.5ex,"R_4"] \arrow[rr,"R_5"] \arrow[dd,xshift=-0.5ex,"R_{12}"'] \& \& 0 \\
\\
 X_6 \arrow[loop left,white,"R_{11}"] \arrow[uu,xshift=0.5ex,"{\{R_{10},R_{11}\}}"'] \& \& X_5 \arrow[uu,"R_8"'] \arrow[dd,xshift=-0.5ex,"R_7"'] \& \& X_7 \arrow[loop right,white,"R_{14}"] \arrow[uu,xshift=0.5ex,"{\{R_{13},R_{14}\}}"']  \arrow[ddll,white,bend left=15,"R_{14}"] \& \& \\
\\
\& \& X_4 \arrow[uu,xshift=0.5ex,"R_{6}"'] \& \& \& \& 
\end{tikzcd}
\]

For the graph $\mathcal{G}_{-V_2}$ there are no autocatalytic reactions so that $\mathcal{R}_A = \emptyset$, $\mathcal{R}_C = \emptyset$, and $\mathcal{Z}_C = \emptyset$, so Algorithm \ref{alg:1} terminates and we have the autocatalytic decomposition $\mathcal{G} = \mathcal{G}_F \oplus \mathcal{G}_{-V}$. Setting $\mathcal{G}_F = \mathcal{G}_{F_1} \oplus \mathcal{G}_{F_2}$ and $\mathcal{G}_{-V} = \mathcal{G}_{-V_2}$, we have the following autocatalytic decomposition:
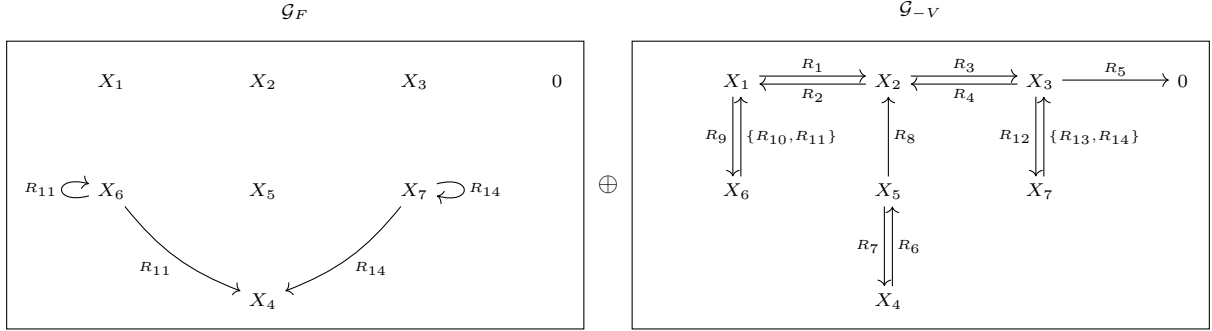
\[ \scriptsize
\begin{tikzcd}[ampersand replacement=\&,execute at end picture={
        \draw[black, thin] 
            ([shift={(-5pt,-5pt)}]current bounding box.south west) 
            rectangle 
            ([shift={(5pt,5pt)}]current bounding box.north east);                  
            \node[above] at ([shift={(0,5pt)}]current bounding box.north) {$\mathcal{G}_F$};
    }
]
X_1 \arrow[rr,white,yshift=0.5ex,"R_1"] \arrow[dd,white,xshift=-0.5ex,"R_9"'] \& \& X_2 \& \& X_3 \& \& 0 \\
\\
 X_6 \arrow[loop left,"R_{11}"] \arrow[ddrr,bend right=15,"R_{11}"'] \& \& X_5 \& \& X_7 \arrow[loop right,"R_{14}"] \arrow[ddll,bend left=15,"R_{14}"] \& \& \\
\\
\& \& X_4 \& \& \& \& 
\end{tikzcd} \; \oplus \;
\begin{tikzcd}[ampersand replacement=\&,execute at end picture={
        \draw[black, thin] 
            ([shift={(-5pt,-5pt)}]current bounding box.south west) 
            rectangle 
            ([shift={(5pt,5pt)}]current bounding box.north east);                  
            \node[above] at ([shift={(0,5pt)}]current bounding box.north) {$\mathcal{G}_{-V}$};
    }
]
X_1 \arrow[rr,yshift=0.5ex,"R_1"] \arrow[dd,xshift=-0.5ex,"R_9"'] \& \& X_2 \arrow[ll,yshift=-0.5ex,"R_2"] \arrow[rr,yshift=0.5ex,"R_3"] \& \& X_3 \arrow[ll,yshift=-0.5ex,"R_4"] \arrow[rr,"R_5"] \arrow[dd,xshift=-0.5ex,"R_{12}"'] \& \& 0 \\
\\
 X_6 \arrow[loop left,white,"R_{11}"] \arrow[uu,xshift=0.5ex,"{\{R_{10},R_{11}\}}"'] \& \& X_5 \arrow[uu,"R_8"'] \arrow[dd,xshift=-0.5ex,"R_7"'] \& \& X_7 \arrow[uu,xshift=0.5ex,"{\{R_{13},R_{14}\}}"'] \& \& \\
\\
\& \& X_4 \arrow[uu,xshift=0.5ex,"R_{6}"'] \& \& \& \& 
\end{tikzcd}
\]
The ALT-graphs $\mathcal{G}_F$ and $\mathcal{G}_V$ satisfy the assumptions of Theorem \ref{thm:1} so that we have a valid splitting. The corresponding ALT-matrices are
\[\scriptsize F = \left[ \begin{array}{ccccccc}
\displaystyle 0 & \displaystyle 0 & \displaystyle 0 & \displaystyle 0 & \displaystyle 0 & \displaystyle 0 & \displaystyle 0 \\
\displaystyle 0 & \displaystyle 0 & \displaystyle 0 & \displaystyle 0 & \displaystyle 0 & \displaystyle 0 & \displaystyle 0 \\
\displaystyle 0 & \displaystyle 0 & \displaystyle 0 & \displaystyle 0 & \displaystyle 0 & \displaystyle 0 & \displaystyle 0 \\
\displaystyle 0 & \displaystyle 0 & \displaystyle 0 & \displaystyle 0 & \displaystyle 0 & \displaystyle \partial_6 R_{11} & \partial_7 R_{14} \\
\displaystyle 0 & \displaystyle 0 & \displaystyle 0 & \displaystyle 0 & \displaystyle 0 & \displaystyle 0 & \displaystyle 0 \\
\displaystyle 0 & \displaystyle 0 & \displaystyle 0 & \displaystyle 0 & \displaystyle 0 & \displaystyle 0 & \displaystyle 0 \\
\displaystyle 0 & \displaystyle 0 & \displaystyle 0 & \displaystyle 0 & \displaystyle 0 & \displaystyle 0 & \displaystyle 0
\end{array}\right]\] 
and
\[ \scriptsize
-V = \left[ \begin{array}{ccccccc}
\displaystyle -\partial_1 (R_1 + R_9) & \displaystyle \partial_2 R_2 & \displaystyle 0 & \displaystyle 0 & \displaystyle 0 & \displaystyle \partial_6 (R_{10} + R_{11}) & \displaystyle 0 \\ 
\displaystyle \partial_1 R_1 & \displaystyle -\partial_2 (R_2 + R_3) & \displaystyle \partial_3 R_4 & \displaystyle 0 & \displaystyle \partial_5 R_8 & \displaystyle 0 & \displaystyle 0 \\ 
\displaystyle 0 & \displaystyle \partial_2 R_3 & \displaystyle -\partial_3 (R_4 + R_5 + R_{12}) & \displaystyle 0 & \displaystyle 0 & \displaystyle 0 & \displaystyle \partial_7 (R_{13} + R_{14}) \\ 
\displaystyle 0 & \displaystyle 0 & \displaystyle 0 & \displaystyle -\partial_4 R_6 & \displaystyle \partial_5 R_7 & \displaystyle 0 & \displaystyle 0 \\ 
\displaystyle 0 & \displaystyle 0 & \displaystyle 0 & \displaystyle \partial_4 R_6 & \displaystyle -\partial_5 (R_7 + R_8) & \displaystyle 0 & \displaystyle 0 \\ 
\displaystyle \partial_1 R_9 & \displaystyle 0 & \displaystyle 0 & \displaystyle 0 & \displaystyle 0 & \displaystyle -\partial_6 (R_{10} + R_{11}) & \displaystyle 0 \\ 
\displaystyle 0 & \displaystyle 0 & \displaystyle \partial_3 R_{12} & \displaystyle 0 & \displaystyle 0 & \displaystyle 0 & \displaystyle -\partial_7 (R_{13} + R_{14})
\end{array} \right].
\]
We notice that $F$ has rank 1 so that $FV^{-1}$ also has rank 1. The reproduction number is simply the $(4,4)$-entry of $FV^{-1}$, which simplifies to:
{\tiny
\[
\rho(FV^{-1}) = \left( \frac{
\left(
\left(
\partial_2 R_2\partial_1 R_9(\partial_3 R_4+\partial_3 R_5)
+
\partial_1 R_1\partial_2 R_3\partial_3 R_{12}
\right)
\partial_7 R_{14}
+
\partial_2 R_2\partial_1 R_9\partial_7 R_{13}(\partial_3 R_4+\partial_3 R_5)
\right)
\partial_6 R_{11}
+
\partial_1 R_1\partial_2 R_3\partial_6 R_{10}\partial_3 R_{12}\partial_7 R_{14}
}{
\partial_1 R_1
\partial_2 R_3
\partial_3 R_5
\left(\partial_6 R_{10}+\partial_6 R_{11}\right)
\left(\partial_7 R_{13}+\partial_7 R_{14}\right)
} \right).
\]
}

\end{example}


\begin{example}
\label{ex:twosite2}
We now present full computational details for the dual-site bifunctional enzyme cycles \eqref{eq:twosite2} studied in Example \ref{ex:twosite}. Recall that the network has been studied in the context of absolute concentration robustness in \cite{joshi2024bifunctional,joshi2026bifunctional} and has been studied in the context of boundary stability in the paper \cite{joshi2026bistability}.

\begin{table}[t]
\centering
\begin{NiceTabular}{ccccccc}[code-before = \rowcolors{2}{white}{gray!12}]
\toprule
\begin{tabular}{c}Reaction\\ label \end{tabular} & & \begin{tabular}{c}Chemical \\ reaction \end{tabular} & & \begin{tabular}{c}$\mathcal{X}$-reduced\\reaction\end{tabular} & & \begin{tabular}{c} ALT-graph\\edge(s)\end{tabular} \\
\midrule
$R_1:$ & & \begin{tikzcd}[ampersand replacement=\&, baseline=(current bounding box.center)] X_1 + Y_2 \arrow[r] \& X_3 \end{tikzcd} & $\Longrightarrow$ & \begin{tikzcd}[ampersand replacement=\&, baseline=(current bounding box.center)] X_1 \arrow[r] \& X_3 \end{tikzcd} & $\Longrightarrow$ & \begin{tikzcd}[ampersand replacement=\&, baseline=(current bounding box.center)] X_1 \arrow[r,"{}_{R_1}"] \& X_3 \end{tikzcd} \\[0pt]
$R_2:$ & & \begin{tikzcd}[ampersand replacement=\&, baseline=(current bounding box.center)] X_3 \arrow[r] \& X_1 + Y_2 \end{tikzcd} & $\Longrightarrow$ & \begin{tikzcd}[ampersand replacement=\&, baseline=(current bounding box.center)] X_3 \arrow[r] \& X_1 \end{tikzcd} & $\Longrightarrow$ & \begin{tikzcd}[ampersand replacement=\&, baseline=(current bounding box.center)] X_3 \arrow[r,"{}_{R_2}"] \& X_1 \end{tikzcd} \\[0pt]
$R_3:$ & & \begin{tikzcd}[ampersand replacement=\&, baseline=(current bounding box.center)] X_3 \arrow[r] \& X_2 + Y_2 \end{tikzcd} & $\Longrightarrow$ & \begin{tikzcd}[ampersand replacement=\&, baseline=(current bounding box.center)] X_3 \arrow[r] \& X_2 \end{tikzcd} & $\Longrightarrow$ & \begin{tikzcd}[ampersand replacement=\&, baseline=(current bounding box.center)] X_3 \arrow[r,"{}_{R_3}"] \& X_2 \end{tikzcd} \\[0pt]
$R_4:$ & & \begin{tikzcd}[ampersand replacement=\&, baseline=(current bounding box.center)] X_2 + Y_2 \arrow[r] \& X_4 \end{tikzcd} & $\Longrightarrow$ & \begin{tikzcd}[ampersand replacement=\&, baseline=(current bounding box.center)] X_2 \arrow[r] \& X_4 \end{tikzcd} & $\Longrightarrow$ & \begin{tikzcd}[ampersand replacement=\&, baseline=(current bounding box.center)] X_2 \arrow[r,"{}_{R_4}"] \& X_4 \end{tikzcd} \\[0pt]
$R_5:$ & & \begin{tikzcd}[ampersand replacement=\&, baseline=(current bounding box.center)] X_4 \arrow[r] \& X_2 + Y_2 \end{tikzcd} & $\Longrightarrow$ & \begin{tikzcd}[ampersand replacement=\&, baseline=(current bounding box.center)] X_4 \arrow[r] \& X_2 \end{tikzcd} & $\Longrightarrow$ & \begin{tikzcd}[ampersand replacement=\&, baseline=(current bounding box.center)] X_4 \arrow[r,"{}_{R_5}"] \& X_2 \end{tikzcd} \\[0pt]
$R_6:$ & & \begin{tikzcd}[ampersand replacement=\&, baseline=(current bounding box.center)] X_4 \arrow[r] \& Y_1 + Y_2 \end{tikzcd} & $\Longrightarrow$ & \begin{tikzcd}[ampersand replacement=\&, baseline=(current bounding box.center)] X_4 \arrow[r] \& 0 \end{tikzcd} & $\Longrightarrow$ & \begin{tikzcd}[ampersand replacement=\&, baseline=(current bounding box.center)] X_4 \arrow[r,"{}_{R_6}"] \& 0 \end{tikzcd} \\[0pt]
$R_7:$ & & \begin{tikzcd}[ampersand replacement=\&, baseline=(current bounding box.center)] \mathcal{E}_1 +  Y_1 \arrow[r] \& X_5 \end{tikzcd} & $\Longrightarrow$ & \begin{tikzcd}[ampersand replacement=\&, baseline=(current bounding box.center)] \mathcal{E}_1 \arrow[r] \& X_5 \end{tikzcd} & $\Longrightarrow$ & \begin{tikzcd}[ampersand replacement=\&, baseline=(current bounding box.center)] \mathcal{E}_1 \arrow[r,"{}_{R_7}"] \& X_5 \end{tikzcd} \\[0pt]
$R_8:$ & & \begin{tikzcd}[ampersand replacement=\&, baseline=(current bounding box.center)] X_5 \arrow[r] \& Y_1 + \mathcal{E}_1 \end{tikzcd} & $\Longrightarrow$ & \begin{tikzcd}[ampersand replacement=\&, baseline=(current bounding box.center)] X_5 \arrow[r] \& \mathcal{E}_1 \end{tikzcd} & $\Longrightarrow$ & \begin{tikzcd}[ampersand replacement=\&, baseline=(current bounding box.center)] X_5 \arrow[r,"{}_{R_8}"] \& \mathcal{E}_1 \end{tikzcd} \\[0pt]
$R_9:$ & & \begin{tikzcd}[ampersand replacement=\&, baseline=(current bounding box.center)] X_5 \arrow[r] \& X_2 + \mathcal{E}_1 \end{tikzcd} & $\Longrightarrow$ & \begin{tikzcd}[ampersand replacement=\&, baseline=(current bounding box.center)]X_5 \arrow[r] \& X_2 + \mathcal{E}_1 \end{tikzcd} & $\Longrightarrow$ & \begin{tikzcd}[ampersand replacement=\&, baseline=(current bounding box.center)] X_2 \& X_5 \arrow[l,"{}_{R_{9}}"'] \arrow[loop below,"{}_{R_{9}}" right] \arrow[r,"{}_{R_{9}}"] \& \mathcal{E}_1 \end{tikzcd} \\[0pt]
$R_{10}:$ & & \begin{tikzcd}[ampersand replacement=\&, baseline=(current bounding box.center)] X_2 + \mathcal{E}_2 \arrow[r] \& X_6 \end{tikzcd} & $\Longrightarrow$ & \begin{tikzcd}[ampersand replacement=\&, baseline=(current bounding box.center)] X_2 + \mathcal{E}_2 \arrow[r] \& X_6 \end{tikzcd} & $\Longrightarrow$ & \begin{tikzcd}[ampersand replacement=\&, baseline=(current bounding box.center)] None \end{tikzcd} \\[0pt]
$R_{11}:$ & & \begin{tikzcd}[ampersand replacement=\&, baseline=(current bounding box.center)] X_6 \arrow[r] \& X_2 + \mathcal{E}_2 \end{tikzcd} & $\Longrightarrow$ & \begin{tikzcd}[ampersand replacement=\&, baseline=(current bounding box.center)] X_6 \arrow[r] \& X_2 + \mathcal{E}_2 \end{tikzcd} & $\Longrightarrow$ & \begin{tikzcd}[ampersand replacement=\&, baseline=(current bounding box.center)] X_2 \& X_6 \arrow[l,"{}_{R_{11}}"'] \arrow[loop below,"{}_{R_{11}}" right] \arrow[r,"{}_{R_{11}}"] \& \mathcal{E}_2 \end{tikzcd} \\[0pt]
$R_{12}:$ & & \begin{tikzcd}[ampersand replacement=\&, baseline=(current bounding box.center)] X_6 \arrow[r] \& X_1 + \mathcal{E}_2 \end{tikzcd} & $\Longrightarrow$ & \begin{tikzcd}[ampersand replacement=\&, baseline=(current bounding box.center)] X_6 \arrow[r] \& X_1 + \mathcal{E}_2 \end{tikzcd} & $\Longrightarrow$ & \begin{tikzcd}[ampersand replacement=\&, baseline=(current bounding box.center)] X_1 \& X_6 \arrow[l,"{}_{R_{12}}"'] \arrow[loop below,"{}_{R_{12}}" right] \arrow[r,"{}_{R_{12}}"] \& \mathcal{E}_2 \end{tikzcd} \\[0pt]
\bottomrule
\end{NiceTabular}
\caption{Shown above are the steps which take each reaction (left) from the chemical reaction network \eqref{eq:twosite2} to the corresponding ALT-graph edges (right). First, we construct the $\mathcal{X}$-reduced edge (middle) by removing the species not in the set $\mathcal{X} = \{ X_1, \ldots, X_6 \}$. We then construct the ALT-graph edges (right) by noting that the $R_6$ is a leak edge, $R_{9}$, $R_{11}$ and $R_{12}$ are catalytic transition (type III) edges, the rest are transition edges. Note that $R_{10}$ is not assigned to an edge in the ALT-graph because it has multiple inputs within the set $\mathcal{X}$. We apply the corresponding rules given in Table \ref{tab:jacobian}. Combining the ALT-graph edges (right) to a single graph on the vertices $\mathcal{X} \cup \{ 0 \}$ gives the ALT-graphs \eqref{alt-twosite} and \eqref{alt-twosite2}.}
\label{tab:jacobian3}
\end{table}

We now focus on the ALT-graph approach to determining a valid Jacobian splitting $J = F - V$. All four mechanisms have the boundary steady state $(S_0, S_1, S_2, E, C_1, C_2, C_3, C_4) = (0,0,S_{tot},E_{tot},0,0,0,0)$ where $S_{tot}$ and $E_{tot}$ are the conserved totals of substrate and enzyme, respectively, satisfying
\[\begin{aligned}
        &S_{tot} = S_0 + S_1 + S_2 + C_1 + C_2 + 2C_3 + 2 C_4, \\
        &E_{tot} = E + C_1 + C_2 + C_3 + C_4. 
\end{aligned}\]
The corresponding siphon is $\mathcal{X} = \{S_0, S_1, C_1, C_2, C_3, C_4\}$ so we relabel $X_1 = S_0, X_2 = S_1, X_3 = C_1, X_4 = C_2, X_5 = C_3, X_6 = C_4, Y_1 = S_2$, and $Y_2 = E$. Reindexing \eqref{eq:twosite2} gives the corresponding system
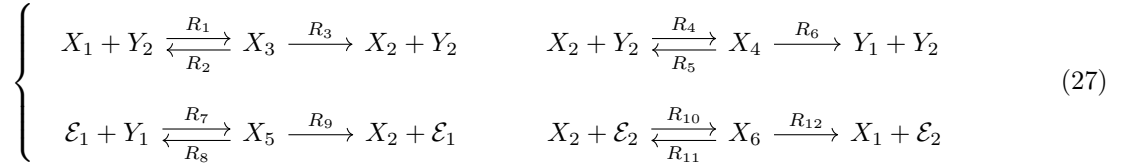
\begin{equation}
\label{eq:twosite4} \left\{ \; \;
\begin{tikzcd}[ampersand replacement=\&]
            X_1 + Y_2 \arrow[r,yshift=0.5ex,"R_1"] \& X_3 \arrow[l,yshift=-0.5ex,"R_2"] \arrow[r,"R_3"] \& X_2 + Y_2 \&  X_2 + Y_2 \arrow[r,yshift=0.5ex,"R_4"] \& X_4 \arrow[l,yshift=-0.5ex,"R_5"] \arrow[r,"R_6"] \& Y_1 + Y_2 \\
            \mathcal{E}_1 + Y_1\arrow[r,yshift=0.5ex,"R_7"] \& X_5 \arrow[l,yshift=-0.5ex,"R_8"] \arrow[r,"R_9"] \& X_2 + \mathcal{E}_1 \& X_2 + \mathcal{E}_2 \arrow[r,yshift=0.5ex,"R_{10}"] \& X_6 \arrow[l,yshift=-0.5ex,"R_{11}"] \arrow[r,"R_{12}"] \& X_1 + \mathcal{E}_2
\end{tikzcd} \right.
\end{equation}
where $\mathcal{E}_1, \mathcal{E}_2 \in \{ X_3, X_4 \}$. The $\mathcal{X}$-reduced network \eqref{eq:twosite4} for all four mechanisms is identical; however, differences arise in the ALT-graph due to the varying identities of $\mathcal{E}_1$ and $\mathcal{E}_2$. The correspondence of each reaction to its $\mathcal{X}$-reduced reaction and associated ALT-graph edges is given in Table \ref{tab:jacobian3}. The assembled ALT-graphs are as follows:
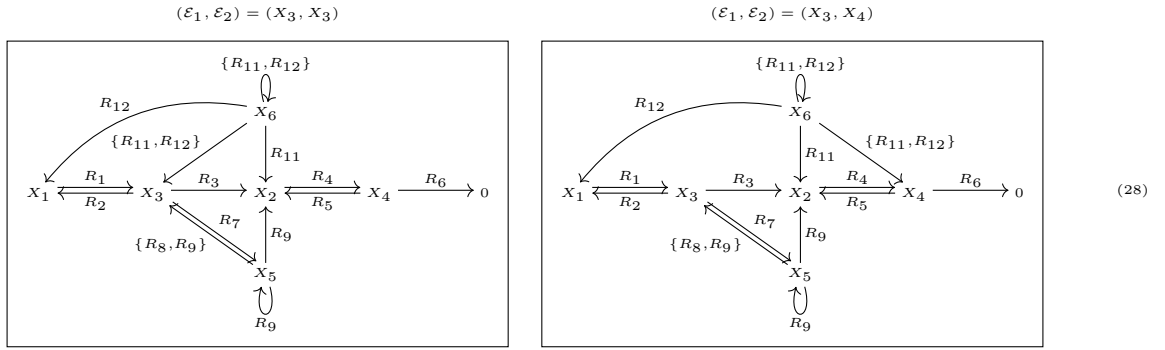
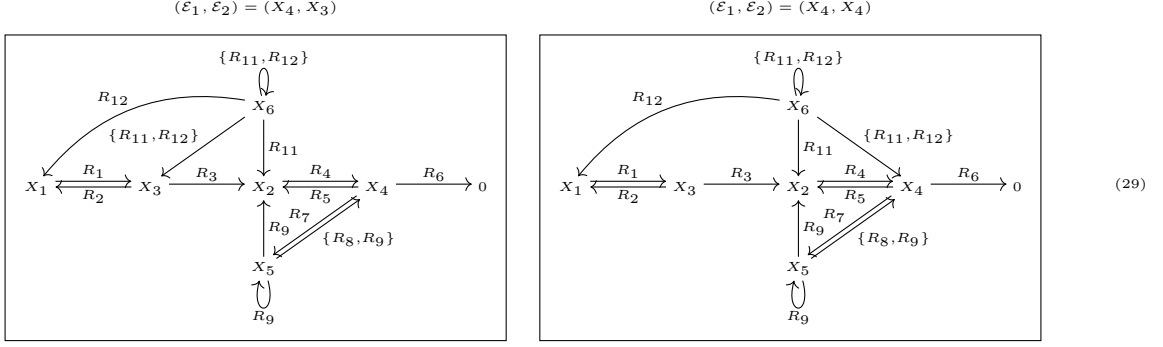
\begin{equation}
\label{alt-twosite}
\tiny
\begin{tikzcd}[ampersand replacement=\&,execute at end picture={
        \draw[black, thin] 
            ([shift={(-5pt,-5pt)}]current bounding box.south west) 
            rectangle 
            ([shift={(5pt,5pt)}]current bounding box.north east);                  
            \node[above] at ([shift={(0,5pt)}]current bounding box.north) {$(\mathcal{E}_1,\mathcal{E}_2)=(X_3,X_3)$};
    }
]
 \& \& \& \& X_6 \arrow[loop above, "{\{R_{11}, R_{12}\}}"] \arrow[lldd,"{\{R_{11}, R_{12}\}}"'] \arrow[lllldd, bend right=30, "R_{12}"'] \arrow[dd, "R_{11}"]\& \& \& \& \\ \\
 X_1 \arrow[rr,yshift=0.5ex,"R_1"] \& \& X_3 \arrow[ll,yshift=-0.5ex,"R_2"] \arrow[rr,"R_3"] \arrow[rrdd,xshift=0.3ex,yshift=0.3ex,"R_7"] \& \& X_2\arrow[rr,yshift=0.5ex,"R_4"] \& \& X_4 \arrow[ll,yshift=-0.5ex,"R_5"] \arrow[rr,"R_6"] \& \& 0 \\ \\
 \& \& \& \& X_5 \arrow[loop below, "R_9"] \arrow[lluu,xshift=-0.3ex,yshift=-0.3ex, "{\{R_{8}, R_9\}}"]  \arrow[uu, "R_9"']\& \& \& \&
\end{tikzcd}\; \; \; \; \; \;
\begin{tikzcd}[ampersand replacement=\&,execute at end picture={
        \draw[black, thin] 
            ([shift={(-5pt,-5pt)}]current bounding box.south west) 
            rectangle 
            ([shift={(5pt,5pt)}]current bounding box.north east);                  
            \node[above] at ([shift={(0,5pt)}]current bounding box.north) {$(\mathcal{E}_1,\mathcal{E}_2)=(X_3,X_4)$};
    }
]
 \& \& \& \& X_6 \arrow[loop above, "{\{R_{11}, R_{12}\}}"] \arrow[rrdd,"{\{R_{11}, R_{12}\}}"] \arrow[lllldd, bend right=30, "R_{12}"'] \arrow[dd, "R_{11}"]\& \& \& \& \\ \\
 X_1 \arrow[rr,yshift=0.5ex,"R_1"] \& \& X_3 \arrow[ll,yshift=-0.5ex,"R_2"] \arrow[rr,"R_3"] \arrow[rrdd,xshift=0.3ex,yshift=0.3ex,"R_7"] \& \& X_2\arrow[rr,yshift=0.5ex,"R_4"] \& \& X_4 \arrow[ll,yshift=-0.5ex,"R_5"] \arrow[rr,"R_6"] \& \& 0 \\ \\
 \& \& \& \& X_5 \arrow[loop below, "R_9"] \arrow[lluu,xshift=-0.3ex,yshift=-0.3ex, "{\{R_{8}, R_9\}}"] \arrow[uu, "R_9"']\& \& \& \&
\end{tikzcd}
\end{equation}
\begin{equation}
\label{alt-twosite2}
\tiny \begin{tikzcd}[ampersand replacement=\&,execute at end picture={
        \draw[black, thin] 
            ([shift={(-5pt,-5pt)}]current bounding box.south west) 
            rectangle 
            ([shift={(5pt,5pt)}]current bounding box.north east);                  
            \node[above] at ([shift={(0,5pt)}]current bounding box.north) {$(\mathcal{E}_1,\mathcal{E}_2)=(X_4,X_3)$};
    }
]
 \& \& \& \& X_6 \arrow[loop above, "{\{R_{11}, R_{12}\}}"] \arrow[lldd,"{\{R_{11}, R_{12}\}}"'] \arrow[lllldd, bend right=30, "R_{12}"'] \arrow[dd, "R_{11}"]\& \& \& \& \\ \\
 X_1 \arrow[rr,yshift=0.5ex,"R_1"] \& \& X_3 \arrow[ll,yshift=-0.5ex,"R_2"] \arrow[rr,"R_3"]  \& \& X_2\arrow[rr,yshift=0.5ex,"R_4"] \& \& X_4 \arrow[ll,yshift=-0.5ex,"R_5"] \arrow[lldd,xshift=-0.3ex,yshift=0.3ex,"R_7"'] \arrow[rr,"R_6"] \& \& 0 \\ \\
 \& \& \& \& X_5 \arrow[loop below, "R_9"] \arrow[rruu,xshift=0.3ex,yshift=-0.3ex, "{\{R_{8}, R_9\}}"'] \arrow[uu, "R_9"']\& \& \& \&
\end{tikzcd}\; \; \; \; \; \;
\begin{tikzcd}[ampersand replacement=\&,execute at end picture={
        \draw[black, thin] 
            ([shift={(-5pt,-5pt)}]current bounding box.south west) 
            rectangle 
            ([shift={(5pt,5pt)}]current bounding box.north east);                  
            \node[above] at ([shift={(0,5pt)}]current bounding box.north) {$(\mathcal{E}_1,\mathcal{E}_2)=(X_4,X_4)$};
    }
]
 \& \& \& \& X_6 \arrow[loop above, "{\{R_{11}, R_{12}\}}"] \arrow[rrdd,"{\{R_{11}, R_{12}\}}"] \arrow[lllldd, bend right=30, "R_{12}"'] \arrow[dd, "R_{11}"]\& \& \& \& \\ \\
 X_1 \arrow[rr,yshift=0.5ex,"R_1"] \& \& X_3 \arrow[ll,yshift=-0.5ex,"R_2"] \arrow[rr,"R_3"]  \& \& X_2\arrow[rr,yshift=0.5ex,"R_4"] \& \& X_4 \arrow[ll,yshift=-0.5ex,"R_5"] \arrow[lldd,xshift=-0.3ex,yshift=0.3ex,"R_7"'] \arrow[rr,"R_6"] \& \& 0 \\ \\
 \& \& \& \& X_5 \arrow[loop below, "R_9"] \arrow[rruu,xshift=0.3ex,yshift=-0.3ex, "{\{R_{8}, R_9\}}"'] \arrow[uu, "R_9"']\& \& \& \&
\end{tikzcd}
\end{equation}

For all four mechanisms, we have that the autocatalytic reaction set is $\mathcal{R}_A = \{ R_9, R_{11}, R_{12}\}$. It remains to determine the critical autocatalytic reaction set $\mathcal{R}_C$ and the critical strongly connected component set $\mathcal{Z}_C$. We note that the strongly connected components of the ALT-graph varies by mechanisms, with $(\mathcal{E}_1,\mathcal{E}_2)=(X_3,X_3)$ and $(\mathcal{E}_1,\mathcal{E}_2)=(X_3,X_4)$ admitting the strongly connected components $\mathcal{Z}_1 = \{ X_1, X_3, X_5\}$,  $\mathcal{Z}_2 = \{ X_2, X_4\}$, $\mathcal{Z}_3 = \{ X_6 \}$, and $\mathcal{Z}_4 = \{ 0 \}$, while the mechanisms with $(\mathcal{E}_1,\mathcal{E}_2)=(X_4,X_3)$ and $(\mathcal{E}_1,\mathcal{E}_2)=(X_4,X_4)$ admit the strongly connected components $\mathcal{Z}_1 = \{ X_1, X_3\}$,  $\mathcal{Z}_2 = \{ X_2, X_4, X_5\}$, $\mathcal{Z}_3 = \{ X_6 \}$, and $\mathcal{Z}_4 = \{ 0 \}$. We have the following analysis.

\paragraph{Critical autocatalytic reactions $(\mathcal{R}_C)$:} In the mechanisms with $(\mathcal{E}_1,\mathcal{E}_2)=(X_3,X_3)$ and $(\mathcal{E}_1,\mathcal{E}_2)=(X_3,X_4)$, the reaction $R_9 \in \mathcal{R}_A$ has exactly one catalytic instance leaving the strongly connected component $\mathcal{Z}_1$. Consequently, the reaction is not critical in these mechanism. In the mechanisms with $(\mathcal{E}_1,\mathcal{E}_2)=(X_4,X_3)$ and $(\mathcal{E}_1,\mathcal{E}_2)=(X_4,X_4)$, however, $R_9 \in \mathcal{R}_A$ has no catalytic instances leaving the strongly connected component $\mathcal{Z}_2$. Consequently, the reaction is critical in these mechanisms.

In all four mechanisms, all catalytic instances of the autocatalytic reactions $R_{11}, R_{12} \in \mathcal{R}_A$ leave the strongly connected component $\mathcal{Z}_3$ and consequently these reactions are not critical in any mechanism.

It follows that $\mathcal{R}_C = \emptyset$ for the mechanisms with $(\mathcal{E}_1,\mathcal{E}_2)=(X_3,X_3)$ and $(\mathcal{E}_1,\mathcal{E}_2)=(X_3,X_4)$ and $\mathcal{R}_C = \{ R_9 \}$ for the mechanisms with $(\mathcal{E}_1,\mathcal{E}_2)=(X_4,X_3)$ and $(\mathcal{E}_1,\mathcal{E}_2)=(X_4,X_4)$.

\paragraph{Critical strongly connected components $(\mathcal{Z}_C)$:} In the mechanisms with $(\mathcal{E}_1,\mathcal{E}_2)=(X_3,X_3)$ and $(\mathcal{E}_1,\mathcal{E}_2)=(X_3,X_4)$, there is a catalytic instance of $R_9 \in \mathcal{R}_A$ which leaves $\mathcal{Z}_1$. However, in both mechanisms, $\mathcal{Z}_1$ may not be disconnected from $0$ by the removal of this instance because $R_3 \in \phi(e_{32})$ connects $\mathcal{Z}_1$ to $\mathcal{Z}_2$ and $R_3$ is not autocatalytic. Consequently, $\mathcal{Z}_1$ is not a critical strongly connected component in these mechanisms. Notice that no catalytic instance of $R_9 \in \mathcal{R}_A$ leaves the strongly connected component $\mathcal{Z}_2$ in the mechanisms with $(\mathcal{E}_1,\mathcal{E}_2)=(X_4,X_3)$ and $(\mathcal{E}_1,\mathcal{E}_2)=(X_4,X_4)$ so that these components do not have the capacity to be critical in these mechanisms.

Now consider the strongly connected component $\mathcal{Z}_3 = \{ X_6 \}$. In all four mechanisms, two catalytic instances of $R_{11} \in \mathcal{R}_A$ and two catalytic instances of $R_{12} \in \mathcal{R}_A$ leave $\mathcal{Z}_3$. However, removing at most one catalytic instance per reaction is insufficient to disconnect $\mathcal{Z}_3$ from $0$. Consequently, $\mathcal{Z}_3$ is not critical in any of the four mechanisms. 

Consequently, $\mathcal{Z}_C = \emptyset$ for all four mechanisms.

\paragraph{Mechanisms with $(\mathcal{E}_1,\mathcal{E}_2)=(X_3,X_3)$ and $(\mathcal{E}_1,\mathcal{E}_2)=(X_3,X_4)$:} For these mechanisms, we have $\mathcal{R}_C = \emptyset$ and $\mathcal{Z}_C = \emptyset$. It follows from Corollary \ref{cor:2} that the boundary steady state is universally stable. This reaffirms the result determined in \cite{joshi2026bistability}.

\paragraph{Mechanisms with $(\mathcal{E}_1,\mathcal{E}_2)=(X_4,X_3)$ and $(\mathcal{E}_1,\mathcal{E}_2)=(X_4,X_4)$:} For these mechanisms, we have $\mathcal{R}_C = \{ R_9 \}$ and $\mathcal{Z}_C = \emptyset$. In both mechanisms, lines 21--22 of Algorithm \ref{alg:1} dictate that we perform a catalytic splitting by removing either $R_9 \in \phi(e_{52})$ or $R_9 \in \phi(e_{54})$ and the associated autocatalytic instance $R_9 \in \phi(e_{55})$. There is no preference between the two options. Performing the catalytic splitting on $R_9 \in \phi(e_{52})$ yields the following splitting for the $(\mathcal{E}_1,\mathcal{E}_2)=(X_4,X_3)$ mechanism:
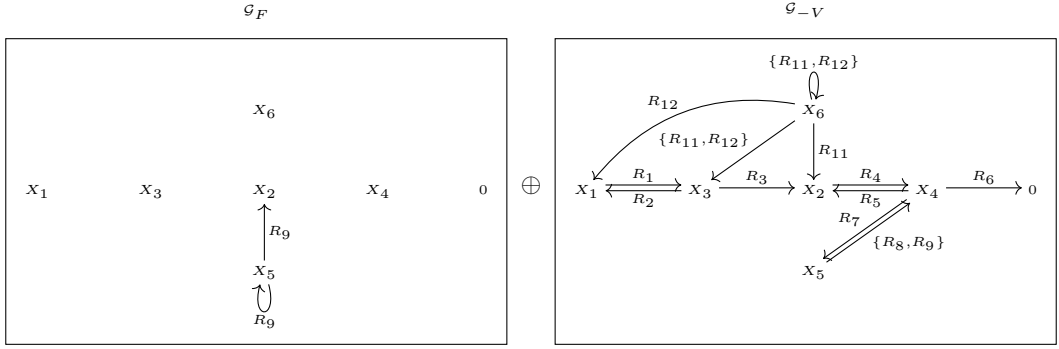
\[\tiny \begin{tikzcd}[ampersand replacement=\&,execute at end picture={
        \draw[black, thin] 
            ([shift={(-5pt,-5pt)}]current bounding box.south west) 
            rectangle 
            ([shift={(5pt,5pt)}]current bounding box.north east);                  
            \node[above] at ([shift={(0,5pt)}]current bounding box.north) {$\mathcal{G}_F$};
    }
]
 \& \& \& \& X_6 \arrow[loop above,white, "{\{R_{11}, R_{12}\}}"] \& \& \& \& \\ \\
 X_1 \& \& X_3  \& \& X_2 \& \& X_4 \& \& 0 \\ \\
 \& \& \& \& X_5 \arrow[loop below, "R_9"] \arrow[uu, "R_9"']\& \& \& \&
\end{tikzcd}\; \oplus \;
\begin{tikzcd}[ampersand replacement=\&,execute at end picture={
        \draw[black, thin] 
            ([shift={(-5pt,-5pt)}]current bounding box.south west) 
            rectangle 
            ([shift={(5pt,5pt)}]current bounding box.north east);                  
            \node[above] at ([shift={(0,5pt)}]current bounding box.north) {$\mathcal{G}_{-V}$};
    }
]
 \& \& \& \& X_6 \arrow[loop above, "{\{R_{11}, R_{12}\}}"] \arrow[lldd,"{\{R_{11}, R_{12}\}}"'] \arrow[lllldd, bend right=30, "R_{12}"'] \arrow[dd, "R_{11}"]\& \& \& \& \\ \\
 X_1 \arrow[rr,yshift=0.5ex,"R_1"] \& \& X_3 \arrow[ll,yshift=-0.5ex,"R_2"] \arrow[rr,"R_3"]  \& \& X_2\arrow[rr,yshift=0.5ex,"R_4"] \& \& X_4 \arrow[ll,yshift=-0.5ex,"R_5"] \arrow[lldd,xshift=-0.3ex,yshift=0.3ex,"R_7"'] \arrow[rr,"R_6"] \& \& 0 \\ \\
 \& \& \& \& X_5 \arrow[loop below,white,"R_9"] \arrow[rruu,xshift=0.3ex,yshift=-0.3ex,"{\{R_{8}, R_9\}}"'] \& \& \& \&
\end{tikzcd}
\]
and the following splitting for the $(\mathcal{E}_1,\mathcal{E}_2)=(X_4,X_4)$ mechanism:
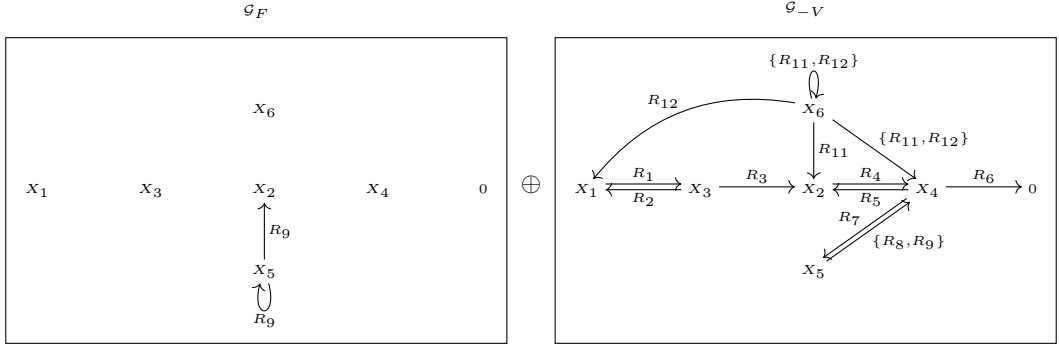
\[\tiny \begin{tikzcd}[ampersand replacement=\&,execute at end picture={
        \draw[black, thin] 
            ([shift={(-5pt,-5pt)}]current bounding box.south west) 
            rectangle 
            ([shift={(5pt,5pt)}]current bounding box.north east);                  
            \node[above] at ([shift={(0,5pt)}]current bounding box.north) {$\mathcal{G}_F$};
    }
]
 \& \& \& \& X_6 \arrow[loop above,white, "{\{R_{11}, R_{12}\}}"] \& \& \& \& \\ \\
 X_1 \& \& X_3  \& \& X_2 \& \& X_4 \& \& 0 \\ \\
 \& \& \& \& X_5 \arrow[loop below, "R_9"] \arrow[uu, "R_9"']\& \& \& \&
\end{tikzcd}\; \oplus \;
\begin{tikzcd}[ampersand replacement=\&,execute at end picture={
        \draw[black, thin] 
            ([shift={(-5pt,-5pt)}]current bounding box.south west) 
            rectangle 
            ([shift={(5pt,5pt)}]current bounding box.north east);                  
            \node[above] at ([shift={(0,5pt)}]current bounding box.north) {$\mathcal{G}_{-V}$};
    }
]
 \& \& \& \& X_6 \arrow[loop above, "{\{R_{11}, R_{12}\}}"] \arrow[rrdd,"{\{R_{11}, R_{12}\}}"] \arrow[lllldd, bend right=30, "R_{12}"'] \arrow[dd, "R_{11}"]\& \& \& \& \\ \\
 X_1 \arrow[rr,yshift=0.5ex,"R_1"] \& \& X_3 \arrow[ll,yshift=-0.5ex,"R_2"] \arrow[rr,"R_3"]  \& \& X_2\arrow[rr,yshift=0.5ex,"R_4"] \& \& X_4 \arrow[ll,yshift=-0.5ex,"R_5"] \arrow[lldd,xshift=-0.3ex,yshift=0.3ex,"R_7"'] \arrow[rr,"R_6"] \& \& 0 \\ \\
 \& \& \& \& X_5 \arrow[loop below,white,"R_9"] \arrow[rruu,xshift=0.3ex,yshift=-0.3ex,"{\{R_{8}, R_9\}}"'] \& \& \& \&
\end{tikzcd}
\]
In both cases, we have that $\mathcal{G}_F \not= \emptyset$ and that $\mathcal{G}_{-V}$ does not have any critical autocatalytic reactions or critical strongly connected components. Consequently, the conclusions of Theorem \ref{thm:1} may be applied, and the next-generation matrix may be computed using the associated ALT-matrices $F$ and $-V$ (see \cite{joshi2026bistability} for computational details in the mass-action case). A summary of the analysis is given in Table \ref{tab:mechanisms-summary}.

\begin{table}[t]
\centering
\begin{NiceTabular}{ccccccc}[code-before = {
  \rowcolor{gray!12}{2,3,6,7}}]
\toprule
\textbf{Mechanism} & $\mathcal{E}_1$ & $\mathcal{E}_2$ & \textbf{Strongly Connected Components} & $\mathcal{R}_C$ & $\mathcal{Z}_C$ & \textbf{Stability} \\
\midrule
1 & $X_3$ & $X_3$ & $\mathcal{Z}_1 = \{X_1, X_3, X_5\}$, $\mathcal{Z}_3 = \{X_6\}$ & $\emptyset$ & $\emptyset$ & Universal \\
  &  &  & $\mathcal{Z}_2 = \{X_2, X_4\}$, $\mathcal{Z}_4 = \{0\}$ &  &  &  \\
\midrule
2 & $X_3$ & $X_4$ & $\mathcal{Z}_1 = \{X_1, X_3, X_5\}$, $\mathcal{Z}_3 = \{X_6\}$ & $\emptyset$ & $\emptyset$ & Universal \\
  &  &  & $\mathcal{Z}_2 = \{X_2, X_4\}$, $\mathcal{Z}_4 = \{0\}$ &  &  &  \\
\midrule
3 & $X_4$ & $X_3$ & $\mathcal{Z}_1 = \{X_1, X_3\}$, $\mathcal{Z}_3 = \{X_6\}$ & $\{R_9\}$ & $\emptyset$ & Conditional \\
  &  &  & $\mathcal{Z}_2 = \{X_2, X_4, X_5\}$, $\mathcal{Z}_4 = \{0\}$ &  &  &  \\
\midrule
4 & $X_4$ & $X_4$ & $\mathcal{Z}_1 = \{X_1, X_3\}$, $\mathcal{Z}_3 = \{X_6\}$ & $\{R_9\}$ & $\emptyset$ & Conditional \\
  &  &  & $\mathcal{Z}_2 = \{X_2, X_4, X_5\}$, $\mathcal{Z}_4 = \{0\}$ &  &  &  \\
\bottomrule
\end{NiceTabular}
\caption{Summary of the four bifunctional enzyme mechanisms (corresponding ALT-graphs shown in equations \eqref{alt-twosite} and \eqref{alt-twosite2}). Each row specifies the mechanism parameterization ($\mathcal{E}_1, \mathcal{E}_2$), the resulting strongly connected component partition, the critical autocatalytic reaction set $\mathcal{R}_C$, the critical strongly connected component set $\mathcal{Z}_C$, and the stability classification of the boundary steady state. Mechanisms 1 and 2 (where $\mathcal{E}_1 = X_3$) exhibit universal stability ($\mathcal{R}_C = \emptyset$), while mechanisms 3 and 4 (where $\mathcal{E}_1 = X_4$) have critical autocatalytic reactions and require parameter-dependent stability conditions.}
\label{tab:mechanisms-summary}
\end{table}

\end{example}

\begin{example}
\label{ex:sir2}
Reconsider the standard SIR (Susceptible-Infectious-Removed) network \eqref{sir} and corresponding interaction network \eqref{sir-interaction} studied in Example \ref{ex:sir}. The mass-action system corresponding to \eqref{sir-interaction} coincides with the system of equations generated from \eqref{sir} via standard epidemiological assumptions. Specifically, we have:
\begin{equation}
\label{sir-de} \left\{ \; \; \;
\begin{aligned}
S' & = \Lambda - \frac{\beta}{N}SI + \alpha R - \delta S \\
I' & = \frac{\beta}{N} SI - (\gamma + \delta) I \\
R' & = \gamma I - (\alpha + \delta) R.
\end{aligned} \right.
\end{equation}

\begin{table}[t]
\centering
\begin{NiceTabular}{ccccccc}[code-before = \rowcolors{2}{white}{gray!12}]
\toprule
\begin{tabular}{c}Reaction\\ label \end{tabular} & & \begin{tabular}{c}Chemical \\ reaction \end{tabular} & & \begin{tabular}{c}$\mathcal{X}$-reduced\\reaction\end{tabular} & & \begin{tabular}{c} ALT-graph\\edge(s)\end{tabular} \\
\midrule
$R_1:$ & & \begin{tikzcd}[ampersand replacement=\&, baseline=(current bounding box.center)] 0 \arrow[r] \& Y_1 \end{tikzcd} & $\Longrightarrow$ & \begin{tikzcd}[ampersand replacement=\&, baseline=(current bounding box.center)] 0 \arrow[r] \& 0 \end{tikzcd} & $\Longrightarrow$ & None \\[0pt]
$R_2:$ & & \begin{tikzcd}[ampersand replacement=\&, baseline=(current bounding box.center)] Y_1+X_1 \arrow[r] \& 2X_1 \end{tikzcd} & $\Longrightarrow$ & \begin{tikzcd}[ampersand replacement=\&, baseline=(current bounding box.center)] X_1 \arrow[r] \& 2X_1 \end{tikzcd} & $\Longrightarrow$ & $\begin{tikzcd}[ampersand replacement=\&, baseline=(current bounding box.center)] X_1 \arrow[loop left,"R_2"] \end{tikzcd}$  \\[0pt]
$R_3:$ & & \begin{tikzcd}[ampersand replacement=\&, baseline=(current bounding box.center)] X_1 \arrow[r] \& X_2 \end{tikzcd} & $\Longrightarrow$ & \begin{tikzcd}[ampersand replacement=\&, baseline=(current bounding box.center)] X_1 \arrow[r] \& X_2 \end{tikzcd} & $\Longrightarrow$ & \begin{tikzcd}[ampersand replacement=\&, baseline=(current bounding box.center)] X_1 \arrow[r,"{}_{R_3}"] \& X_2 \end{tikzcd} \\[0pt]
$R_4:$ & & \begin{tikzcd}[ampersand replacement=\&, baseline=(current bounding box.center)]  X_2 \arrow[r] \& Y_1 \end{tikzcd} & $\Longrightarrow$ & \begin{tikzcd}[ampersand replacement=\&, baseline=(current bounding box.center)] X_2 \arrow[r] \& 0 \end{tikzcd} & $\Longrightarrow$ & \begin{tikzcd}[ampersand replacement=\&, baseline=(current bounding box.center)] X_2 \arrow[r,"{}_{R_4}"] \& 0 \end{tikzcd} \\[0pt]
$R_5:$ & & \begin{tikzcd}[ampersand replacement=\&, baseline=(current bounding box.center)]  Y_1 \arrow[r] \& 0 \end{tikzcd} & $\Longrightarrow$ & \begin{tikzcd}[ampersand replacement=\&, baseline=(current bounding box.center)] 0 \arrow[r] \& 0 \end{tikzcd} & $\Longrightarrow$ & None \\[0pt]
$R_6:$ & & \begin{tikzcd}[ampersand replacement=\&, baseline=(current bounding box.center)]  X_1 \arrow[r] \& 0 \end{tikzcd} & $\Longrightarrow$ & \begin{tikzcd}[ampersand replacement=\&, baseline=(current bounding box.center)] X_1 \arrow[r] \& 0 \end{tikzcd} & $\Longrightarrow$ & \begin{tikzcd}[ampersand replacement=\&, baseline=(current bounding box.center)] X_1 \arrow[r,"{}_{R_6}"] \& 0 \end{tikzcd} \\[0pt]
$R_7:$ & & \begin{tikzcd}[ampersand replacement=\&, baseline=(current bounding box.center)]  X_2 \arrow[r] \& 0 \end{tikzcd} & $\Longrightarrow$ & \begin{tikzcd}[ampersand replacement=\&, baseline=(current bounding box.center)] X_2 \arrow[r] \& 0 \end{tikzcd} & $\Longrightarrow$ & \begin{tikzcd}[ampersand replacement=\&, baseline=(current bounding box.center)] X_2 \arrow[r,"{}_{R_7}"] \& 0 \end{tikzcd} \\[0pt]
\bottomrule
\end{NiceTabular}
\caption{Reaction-by-reaction construction of the ALT-graph for the SIR model \eqref{sir}. The siphon $\mathcal{X}=\{I,R\}$ is first identified (Definition \ref{def:siphon}), and the species are reindexed as $X_1=I$, $X_2=R$, and $Y_1=S$ as described in Section \ref{sec:siphon}. The reactions are then reduced by removing the anti-siphon species $Y_1$ (Definition \ref{def:reduced-graph}), after which the resulting reduced reactions are mapped to ALT-graph edges according to Definition \ref{def:alt-graph} and Table \ref{tab:jacobian}. These edges are finally assembled to obtain the ALT-graph on the vertex set $\mathcal{X}=\{X_1,X_2\}\cup \{ 0 \}$, shown in \eqref{sir-alt}. Note that we exclude reactions which reduced to $0 \to 0$ (i.e., $R_1$ and $R_5$) by the requirement that $\alpha'_{ik} \not= \beta'_{ik}$ for at least one $i \in \{ 1, \ldots, m \}$ of Definition \ref{def:reduced-graph}. Note also that we have merged $R_4$ and $R_7$ as the multiset $\{ R_4, R_7 \}$ for the label of $e_{20}$, consistent with the minimal multiset requirement of Definition \ref{def:alt-graph}.}
\label{tab:jacobian4}
\end{table}

The interaction network \eqref{sir-interaction} has the siphon $\mathcal{X} = \{ I, R \}$ corresponding to the boundary steady state $(S,I,R) = (N,0,0)$. We reindex $X_1 = I$, $X_2 = R$, and $Y_1 = S$, and label the reactions $R_1$ through $R_7$. This results in the following reindexed interaction network:
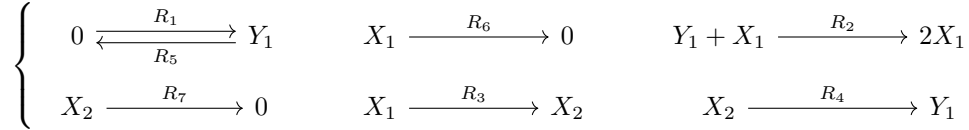
\[
\left\{ \; \;
\begin{tikzcd}
0 \arrow[rr,yshift=0.5ex,"R_1"] & & Y_1 \arrow[ll,yshift=-0.5ex,"R_5"] & X_1 \arrow[rr,"R_6"] & & 0 & Y_1 + X_1 \arrow[rr,"R_2"] & & 2X_1 \\[-0.1in]
X_2 \arrow[rr,"R_7"] & & 0 & X_1 \arrow[rr,"R_3"] & & X_2 & X_2 \arrow[rr,"R_4"] & & Y_1 
\end{tikzcd} \right.
\]

The correspondence of each reaction to its $\mathcal{X}$-reduced reaction and associated ALT-graph edges is given in Table \ref{tab:jacobian4}. The corresponding ALT-graph is as follows:
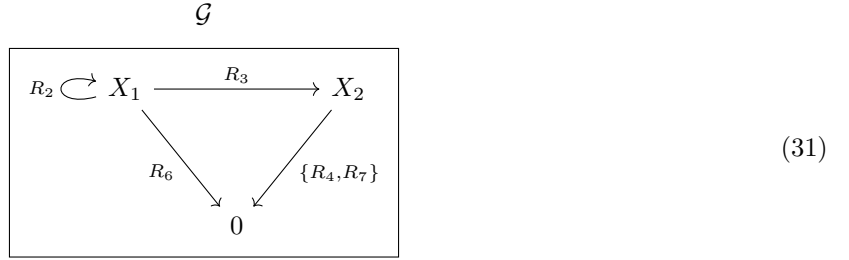
\begin{equation}
\label{sir-alt}
 \begin{tikzcd}[ampersand replacement=\&,execute at end picture={
         \draw[black, thin] 
             ([shift={(-5pt,-5pt)}]current bounding box.south west) 
             rectangle 
             ([shift={(5pt,5pt)}]current bounding box.north east);                  
             \node[above] at ([shift={(0,5pt)}]current bounding box.north) {$\mathcal{G}$};
     }
 ]
 X_1 \arrow[loop left,"R_2"] \arrow[rr,"R_3"] \arrow[rdd,"R_6"'] \& \& X_2 \arrow[ldd,"{\{R_4,R_7\}}"] \\
 \& \& \\
 \& 0 \&
 \end{tikzcd}
 \end{equation}
We have the autocatalytic reaction set $\mathcal{R}_A = \{R_2\}$ and note that $R_2$ corresponds to new infections. The strongly connected components are $\mathcal{Z}_1 = \{ X_1 \}$, $\mathcal{Z}_2 = \{ X_2 \}$, and $\mathcal{Z}_3 = \{ 0 \}$. Since there are no catalytic instances of $R_2$, it satisfies the conditions for criticality trivially so that $\mathcal{R}_C = \{ R_2 \}$. Furthermore, since there are no catalytic instances of the only autocatalytic reaction $R_2$ and there is a path from $\mathcal{Z}_2$ to $0$ through $R_6$, the connectivity of $\mathcal{Z}_2$ and $0$ cannot be altered by removal of catalytic instances. It follows that there are no critical strongly connected components so that $\mathcal{Z}_C = \emptyset$. By lines 20--21 and 24 of Algorithm \ref{alg:1}, we remove the instance $R_2 \in \phi(e_{11})$ from $\mathcal{G}$ and add it to $\mathcal{G}_F$. This gives the following autocatalytic splitting:
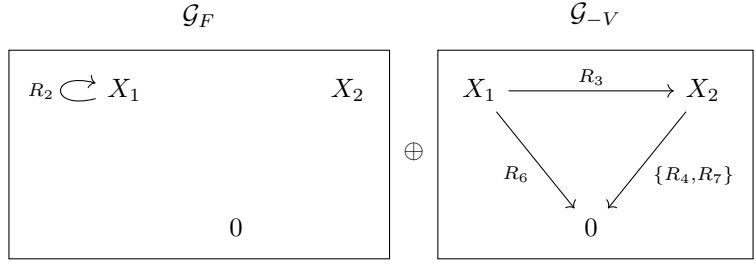
\[
\begin{tikzcd}[ampersand replacement=\&,execute at end picture={
        \draw[black, thin] 
            ([shift={(-5pt,-5pt)}]current bounding box.south west) 
            rectangle 
            ([shift={(5pt,5pt)}]current bounding box.north east);                  
            \node[above] at ([shift={(0,5pt)}]current bounding box.north) {$\mathcal{G}_F$};
    }
]
X_1 \arrow[loop left,"R_2"] \arrow[rr,white,"R_3"] \arrow[rdd,white,"R_6"']\& \& X_2  \\
\& \& \\
\& 0 \&
\end{tikzcd} \; \oplus \;
\begin{tikzcd}[ampersand replacement=\&,execute at end picture={
        \draw[black, thin] 
            ([shift={(-5pt,-5pt)}]current bounding box.south west) 
            rectangle 
            ([shift={(5pt,5pt)}]current bounding box.north east);                  
            \node[above] at ([shift={(0,5pt)}]current bounding box.north) {$\mathcal{G}_{-V}$};
    }
]
X_1 \arrow[rr,"R_3"] \arrow[rdd,"R_6"'] \& \& X_2 \arrow[ldd,"{\{R_4,R_7\}}"] \\
\& \& \\
\& 0 \&
\end{tikzcd}
\]
There are no autocatalytic reactions in $\mathcal{G}_{-V}$ so Algorithm \ref{alg:1} terminates. Theorem \ref{thm:1} guarantees the corresponding Jacobian splitting $J=F-V$ leads to the valid next-generation matrix $FV^{-1}$ from which the basic reproduction number $\mathscr{R}_0 = \rho(FV^{-1})$ may be computed. For general kinetic functions $R_1$ through $R_7$ satisfying \textbf{(K1)}--\textbf{(K5)}, we have that
\[
F = \begin{bmatrix}
\partial_1 R_2 & 0 \\
0 & 0
\end{bmatrix} \mbox{ and }
V = \begin{bmatrix}
\partial_1 (R_3+R_6) & 0 \\
-\partial_1 R_3 & \partial_2 (R_4+R_7) 
\end{bmatrix}
\]
so that
\[
V^{-1}
=
\begin{bmatrix}
\dfrac{1}{\partial_1(R_3+R_6)} & 0\\[8pt]
\dfrac{\partial_1R_3}
{\partial_1(R_3+R_6)\,\partial_2(R_4+R_7)}
&
\dfrac{1}{\partial_2(R_4+R_7)}
\end{bmatrix} \mbox{ and }
FV^{-1}
=
\begin{bmatrix}
\dfrac{\partial_1 R_2}{\partial_1(R_3+R_6)} & 0\\[8pt]
0 & 0
\end{bmatrix}
\]
where $\partial_i R_k:= \frac{\partial R_k(x^*,y^*)}{\partial x_i}$ from \ref{def:alt-matrix}. It follows that the basic reproduction number is $\displaystyle \mathscr{R}_0 = \rho(FV^{-1}) = \dfrac{\partial_1 R_2}{\partial_1(R_3+R_6)}$.

For the specific kinetic functions in \eqref{sir-de}, we have that the respective partial derivatives evaluated at the boundary steady state are $\partial_1 R_2 = \frac{\partial}{\partial I} \left. \left( \frac{\beta}{N} SI \right) \right|_{(S,I,R)=(N,0,0)} = \frac{\beta}{N} N = \beta$, $\partial_1 R_3 = \gamma$, and $\partial_1 R_6 = \delta$ so that the basic reproduction number is:
\[\mathscr{R}_0 = \rho(FV^{-1}) = \dfrac{\beta}{\gamma + \delta}.\]
\end{example}

\begin{example}
\label{ex:vectorborne2}

Reconsider the vector-borne disease spread model \eqref{sir-vector} and corresponding interaction network \eqref{vector} studied in Example \ref{ex:vectorborne}. The mass-action system corresponding to \eqref{vector} coincides with the system of equations generated from \eqref{sir-vector} via standard epidemiological assumptions. Specifically, we have:
\begin{equation}
\label{sir-vector-des}
\left\{ \; \;
\begin{aligned}
S' & = \alpha_s-\delta_s S+\gamma_s I - \beta_s S V \\
I' & = \beta_s S V - (\delta_s + \gamma_s)I\\
M' & = \alpha_m-\delta_m M - \beta_m M I \\
V' & = \beta_m M I - \delta_m V.
\end{aligned}
\right.
\end{equation}

The interaction network \eqref{vector} has the siphon $\mathcal{X} = \{ I, V \}$ corresponding to the boundary steady state $(S,I,M,V) = (\alpha_s/\delta_s,0,\alpha_m/\delta_m,0)$. We reindex $X_1 = I$, $X_2 = V$, $Y_1 = S$, and $Y_2= M$ and label the reactions $R_1$ through $R_9$. This results in the following reindexed interaction network:
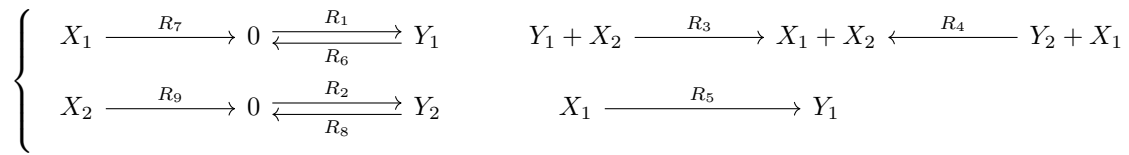
\[
\left\{ \; \;
\begin{tikzcd}
X_1 \arrow[rr,"R_7"] & & 0 \arrow[rr,yshift=0.5ex,"R_1"] & & Y_1 \arrow[ll,yshift=-0.5ex,"R_6"] & Y_1 + X_2 \arrow[rr,"R_3"] & & X_1 + X_2 & & Y_2 + X_1 \arrow[ll,"R_4"'] \\[-0.1in]
X_2 \arrow[rr,"R_9"] & & 0 \arrow[rr,yshift=0.5ex,"R_2"] & & Y_2 \arrow[ll,yshift=-0.5ex,"R_8"] & X_1 \arrow[rr,"R_5"] & & Y_1 & & \\[-0.1in]
\end{tikzcd}
\right.
\]

\begin{table}[t]
\centering
\begin{NiceTabular}{ccccccc}[code-before = \rowcolors{2}{white}{gray!12}]
\toprule
\begin{tabular}{c}Reaction\\ label \end{tabular} & & \begin{tabular}{c}Chemical \\ reaction \end{tabular} & & \begin{tabular}{c}$\mathcal{X}$-reduced\\reaction\end{tabular} & & \begin{tabular}{c} ALT-graph\\edge(s)\end{tabular} \\
\midrule
$R_1:$ & & \begin{tikzcd}[ampersand replacement=\&, baseline=(current bounding box.center)] 0 \arrow[r] \& Y_1 \end{tikzcd} & $\Longrightarrow$ & \begin{tikzcd}[ampersand replacement=\&, baseline=(current bounding box.center)] 0 \arrow[r] \& 0 \end{tikzcd} & $\Longrightarrow$ & None \\[0pt]
$R_2:$ & & \begin{tikzcd}[ampersand replacement=\&, baseline=(current bounding box.center)] 0 \arrow[r] \& Y_2 \end{tikzcd} & $\Longrightarrow$ & \begin{tikzcd}[ampersand replacement=\&, baseline=(current bounding box.center)] 0 \arrow[r] \& 0 \end{tikzcd} & $\Longrightarrow$ & None \\[0pt]
$R_3:$ & & \begin{tikzcd}[ampersand replacement=\&, baseline=(current bounding box.center)] Y_1+X_2 \arrow[r] \& X_1+X_2 \end{tikzcd} & $\Longrightarrow$ & \begin{tikzcd}[ampersand replacement=\&, baseline=(current bounding box.center)] X_2 \arrow[r] \& X_1+X_2 \end{tikzcd} & $\Longrightarrow$ & $\begin{tikzcd}[ampersand replacement=\&, baseline=(current bounding box.center)] X_2 \arrow[loop left,"R_3"] \arrow[r,"R_3"] \& X_1 \end{tikzcd}$  \\[0pt]
$R_4:$ & & \begin{tikzcd}[ampersand replacement=\&, baseline=(current bounding box.center)] Y_2+X_1 \arrow[r] \& X_1+X_2 \end{tikzcd} & $\Longrightarrow$ & \begin{tikzcd}[ampersand replacement=\&, baseline=(current bounding box.center)] X_1 \arrow[r] \& X_1+X_2 \end{tikzcd} & $\Longrightarrow$ & $\begin{tikzcd}[ampersand replacement=\&, baseline=(current bounding box.center)] X_1 \arrow[loop left,"R_4"] \arrow[r,"R_4"] \& X_2 \end{tikzcd}$  \\[0pt]
$R_5:$ & & \begin{tikzcd}[ampersand replacement=\&, baseline=(current bounding box.center)] X_1 \arrow[r] \& Y_1 \end{tikzcd} & $\Longrightarrow$ & \begin{tikzcd}[ampersand replacement=\&, baseline=(current bounding box.center)] X_1 \arrow[r] \& 0 \end{tikzcd} & $\Longrightarrow$ & \begin{tikzcd}[ampersand replacement=\&, baseline=(current bounding box.center)] X_1 \arrow[r,"{}_{R_5}"] \& 0 \end{tikzcd} \\[0pt]
$R_6:$ & & \begin{tikzcd}[ampersand replacement=\&, baseline=(current bounding box.center)]  Y_1 \arrow[r] \& 0 \end{tikzcd} & $\Longrightarrow$ & \begin{tikzcd}[ampersand replacement=\&, baseline=(current bounding box.center)] 0 \arrow[r] \& 0 \end{tikzcd} & $\Longrightarrow$ & None \\[0pt]
$R_7:$ & & \begin{tikzcd}[ampersand replacement=\&, baseline=(current bounding box.center)]  X_1 \arrow[r] \& 0 \end{tikzcd} & $\Longrightarrow$ & \begin{tikzcd}[ampersand replacement=\&, baseline=(current bounding box.center)] X_1 \arrow[r] \& 0 \end{tikzcd} & $\Longrightarrow$ & \begin{tikzcd}[ampersand replacement=\&, baseline=(current bounding box.center)] X_1 \arrow[r,"{}_{R_7}"] \& 0 \end{tikzcd} \\[0pt]
$R_8:$ & & \begin{tikzcd}[ampersand replacement=\&, baseline=(current bounding box.center)]  Y_2 \arrow[r] \& 0 \end{tikzcd} & $\Longrightarrow$ & \begin{tikzcd}[ampersand replacement=\&, baseline=(current bounding box.center)] 0 \arrow[r] \& 0 \end{tikzcd} & $\Longrightarrow$ & None \\[0pt]
$R_9:$ & & \begin{tikzcd}[ampersand replacement=\&, baseline=(current bounding box.center)]  X_2 \arrow[r] \& 0 \end{tikzcd} & $\Longrightarrow$ & \begin{tikzcd}[ampersand replacement=\&, baseline=(current bounding box.center)] X_2 \arrow[r] \& 0 \end{tikzcd} & $\Longrightarrow$ & \begin{tikzcd}[ampersand replacement=\&, baseline=(current bounding box.center)] X_2 \arrow[r,"{}_{R_9}"] \& 0 \end{tikzcd} \\[0pt]
\bottomrule
\end{NiceTabular}
\caption{Reaction-by-reaction construction of the ALT-graph for the vector-borne disease spread model \eqref{sir-vector}. The siphon $\mathcal{X}=\{I,V\}$ is first identified (Definition \ref{def:siphon}), and the species are reindexed as $X_1 = I$, $X_2 = V$, $Y_1 = S$, and $Y_2= M$ as described in Section \ref{sec:siphon}. The reactions are then reduced by removing the antisiphon species $Y_1$ and $Y_2$ (Definition \ref{def:reduced-graph}), after which the resulting reduced reactions are mapped to ALT-graph edges according to Definition \ref{def:alt-graph} and Table \ref{tab:jacobian}. These edges are finally assembled to obtain the ALT-graph on the vertex set $\mathcal{X}=\{X_1,X_2\}\cup \{ 0 \}$, shown in \eqref{vector-alt2}. Note that we exclude reactions which reduced to $0 \to 0$ (i.e., $R_1$, $R_2$, $R_6$, and $R_8$) by the requirement that $\alpha'_{ik} \not= \beta'_{ik}$ for at least one $i \in \{ 1, \ldots, m \}$ of Definition \ref{def:reduced-graph}. Note also that we have merged $R_5$ and $R_7$ as the multiset $\{ R_5, R_7 \}$ for the label of $e_{10}$, consistent with the minimal multiset requirement of Definition \ref{def:alt-graph}.}
\label{tab:jacobian5}
\end{table}

The correspondence of each reaction to its $\mathcal{X}$-reduced reaction and associated ALT-graph edges is given in Table \ref{tab:jacobian5}. The corresponding ALT-graph is as follows:
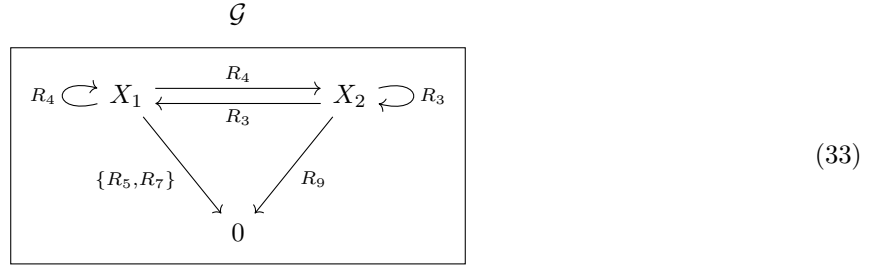
\begin{equation}
\label{vector-alt2}
\begin{tikzcd}[ampersand replacement=\&,execute at end picture={
        \draw[black, thin] 
            ([shift={(-5pt,-5pt)}]current bounding box.south west) 
            rectangle 
            ([shift={(5pt,5pt)}]current bounding box.north east);                  
            \node[above] at ([shift={(0,5pt)}]current bounding box.north) {$\mathcal{G}$};
    }
]
X_1 \arrow[loop left,"R_4"] \arrow[rr,yshift=0.1cm,"R_4"] \arrow[rdd,"{\{R_5,R_7\}}"'] \& \& X_2 \arrow[loop right,"R_3"] \arrow[ll,yshift=-0.1cm,"R_3"] \arrow[ldd,"R_9"] \\
\& \& \\
\& 0 \&
\end{tikzcd}
\end{equation}
We have the autocatalytic reaction set $\mathcal{R}_A = \{ R_3, R_4 \}$, where we note that each reaction corresponds to a cross infection event---specifically, mosquitoes infecting humans ($R_3$) and humans infecting mosquitoes ($R_4$). The strongly connected components are $\mathcal{Z}_1 = \{ X_1, X_2 \}$ and $\mathcal{Z}_2 = \{ 0 \}$. For $R_3 \in \mathcal{R}_C$, we have an autocatalytic instance $R_3 \in \phi(e_{22})$ and a catalytic instance $R_3 \in \phi(e_{21})$ and, for $R_4 \in \mathcal{R}_C$, we have an autocatalytic instance $R_4 \in \phi(e_{11})$ and a catalytic instance $R_4 \in \phi(e_{12})$. Since all catalytic instances connect vertices in $\mathcal{Z}_1$, we have that both autocatalytic reactions are critical by Definition \ref{def:autocatalytic} so that $\mathcal{R}_C = \{ R_3, R_4 \}$. Furthermore, since the instances $R_5 \in \phi(e_{10})$, $R_7 \in \phi(e_{10})$, and $R_9 \in \phi(e_{20})$ connect $\mathcal{Z}_1$ to $0$, no removal of catalytic instances of $R_3 \in \mathcal{R}_A$ or $R_4 \in \mathcal{R}_A$ can disconnect $\mathcal{Z}_1$ from $0$. Consequently, there are no critical strongly connected components by Definition \ref{def:autocatalytic} so that $\mathcal{Z}_C = \emptyset$.

Since $X_1$ is target of two instances of an autocatalytic reaction ($R_3 \in \phi(e_{21})$ and $R_4 \in \phi(e_{11})$) and $X_2$ is the target of two instances of an autocatalytic reaction ($R_3 \in \phi(e_{22})$ and $R_4 \in \phi(e_{12})$), we may apply lines 16--17 of Algorithm \eqref{alg:1} to either $R_3 \in \phi(e_{21})$ or $R_4 \in \phi(e_{12})$. Performing a catalytic splitting on the instance $R_4 \in \phi(e_{12})$ gives the following autocatalytic decomposition $\mathcal{G} = \mathcal{G}_{F} \oplus \mathcal{G}_{-V}$:
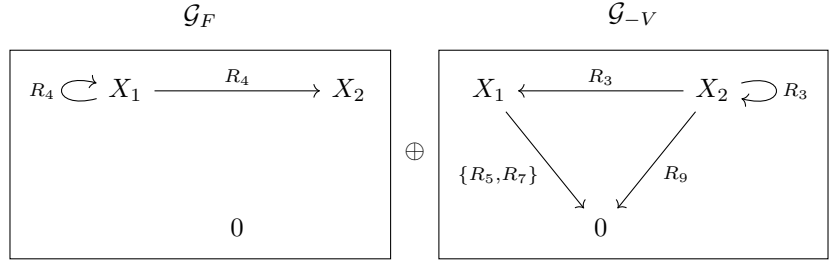
\[
\begin{tikzcd}[ampersand replacement=\&,execute at end picture={
        \draw[black, thin] 
            ([shift={(-5pt,-5pt)}]current bounding box.south west) 
            rectangle 
            ([shift={(5pt,5pt)}]current bounding box.north east);                  
            \node[above] at ([shift={(0,5pt)}]current bounding box.north) {$\mathcal{G}_F$};
    }
]
X_1 \arrow[loop left,"R_4"] \arrow[rr,"R_4"]  \& \& X_2  \\
\& \& \\
\& 0 \&
\end{tikzcd} \; \oplus \; \begin{tikzcd}[ampersand replacement=\&,execute at end picture={
        \draw[black, thin] 
            ([shift={(-5pt,-5pt)}]current bounding box.south west) 
            rectangle 
            ([shift={(5pt,5pt)}]current bounding box.north east);                  
            \node[above] at ([shift={(0,5pt)}]current bounding box.north) {$\mathcal{G}_{-V}$};
    }
]
X_1  \arrow[rdd,"{\{R_5,R_7\}}"'] \& \& X_2 \arrow[loop right,"R_3"] \arrow[ll,"R_3"'] \arrow[ldd,"R_9"] \\
\& \& \\
\& 0 \&
\end{tikzcd}
\]
We now recompute $\mathcal{R}_A$, $\mathcal{R}_C$, and $\mathcal{Z}_C$ for $\mathcal{G}_{-V}$. We have a single autocatalytic reaction, $\mathcal{R}_A = \{ R_3 \}$; however, the strongly connected components have changed from $\mathcal{G}$ to $\mathcal{G}_{-V}$ and we have that $\mathcal{Z}_1 = \{X_1\}$, $\mathcal{Z}_2 = \{ X_2 \}$, and $\mathcal{Z}_3 = \{ 0 \}$. Since $R_3 \in \mathcal{R}_A$ has a catalytic instance $R_3 \in \phi_{-V}(e_{21})$ which goes between strongly connected components, we have that $R_3$ is not critical, i.e., $\mathcal{R}_C = \emptyset$. Furthermore, when we remove the lone catalytic instance $R_3 \in \phi(e_{21})$ there is still a path from $\mathcal{Z}_2$ to $0$ through $R_9 \in \phi(e_{20})$. Consequently, by Definition \ref{def:autocatalytic}, there are no critical strongly connected components, i.e., $\mathcal{Z}_C = \emptyset$. Consequently, Algorithm \ref{alg:1} terminates.

By Theorem \ref{thm:1}, it follows that the ALT-matrices $F$ and $-V$ generate a valid next-generation matrix $FV^{-1}$. For general kinetic functions $R_1$ through $R_9$ satisfying \textbf{(K1)}--\textbf{(K5)}, we have that
\[
F = \begin{bmatrix}
0 & 0 \\
\partial_1 R_4 & 0
\end{bmatrix} \mbox{ and }
V = \begin{bmatrix}
\partial_1 (R_5+R_7) & -\partial_2 R_3 \\
0 & \partial_2 R_9 
\end{bmatrix}
\]
so that
\[
V^{-1} =
\begin{bmatrix}
\displaystyle \dfrac{1}{\partial_1 (R_5+R_7)} &
\displaystyle \dfrac{\partial_2 R_3}{\partial_1 (R_5+R_7)\partial_2 R_9} \\
0 &
\displaystyle \dfrac{1}{\partial_2 R_9}
\end{bmatrix} \mbox{ and }
FV^{-1} =
\begin{bmatrix}
0 & 0 \\
\displaystyle \dfrac{\partial_1 R_4}{\partial_1 (R_5+R_7)} &
\displaystyle \dfrac{\partial_2 R_3 \partial_1 R_4}
{\partial_1 (R_5+R_7)\partial_2 R_9}
\end{bmatrix}
\]
It follows that the basic reproduction number is $\displaystyle \mathscr{R}_0 = \rho(FV^{-1}) = \dfrac{\partial_2 R_3\partial_1 R_4}
{\partial_1 (R_5+R_7)\partial_2 R_9}$. 

For the specific kinetic functions in \eqref{sir-vector-des}, we have that the respective partial derivatives evaluated at the boundary steady state are $\partial_2 R_3 = \beta_s\alpha_s/\delta_s$, $\partial_1 R_4 = \beta_m\alpha_m/\delta_m$, $\partial_1 R_5 = \gamma_s$, $\partial_1 R_7 = \delta_s$, and $\partial_2 R_9 = \delta_m$ so that the basic reproduction number is:
\[\displaystyle \mathscr{R}_0 = \rho(FV^{-1}) = \dfrac{\beta_s \beta_m \alpha_s \alpha_m}{(\gamma_s + \delta_s)\delta_s \delta_m^2}.\]

\end{example}

\end{document}